\documentclass[a4paper,10pt]{article}
\usepackage[DIV=11]{typearea}
\usepackage{amsmath,amssymb,amsthm}
\usepackage{setspace}
\usepackage[dvipsnames]{xcolor}
\usepackage[unicode, final, colorlinks=true]{hyperref}
\usepackage{enumitem, booktabs}
\usepackage{xcolor}
\usepackage{siunitx}
\usepackage{bbm}
\usepackage{mathtools}
\usepackage{tikz, pgfplots}
\usepackage{lineno}
\usepackage{graphicx}
\usepackage{relsize}
\usepackage{mleftright}
\usepackage{cite}
\usepackage{mathrsfs}
\usepackage[normalem]{ulem}
\usepackage{subcaption}
\usepackage{cancel}

\usepackage[a4paper,margin=2.5cm]{geometry}

\hypersetup{citecolor=teal, linkcolor=BrickRed, urlcolor=NavyBlue}

\pgfplotsset{compat=1.17}
\usepgfplotslibrary{groupplots}
\usepgfplotslibrary{fillbetween}
\usetikzlibrary{backgrounds}
\usetikzlibrary{arrows.meta}
\pgfplotsset{plot coordinates/math parser=false}
\usetikzlibrary{external}
\usepgfplotslibrary{patchplots}
\usetikzlibrary{shapes.geometric}
\usetikzlibrary{backgrounds}
\usetikzlibrary{intersections}
\usetikzlibrary{spy}

\mathtoolsset{showonlyrefs}

\newcommand{\semibf}[1]{%
  \ensuremath{%
    \text{%
      \kern0pt\rlap{$#1$}\kern0.03em$#1$%
    }%
  }%
}

\newcommand{\N}{\mathbb{N}}
\newcommand{\R}{\mathbb{R}}
\newcommand{\Z}{\mathbb{Z}}

\newcommand{\defeq}{\coloneqq}
\newcommand{\eqdef}{\eqqcolon}
\newcommand{\eps}{\varepsilon}
\newcommand{\ddt}{\partial_t}
\newcommand{\ddx}{\partial_{x_1}}
\newcommand{\ddxx}{\partial_{x_1 x_1}^2}
\newcommand{\ddy}{\partial_{x_2}}
\newcommand{\dx}{\Delta x_1}
\newcommand{\dy}{\Delta x_2}
\newcommand{\dt}{\Delta t}

\newcommand{\xb}{\textbf{x}}
\newcommand{\Div}{\textnormal{div}_\xb \,}
\newcommand{\rhob}{\boldsymbol{\rho}}
\newcommand{\varrhob}{\boldsymbol{\varrho}}
\newcommand{\etab}{\boldsymbol{\eta}}
\newcommand{\nub}{\boldsymbol{\nu}}
\newcommand{\fb}{\boldsymbol{f}}
\newcommand{\ub}{\boldsymbol{u}}
\newcommand{\yb}{\textbf{y}}
\newcommand{\Rb}{\semibf{R}}

\newcommand{\imgrho}{\mathcal{I}}
\newcommand{\norm}[1]{\lVert #1 \rVert}
\newcommand{\mydot}{\raisebox{0.3ex}{\scalebox{0.5}{$\bullet$}}}

\newcommand{\numfluxkn}[5]{%
  F_{#1,k,n}^{#2,#3}%
  \if\relax\detokenize{#4}\relax
  \else
    \mleft( #4^{k,n}, #5^{k,n} \mright)%
  \fi
}

\newcommand{\numflux}[5]{
  F_{#1,n}^{#2,#3}
  \if\relax\detokenize{#4}\relax
  \else
    \mleft( #4, #5 \mright)
  \fi
}

\newcommand{\numfluxlong}[6]{
  F_{#1,n}^{#2,#3}
  \if\relax\detokenize{#4}\relax
  \else
    \mleft( #4, #5; #6 \mright)
  \fi
}

\newcommand{\flux}[4]{f_{#1,n}^{#2,#3}(#4)}
\newcommand{\redfluxk}[4]{f_{#1,\reva{k},n}^{#2,#3}(#4)}
\newcommand{\fluxk}[4]{f^k_{#1}\mleft(t^n, \xb^{#2,#3}, #4, \Rb_{#2,#3}^n \mright)}
\newcommand{\fluxself}[4]{f_{#1}\mleft(t^n, #2, #3, #4 \mright)}
\newcommand{\entropyfluxkn}[5]{\mathcal{F}_{#1,k,n}^{#2,#3,\kappa} \mleft(#4, #5 \mright)}

\newcommand{\entropyfluxknkappa}[6]{\mathcal{F}_{#1,k,n}^{#2,#3,#6} \mleft(#4, #5 \mright)}
\newcommand{\Lipf}[2]{
  \if\relax\detokenize{#2}\relax
    \mathcal{L}_{#1}
  \else
    \mathcal{L}_{#1,#2}
  \fi
}

\newcommand{\contconv}[1]{\boldsymbol{\mathcal{{R}}}[#1]\mleft( t, \xb\mright)}
\newcommand{\contconarg}[3]{\boldsymbol{\mathcal{{R}}}[#1]\mleft( #2, #3 \mright)}

\newcommand{\bigeps}{{\scalebox{1.5}{$\varepsilon$}}}
\newcommand{\biglambda}{\Lambda}

\newcommand{\reva}[1]{{\color{black} #1}}
\newcommand{\revb}[1]{{\color{black} #1}}
\newcommand{\revc}[1]{{\color{black} #1}}
\newcommand{\revd}[1]{{\color{black} #1}}

\providecommand{\msc}[1]{\textit{2020 MSC:} #1}

\theoremstyle{plain}
\newtheorem{theorem}{Theorem}[section]
\newtheorem{proposition}[theorem]{Proposition}
\newtheorem{lemma}[theorem]{Lemma}
\newtheorem{corollary}[theorem]{Corollary}
\newtheorem{definition}[theorem]{Definition}

\newtheorem{remark}[theorem]{Remark}
\newtheorem{assumption}[theorem]{Assumption}

\begin{document}
\date{
  \small
   \today
  }

\title{Monotone-based Numerical Schemes for Two-Dimensional Systems of Nonlocal Conservation Laws}
\author{Anika Beckers\thanks{Chair of Numerical Analysis, Institute for Geometry and Applied Mathematics, RWTH Aachen University, Im Süsterfeld 2, 52072 Aachen, Germany, \tt{beckers@igpm.rwth-aachen.de}.} \and Jan Friedrich\thanks{Chair of Optimal Control, Department for Mathematics, School of Computation, Information and Technology, Technical University of Munich, Boltzmannstraße 3, 85748 Garching b. Munich, Germany, \tt{jan.friedrich@cit.tum.de}.}}

\maketitle

\begin{abstract}
 \noindent We present a \revc{general} class of numerical schemes for two-dimensional systems of nonlocal conservation laws, which are based on utilizing well-known monotone numerical flux functions after suitably approximating the nonlocal terms. 
  The considered systems are weakly coupled by the nonlocal terms and the underlying flux function is rather general to guarantee that our results are applicable to a wide range of common nonlocal models.
  We state sufficient conditions to ensure the convergence of the monotone-based numerical schemes to the unique weak entropy solution. 
  Moreover, we provide an error estimate that yields the convergence rate of $\mathcal{O}(\sqrt{\Delta t})$ for the numerical approximations of the solution.
  Our results include an existence and uniqueness proof of the nonlocal system, too.
  Numerical results illustrate our theoretical findings.
  
  \msc{35L65, 35L03, 65M12, 65M15, 76A30}
\end{abstract}

\section{Introduction}
In recent years, multidimensional nonlocal conservation laws have been used to model, e.g.\ crowd motion~\revc{\cite{ACG15,CR19,BGIV20,goatin2025pedestrians,GR24,CGL12,colombo2018nonlocal,aggarwal2016crowd}}, material flow on conveyor belts~\revc{\cite{gottlich2014modeling,RWGG20,GGZ25}}, cluster formation \revc{and cryptography~\cite{CG25,CS26} or, applications related to laser technology~\cite{colombo2015nonlocal}}.
All these models can be described by the following system of nonlocal conservation laws in two space dimensions
\begin{equation}\label{eq:system}
    \begin{cases}
        \begin{aligned}
            &\ddt \rho^k + \Div \fb^k\mleft( t, \xb, \rho^k, \etab *  \rhob \mright) = 0 \quad &&,(t, \xb) \in \R^+ \times \R^2,\; k=1,\ldots,K, \\
            &\rhob(0,\xb) = \rhob_0(\xb)&&, \xb \in  \R^2
        \end{aligned}
    \end{cases}
\end{equation}
for the state variables $\rhob = \mleft( \rho^1,\ldots,\rho^K \mright)$, considered as a function $\rhob\!:\R^+ \times \R^2 \rightarrow \R^K$, $\mleft( t,\xb \mright) \mapsto \rhob\mleft( t,\xb \mright) $
and the nonlocal flux $\fb^k\!: \R^+ \times \R^2 \times \R \times \R^M \to \R^2$, $k=1,\ldots,K$ \revb{depending on convolutions of the state variables} with a mollifier $\boldsymbol{\eta}: \R^2 \rightarrow \R^{M\times K}$.
Thus, $\etab *  \rhob \in \R^M$, where $M$ is the number of combinations of kernels and state variables that have to be convoluted. 
In particular, we define for $m=1,\ldots,M$
\begin{equation}\label{eq:conv}
  \mleft( \etab * \rhob  \mright)_m \mleft( t, \xb \mright) = \int_{\R^2} \sum_{k=1}^{K} \eta^{m,k}\mleft( \xb - \tilde\xb \mright) \rho^k\mleft( t, \tilde\xb \mright) d\tilde\xb.
\end{equation}
\revb{The system \eqref{eq:system} is coupled by the nonlocal terms \eqref{eq:conv} and}
has first been studied in \cite{ACG15}.
Here, a general result for the existence of weak entropy solutions was shown using a Lax-Friedrichs-type numerical scheme.
The uniqueness of the weak entropy solutions was proven for specific modeling equations, e.g.\ \revc{\cite{RWGG20, GR24,CGL12,colombo2018nonlocal,CG25,aggarwal2025convergence}}.
\revc{More general results for the scalar case, i.e. $K=1$, can be found in the recent work \cite{gong2026existence}. 
Additionally in the scalar case, if $\fb^1$ is} linear in $\rho$, weak solutions are already unique as proven in \cite{keimer2018multi}.
We also refer to \cite{keimer2023nonlocal} for \reva{an} overview on nonlocal balance laws with unique weak solutions. \revc{Weaker assumptions on the kernel are studied in \cite{colombo2024multidimensional}.}

\revc{In two spatial dimensions the singular limit problem, i.e. the mollifier approaching a Dirac delta,
is still an open problem. 
In case of an additional viscosity term and for a scalar equation results can be found in \cite{colombo2019singular}. 
For an overview on the singular limit in the one-dimensional setting,
we refer to \cite{keimer2023nonlocal,colombo2023overview} and the references therein. 
}

Since their introduction in \cite{ACT15,ACG15,CGL12}, nonlocal conservation laws have a deep connection to approximate numerical solutions and in particular finite volume methods.
The main reason is that they are an effective tool to prove the existence of solutions via their limit. 
A usual choice are Lax-Friedrichs-type numerical schemes, e.g. \revc{\cite{ACG15,ACT15,blandin2016well,colombo2015nonlocal,aggarwal2016crowd,betancourt2011nonlocal,chiarello2018global}}, but in recent years other schemes such as Upwind or Godunov-type schemes \revc{\cite{aggarwal2025error,friedrich2018godunov,burger2023hilliges,RWGG20,chiarello2020non,friedrich2021nonlocal,chiarello2023existence}} became popular.
For one-dimensional equations, even more general results for complete classes of numerical schemes are available.
It was shown that with an accurate approximation of the nonlocal term the commonly known monotone schemes can be applied \cite{FSS23,aggarwal2024accuracy,huang2024asymptotic}.
In particular, \cite{FSS23} follows the idea of approximating the flux at each cell interface, such that the problem can be reduced to a 'local' flux, which allows to apply monotone numerical schemes. 
Although, it should be noted that due to the nonlocality the resulting schemes are not monotone in every argument such that the usual convergence arguments cannot be directly applied.

However, for two-dimensional systems of nonlocal conservation laws, typically numerical approximations based on Lax-Friedrichs-type numerical schemes are considered \revc{\cite{ACG15,CR19,BGIV20,goatin2025pedestrians,aggarwal2016crowd,colombo2015nonlocal}}.
Here, the difficulty arises when the flux is nonlinear with respect to the local variable $\rho^k$ as in pedestrian models.
If the flux is linear, more accurate numerical methods have already been studied.
The authors of \cite{keimer2018multi,keimer2023nonlocal} consider numerical schemes based on the characteristics (in the scalar case $K=1$). 
\revc{For a nonlinear scalar flux a semi-discrete Lagrangian–Eulerian scheme is studied in \cite{abreu2025semi}.} 
In addition, Roe-type schemes have been studied for the nonlocal material flow model (including the system case) in \cite{RWGG20,GGZ25}.
To the best of the authors' knowledge a general class of numerical schemes for the equation \eqref{eq:system}, based on monotone schemes and providing sufficient conditions for convergence, has not been studied yet.

Furthermore, error estimates in one spatial dimension have been investigated very recently in a series of papers \cite{aggarwal2024accuracy,aggarwal2024well,aggarwal2025error} for different nonlocal models.
In two-dimensions, the only known result follows from \cite{aggarwal2024well} via a splitting argument. 
However, the flux considered in \cite{aggarwal2024well} does not include most of the pedestrians models, since direct space and time dependencies are not present. 

During this work, we start from a semi-discrete approach, derive a first order scheme and concentrate on appropriate numerical flux functions.
Using these flux functions together with a more accurate approximation of the semi-discrete form allows for a straightforward extension to higher order numerical schemes.
In one spatial dimension higher order schemes for nonlocal conservation laws have been studied in \cite{friedrich2019maximum,chalons2018high} and in two-dimensions, e.g.\ using finite differences, in \cite{BGIV20,goatin2025pedestrians}, \revc{or, second-order schemes in \cite{manoj2026positivity}}. 

The goal of this work is threefold:  We will provide a rigorous study of \emph{i)} the general flux \eqref{eq:system}, \emph{ii)} sufficient assumptions on numerical schemes for convergence and \emph{iii)} the error estimates for the developed schemes.
\revc{Specifically, we analyze convergence and error estimates for a broad class of monotone-based numerical flux schemes, thereby extending previous results that focused on specific fluxes. 
Analogous to the one-dimensional case—under one additional technical assumption—we prove that any monotone numerical flux, applied to an approximated flux at the cell interfaces, guarantees convergence toward the weak entropy solutions of \eqref{eq:system}. 
This provides a rigorous justification for the use of the Godunov-type and Lax-Friedrichs fluxes derived in this work for nonlocal, nonlinear fluxes. 
Such approaches are rarely employed in pedestrian modeling, yet they result in a more accurate approximation of the solutions. 
Furthermore, we establish a uniqueness result for the space-time-dependent flux in \eqref{eq:system} which, to the authors' knowledge, has not yet been addressed in this level of generality.}

The structure of this work is thereby such that after introducing the most important definitions \revb{and} assumptions in Section \ref{sec:DefAss}, we present a class of numerical schemes for the nonlocal multidimensional conservation laws \eqref{eq:system} in Section \ref{sec:scheme}.
They are based on monotone numerical fluxes for (local) conservation laws.
After discussing suitable approximations of the nonlocal term we state general assumptions on the numerical flux function to ensure convergence to a weak entropy solution.
The remaining section is devoted to the convergence proof.
In Section \ref{sec:errorestimate}, we provide an error estimate for the numerical approximation of the weak entropy solution.
Here, we first examine a Kuznetsov-type lemma that provides uniqueness of the weak entropy solution of \eqref{eq:system}, too.
Then, we establish the convergence rate of the defined class of numerical schemes.
Section \ref{sec:experiments} presents convergence studies based on numerical examples.
Thereby, we compare the accuracy of different schemes that are included in the introduced class. 
We close this work by a conclusion in Section \ref{sec:conclusion}.

\section{Definitions and assumptions}\label{sec:DefAss}
In this section, we state the most important definitions and assumptions we will need in the following to guarantee the existence of solutions to \eqref{eq:system}.
Weak entropy solutions are thereby defined similar to \cite[Def. 2.1]{CGL12}.
\begin{definition}[Weak entropy solution]\label{Def:entropysol}
    Let $T>0$ and $\rhob_0 \in L^1\mleft( \R^2;\R^K \mright)$. A \textit{weak entropy solution} to \eqref{eq:system} is a function $\rhob \in \mathbf{C}^0\mleft( [0,T]; L^1\mleft( \R^2;\R^K \mright) \mright)\cap L^\infty\mleft([0,T];BV\mleft(\R;\R^K\mright)\mright)$ if, for $k=1,\ldots,K$ $\rho^k$ is a Kru\v{z}kov solution to the Cauchy problem 
    \begin{equation}
    \begin{cases}
        \begin{aligned}
            &\ddt \rho^k + \Div \fb^k\mleft( t, \xb, \rho^k, \contconv{\rhob} \mright)  = 0 \quad &&,(t, \xb) \in \R^+ \times \R^2,\\
            &\rho^k(0,\xb) = \rho^k_0(\xb)&&, \xb \in  \R^2. 
        \end{aligned}
    \end{cases}
  \end{equation}
  setting $\contconv{\rhob} =  \etab *  \rhob$, i.e.\
  \begin{align*}
    \int_{0}^{T} \int_{\R^2} \mleft[ \left| \rho^k - \kappa \right| \ddt \phi + \textnormal{sgn}(\rho^k - \kappa)\mleft( \fb^k\mleft( t, \xb, \rho^k, \contconv{\rhob} \mright) - \fb^k\mleft( t, \xb, \kappa, \contconv{\rhob} \mright)
    \mright) \mydot \nabla_\xb \phi \mright] d\xb \; dt \\
    - \int_{0}^{T} \int_{\R^2}  \textnormal{sgn}(\rho^k-\kappa) \Div \fb^k\mleft( t, \xb, \kappa, \contconv{\rhob} \mright)   \phi d\xb \; dt + \int_{\R^2} \left| \rho_0^k(\xb)-\kappa \right| \phi(0,\xb) d\xb \geq 0 
  .\end{align*}
  holds for all $\kappa \in \R$ and $\phi \in \mathbf{C}_c^\infty$\textnormal{([0,T) $\times \R^2$; $\R^+$)}.

\end{definition}
During this work, we consider the following set of assumptions on the involved input data which guarantees the existence \reva{and uniqueness} of entropy solutions.
\begin{assumption}\label{asm}
  We pose the following assumptions:
  \begin{enumerate}
        \item[$(\rhob_m)$] there exists $0\leq \rho_m^k \in \R$ such that $\fb^k(\cdot ,\cdot ,\rho_m^k,\cdot )=0$ and $\rho_m^k \leq  \rho_0^{k}$ for $k=1,\ldots,K$.
        \item[$(\etab)$] $\etab \in \mleft( \mathbf{C}^2 \cap \mathbf{W}^{2,\infty} \cap L^1 \mright)\mleft( \R^2; \R^{M \times K} \mright)  $. 
  \end{enumerate}
  For the next assumptions, we denote by $\imgrho_k$ the image space of $\rho^k$.
  In the general case we set ${\imgrho_k = [\rho_m^k,\infty)}$. 
  If there additionally exists $\rho_M^k \in \R$ such that $\fb^k(\cdot ,\cdot ,\rho_M^k,\cdot )=0$ and $\rho_0^{k} \leq \rho_M^k$ we set $\imgrho_k = [\rho_m^k, \rho_M^k]$. 
  We further assume:
  \begin{enumerate}
        \item[$(\rhob_0)$] $\rhob_0 \in (L^1 \cap L^\infty \cap BV)(\R^2;\R^K)$ and $\rho_0^k(\xb) \in \imgrho_k$ for $k=1,\ldots,K$.
        \item[$(\boldsymbol{f})$] $\fb^k \in \mleft( \mathbf{C}^2 \cap \mathbf{W}^{2,\infty} \mright)\mleft( \R^+ \times \R^2 \times \imgrho_k \times \R^M ; \R^2 \mright)$ 
        and with abuse of notation, the first- and second-order derivatives of $\fb^k$ with respect to $t,x_1,x_2$, as well as the gradient and Laplacian of $\Rb$, can be bounded by the following estimates $\mleft| \partial^\alpha_{t,x_1,x_2,\Rb}\fb^k(t,\xb,\rho,\Rb) \mright| \le \mathcal{M}\,|\rho|,$ for all multi-indices $\alpha$ with $|\alpha|=1,2$, $\rho \in \imgrho_k$ \revb{and a constant $\mathcal{M} \in \R^+$}.
  \end{enumerate}
\end{assumption}
Note that the assumptions are rather similar to \cite{ACG15}. 
Here, the postulated regularity on the involved functions only needs to hold on the respective domains for the state variables $\rho^k$, $k=1,\ldots,K$ denoted by $\imgrho_k$.
Later in Thm.~\ref{eq:maxprinciple1}, we show that if $\rho_0^{k} \in \imgrho_k$, $k=1,\ldots,K$, the solution remains in this set.
It is also possible to weaken the assumption $(\boldsymbol{f})$  as outlined in the next remarks.
\begin{remark}
  For compactly supported initial data and finite times the solution of \eqref{eq:system} is also compactly supported in space.
  Therefore, it is sufficient to satisfy the assumptions $(\boldsymbol{f})$ and $(\etab)$ for any compact subsets associated with the time space $\R^+$ and the spatial domain $\R^2$. 
  Similar to \cite[Rem.~2.1]{ACG15}, we can follow the proof of \cite[Thm.~2.2]{borsche2015differential},
  such that the existence of the upper bound $\mathcal{M}\,|\rho|$ on the derivatives is already implied and does not need to be assumed in this case.
\end{remark}
\begin{remark}
  Furthermore, the proofs outlined below can be applied under the weaker assumptions on the fluxes: $\fb^k$ is $C^1$ with respect to $\rho^k$, 
  the gradients with respect to $\Rb$ and $\xb$ are Lipschitz continuous with respect to $\rho^k$ 
  and the partial derivative with respect to $\rho^k$ is $C^1$ in all other variables. 
  For simplicity, we chose $\fb^k$ to be $C^2$ with respect to $\rho^k$.
\end{remark}

Further, assumption $(\boldsymbol{f})$ implies that $\fb$ and its first order derivatives are Lipschitz continuous in all arguments.
We denote the corresponding constants, for instance, by $\Lipf{\rho}{x_1} \coloneq \max_{k \in  K} \| \partial^2_{\rho,x_1} \fb^k \|_{\infty}$ and analogously for all other derivatives.
In general, $L^p$-norms, $1\leq p\leq \infty$, of vectorial functions are defined componentwise, i.e.\ \revb{$\norm{g}_{L^p}\coloneq \sum_{m=1}^{\widetilde{M}} \norm{g_m}_{L^p}$ for $g \in L^p(\R^2;\R^{\widetilde{M}})$ with $\widetilde{M} \in \N$}.
In addition, for $a,b\in\R^K$ $a\,\mydot\, b$ denotes the standard scalar product of $a$ and $b$.

\section{A class of monotone-based first order schemes}\label{sec:scheme}
We discretize equidistantly in space and time. In space, the step sizes are $\dx$ and $\dy$, which correspond to the two dimensions, such that we have rectangular cells $C_{i,j}=[x_1^{i-\frac{1}{2}},x_1^{i+\frac{1}{2}}) \times [x_2^{j-\frac{1}{2}}, x_2^{j+\frac{1}{2}} )$ with centered nodes $\xb^{i,j} = (x_1^i, x_2^j) = (i\dx, j\dy)$, $i,j \in \Z$.
Additionally, the time is discretized with step size $\dt$, such that $t^n = n \dt$ with $n \in \N$.
We define the following piecewise constant finite volume approximation
$$
\rho_\Delta^k\mleft( t, \xb \mright) = \sum_{i,j \in \Z} \rho_{i,j}^{k,n} \; \mathlarger{\chi}_{[t^n, t^{n+1}) \times C_{i,j}}   \mleft( t, \xb \mright)
$$
for $k=1,\ldots,K$.
Here, the cell-averages $\rho_{i,j}^{k,n}$ over the cells $C_{i,j}$ at time $t^n$ are similarly defined as the initial data which is given for $k=1,\ldots,K$ by 
\begin{equation}\label{eq:rho0_L1}
  \rho_{i,j}^{k,0}= \frac{1}{\dx \dy} \int_{C_{i,j}} \rho_0^k\mleft( \xb \mright) d\xb.
\end{equation}
We start by describing the approximation of the nonlocal term.
For a fixed time $t^n$, the nonlocal terms $\etab * \rhob $ are approximated at a cell interface $\xb^{i+\frac{1}{2},j}$ or $\xb^{i,j+\frac{1}{2}}$ and denoted by $\Rb_{i+\frac12,j}^n$ and $\Rb_{i,j+\frac12}^n$, respectively.
Since the components of $ \etab * \rhob$ consist of different convolutions, i.e.\ $\mleft( \etab * \rhob  \mright)_m \mleft( t, \xb \mright) = \sum_{k=1}^{K} \mleft( \eta^{m,k}* \rho^k\mleft(t  \mright) \mright) \mleft(\xb \mright)$, $m=1,\ldots, M$, every convolution is approximated by composite midpoint rule
\begin{subequations}
\mathtoolsset{showonlyrefs=false}
\begin{align}\label{eq:Ri}
\revd{\sum _{k=1} ^K} \mleft( \eta^{m,k} * \rho_\Delta^k(t^n) \mright) \mleft( \xb^{i+\frac{1}{2},j} \mright) 
    &\approx \dx \dy \revd{\sum _{k=1} ^K}\sum_{p, q \in \Z} \eta^{m,k} \mleft( (p+{\textstyle \frac{1}{2}})\dx, q \dy  \mright) \rho^{k,n}_{i-p,j-q} \eqdef \mleft( \Rb_{i+\frac12,j}^n \mright)_m, 
\end{align}
\vspace{-\baselineskip}
\vspace{-\baselineskip}
\begin{align}
  \revd{\sum _{k=1} ^K} \mleft( \eta^{m,k} * \rho_\Delta^k(t^n) \mright) \mleft( \xb^{i,j+\frac{1}{2}} \mright) &\approx \dx \dy \revd{\sum _{k=1} ^K}\sum_{p, q \in \Z} \eta^{m,k} \mleft(p \dx,  (q+{\textstyle \frac{1}{2}}) \dy \mright) \rho^{k,n}_{i-p,j-q} \eqdef \mleft( \Rb_{i,j+\frac12}^n \mright)_m.\label{eq:Rj}
\end{align}
\end{subequations}
By this approximation the flux of \eqref{eq:system} reduces to $\widetilde{f}^k_1( \rho^k; t^n, \reva{\xb^{i+\frac{1}{2},j}} ) = f_1^k( t^n, \reva{\xb^{i+\frac{1}{2},j}}, \rho^k, \Rb_{i+\frac12,j}^n )$ in the first component, which we denote by $\redfluxk{1}{i+\frac{1}{2}}{j}{\rho^k}$ and analogously for the second component.
It is important to note that at the cell interface $\xb^{i+\frac{1}{2},j}$ or $\xb^{i,j+\frac{1}{2}}$ the flux is continuous in $\xb$ and the same holds for the convolutions due to the regularity assumptions on the kernel. 
The only discontinuities arise from the state variable $\rho^k$.

  \begin{remark}[Decoupling of the system]
  Since the system \eqref{eq:system} is coupled solely by the convolution terms, the system based on the reduced fluxes $\widetilde{\fb}^k$, $k=1,\ldots,K$ 
  is decoupled at a fixed time $t^n$. 
  Hence, the solution for the next time step can be computed based on the reduced fluxes by a finite volume scheme for every equation independently.
  For simplicity and to improve readability, we will often omit the superscript $k$ in the following proofs.
  \end{remark}

Now, a first order finite volume scheme for \eqref{eq:system} can be similarly derived to the local case:
Integrating \eqref{eq:system} over the cells $C_{i,j}$ and dividing by $\dx \dy$ provides a semi-discretization for the cell-averages. 
Approximating the integral of the flux term over the boundaries by the midpoint rule
and using a forward Euler step in time while discretizing the convolution terms as in \eqref{eq:Ri} and \eqref{eq:Rj}, results in a first order scheme given by 
\begin{equation}\label{eq:1stscheme}
  \begin{aligned}
        \rho_{i,j}^{k,n+1} = \rho_{i,j}^{k,n} &- \lambda_1 \mleft(\numfluxkn{1}{i+\frac{1}{2}}{j}{\rho_{i,j}}{\rho_{i+1,j}} 
        - \numfluxkn{1}{i-\frac{1}{2}}{j}{\rho_{i-1,j}}{\rho_{i,j}}\mright) 
        \\& -\lambda_2 \mleft(\numfluxkn{2}{i}{j+\frac{1}{2}}{\rho_{i,j}}{\rho_{i,j+1}} - \numfluxkn{2}{i}{j-\frac{1}{2}}{\rho_{i,j-1}}{\rho_{i,j}} \mright) 
    \end{aligned}
\end{equation}
with $\lambda_1 = \frac{\dt}{\dx}$ and $\lambda_2 = \frac{\dt}{\dy}$ and suitable numerical flux functions $F_1$ and $F_2$ which we specify in the following.
We note that from a semi-discretization higher order schemes can be derived in a similar fashion using appropriate approximations of the integrals over the cell boundaries and of the time integration. The resulting schemes do not use dimensional splitting as well and require consistent numerical fluxes, as considered in this work. 

The assumptions $(\fb)$ and $(\etab)$ concerning the regularity in space of the flux and nonlocal term allow us to extend the ideas in the one-dimensional case from \cite{FSS23}.
Similarly, we now pose some general assumptions on the numerical flux functions which guarantee the convergence towards weak entropy solutions of \eqref{eq:system}. 
Although the schemes are not monotone in every argument due to the nonlocal terms in \eqref{eq:Ri} and \eqref{eq:Rj}, they are based on and very similar to classical local \textit{monotone flux schemes} \cite[Def. 21.1]{EGH00}.
  \begin{definition}[\revb{Monotone}-based numerical flux function]\label{def:flux}
    Let the nonlocal terms be approximated as in \eqref{eq:Ri} and \eqref{eq:Rj} resulting in the approximate fluxes $\redfluxk{1}{i+\frac{1}{2}}{j}{\rho^k}\coloneq f_1^k\mleft( t^n, \xb_{i+\frac{1}{2},j}, \rho^k, \Rb_{i+\frac12,j}^n \mright) $ and 
    $\redfluxk{2}{i}{j+\frac12}{\rho^k}$ $\coloneq f_2^k\mleft( t^n, \xb_{i,j+\frac12}, \rho^k, \Rb_{i,j+\frac12}^n \mright)$.
    The numerical flux functions $\numfluxkn{1}{i+\frac{1}{2}}{j}{}{}, \numfluxkn{2}{i}{j+\frac{1}{2}}{}{}$, $i,j \in \Z, n \in \N, k=1,\ldots,K$ satisfy the following conditions
    \begin{enumerate}
        \item Consistency, i.e.\ for $\rho \in \imgrho_k$
        \begin{align*}
          \numflux{1,k}{i+\frac{1}{2}}{j}{\rho}{\rho} = \redfluxk{1}{i+\frac{1}{2}}{j}{\rho},\quad\numflux{2,k}{i}{j+\frac{1}{2}}{\rho}{\rho} = \redfluxk{2}{i}{j+\frac{1}{2}}{\rho}
        .\end{align*}
        \item The maps $\mleft( a,b \mright) \mapsto \numflux{1,k}{i+\frac{1}{2}}{j}{a}{b}$ and $\mleft( a,b \mright) \mapsto \numflux{2,k}{i}{j+\frac{1}{2}}{a}{b}$ from $\imgrho_k \times \imgrho_k$ to $\R$ are nondecreasing in the first argument and nonincreasing in the second. 
        \item Lipschitz continuity in the two arguments with Lipschitz constants
        $L_{1,1},L_{1,2}>0$ for $\numfluxkn{1}{i+\frac{1}{2}}{j}{}{}$ and $L_{2,1},L_{2,2} >0$ for $\numfluxkn{2}{i}{j+\frac{1}{2}}{}{}$.
        \item There exist constants $L_{1,1}', L_{1,2}', L_{2,1}', L_{2,2}' >0 $ such that for $\rho_{i-1,j}^{k,n}, \rho_{i,j}^{k,n}, \rho_{i,j-1}^{k,n}  \in  \imgrho_k$
        \begin{align*}
          \left| \numfluxkn{1}{i-\frac{1}{2}}{j+1}{\rho_{i-1,j}}{\rho_{i,j}}
                - \numfluxkn{1}{i-\frac{1}{2}}{j+1}{\rho_{i,j}}{\rho_{i,j}}
                -  \numfluxkn{1}{i-\frac{1}{2}}{j}{\rho_{i-1,j}}{\rho_{i,j}} 
                + \numfluxkn{1}{i-\frac{1}{2}}{j}{\rho_{i,j}}{\rho_{i,j}} \right|\\
                 \leq (\dy \left| \rho_{i,j}^n - \rho_{i-1,j}^n \right| + \mathcal{M} |\rho_{i,j}^n| \dy^2) L_{1,1}'\\
          \left| \numfluxkn{2}{i+1}{j-\frac{1}{2}}{\rho_{i,j-1}}{\rho_{i,j}}
                - \numfluxkn{2}{i+1}{j-\frac{1}{2}}{\rho_{i,j}}{\rho_{i,j}}
                -  \numfluxkn{2}{i}{j-\frac{1}{2}}{\rho_{i,j-1}}{\rho_{i,j}} 
                + \numfluxkn{2}{i}{j-\frac{1}{2}}{\rho_{i,j}}{\rho_{i,j}} \right|\\
                 \leq (\dx \left| \rho_{i,j}^n - \rho_{i,j-1}^n \right|+ \mathcal{M} |\rho_{i,j}^n| \dx^2) L_{2,1}'
        \end{align*}
        and analogously for changes in the second component of $F_1, F_2$.

    \end{enumerate}
  \end{definition}

\begin{remark}
  The $\dx^2$ and $\dy^2$ in \revb{condition} 4 arise from the difference in space and the nonlocal term. A suitable approximation of the latter is needed, as e.g.\ in \eqref{eq:Ri} and \eqref{eq:Rj}.
\end{remark}

  To prove convergence of the numerical solution to a weak entropy solution of \eqref{eq:system} we require the following CFL condition 
\begin{equation}\label{eq:CFL}
  \lambda_1 := \frac{\dt}{\dx} \leq \frac{1}{2 \mleft( L_{1,1} + L_{1,2} \mright)}, \quad \lambda_2 := \frac{\dt}{\dy} \leq \frac{1}{2 \mleft( L_{2,1} + L_{2,2} \mright)}.
\end{equation}
The approach presented above in Def.~\ref{def:flux} is applicable to the one-dimensional case as well.
Therefore, it generalizes the approaches in \cite{FSS23}, which only considered a specific choice for the flux function.

Before we present our main result, we give an example of numerical flux functions that fits into the framework of the monotone-based schemes proposed in Def.~\ref{def:flux}, i.e.\ the Lax-Friedrichs-type scheme as introduced in \cite{ACG15}
\begin{align}\label{eq:LxF}
    \numfluxkn{1}{i+\frac{1}{2}}{j}{\rho_{i,j}}{\rho_{i+1,j}} = 
      \frac{1}{2}\Bigl( \redfluxk{1}{i+\frac{1}{2}}{j}{\rho^{\reva{n}}_{i,j}} + \redfluxk{1}{i+\frac{1}{2}}{j}{\rho_{i+1,j}^n} \Bigr) - \frac{\alpha}{2} \mleft( \rho_{i+1,j}^n - \rho_{i,j}^n \mright)
  \end{align}
  with $\alpha \geq \sup_{\rho \in \imgrho_{\revb{k}}} |\partial_\rho \redfluxk{1}{i+\frac{1}{2}}{j}{\rho}|$ in the first component and analogously for the second component in the $x_2$-direction. 
  This scheme \revd{satisfies} assumption~4 of Def.~\ref{def:flux}, which is shown in the appendix.
  We note that in \cite{CR19} another Lax-Friedrichs-type scheme for a particular flux function is considered approximating the convolution at the cell centers instead of at the cell interfaces.

Further schemes can be derived for specific flux functions.
In particular, in many applications of nonlocal conservation laws, e.g.\ pedestrian \cite{ACG15,CR19,BGIV20,goatin2025pedestrians,GR24,CGL12,colombo2018nonlocal} or material flow \cite{GGZ25}, the flux function is multiplicative,
i.e.\ it can be written as $f^k\mleft( t,\xb, \rho^k, \etab * \rhob \mright) =g^k\mleft( \rho^k \mright)  \nub^k\mleft( t,\xb,\etab * \rhob \mright)$. 
In this case we obtain a suitable class that fits into the framework of Def.~\ref{def:flux}. In particular, assumption~4 of Def.~\ref{def:flux} holds immediately according to Prop.~\ref{prop:dblLipmult}.

\begin{remark}[Multiplicative flux functions]\label{rem:multflux}
If $\reva{\fb}^k\mleft( t,\xb, \rho^k, \etab * \rhob \mright) =g^k\mleft( \rho^k \mright)  \nub^k\mleft( t,\xb,\etab * \rhob \mright)$ holds (with appropriate assumption on $g^k$ and $\nub^k$),
we can, after approximating the convolutions at $( t^n, \xb^{i+\frac{1}{2},j})$, see \eqref{eq:Ri}, evaluate $\nu_1^k$ and denote 
$V_{i+\frac{1}{2},j}^{1,k,n} = \nu_1^k\mleft( t^n,\xb^{i+\frac{1}{2},j}, \Rb_{i+\frac{1}{2},j} \mright) $ to consider the reduced flux $\widetilde{f}^k_1\mleft( \rho^k; t^n, \xb_{i+\frac{1}{2},j} \mright)$ ${= g^k(\rho^k) V_{i+\frac{1}{2},j}^{1,k,n}}$. 
This allows us to make a specific approach on the numerical flux, namely to split up the reduced flux into the absolute value $\left|V_{i+\frac{1}{2},j}^{1,k,n}\right|$ and $g^k(\rho^k) \textnormal{sgn}\mleft( V_{i+\frac{1}{2},j}^{1,k,n} \mright)$ to only apply a numerical flux function $G_1^k$ on the latter. Thus, we obtain 
\begin{equation}\label{eq:F=GV}
      F_1^k\mleft( a,b, V_{i+\frac{1}{2},j}^{1,k,n} \mright) \defeq G_1^k\mleft( a,b; \textnormal{sgn}\mleft( V_{i+\frac{1}{2},j}^{1,k,n} \mright) \mright) \left| V_{i+\frac{1}{2},j}^{1,k,n} \right|  
\end{equation}
requiring the same assumptions as in Def.~\ref{def:flux}. In this case, \revb{condition}~4 is always satisfied, see Prop.~\ref{prop:dblLipmult} in the appendix. Again, the analogous holds for the $x_2$-direction. This approach provides us with a new Lax-Friedrichs-type scheme
\begin{equation}\label{eq:LxFnew}
  \begin{aligned}
  F_1^k\mleft( \rho_{i,j}^n,\rho_{i+1,j}^n, V_{i+\frac{1}{2},j}^{1,k,n}\mright) &= G_1\mleft( \rho_{i,j}^n,\rho_{i+1,j}^n, \textnormal{sgn}(V_{i+\frac{1}{2},j}^{1,k,n})\mright) \left| V_{i+\frac{1}{2},j}^{1,k,n} \right| \\
  & \hspace{-2cm}= \frac{1}{2}\Bigl( \mleft( g^{\reva{k}}( \rho_{i,j}^n) \textnormal{sgn}(V_{i+\frac{1}{2},j}^{1,k,n}) + g^{\reva{k}}(\rho_{i+1,j}^n) \textnormal{sgn}(V_{i+\frac{1}{2},j}^{1,k,n}) \mright) - \alpha \mleft( \rho_{i+1,j}^n - \rho_{i,j}^n \mright) \Bigr) \left| V_{i+\frac{1}{2},j}^{1,k,n} \right|
\end{aligned}
\end{equation}
with $\alpha \geq \sup_{\rho\in \imgrho_k} \left|  \reva{(g^k)}'(\rho) \right| $. 
Moreover, a Godunov-type scheme is then given by
\begin{align}\label{eq:God}
  F_1^k\mleft( \rho_{i,j}^n,\rho_{i+1,j}^n, V_{i+\frac{1}{2},j}^{1,k,n}\mright) = g^\reva{k}( \rho^*(\rho_{i,j}^n,\rho_{i+1,j}^n)) V_{i+\frac{1}{2},j}^{1,k,n} 
\end{align}
where $\rho^*(\rho_{i,j}^n,\rho_{i+1,j}^n)$ is the solution to the one-dimensional Riemann problem with the flux function $\rho \mapsto g^k(\rho) \textnormal{sgn}\mleft( V_{i+\frac{1}{2},j}^{1,k,n} \mright)$. 
\end{remark}

Now, we state our \reva{first} main result.

\begin{theorem}\label{thm:mainres}
  Let the Asm.~\ref{asm} hold and $\rhob_0 \in \textnormal{BV}(\R^2) \cap L^\infty(\R^2)$.
  Then a monotone-based numerical scheme \eqref{eq:1stscheme} with a flux function from Def.~\ref{def:flux} converges for $\dx, \dy, \dt \to 0$ in $\left(L_{loc}^{1}\mleft( \R^2 \mright)\right)^K$ to the unique weak entropy solution $\rhob= \rhob(t,\xb) \in C^0(\R^+; L^1(\R^2;\R^K))$ of \eqref{eq:system} as in Def.~\ref{Def:entropysol}, where $\dx = \delta \dy$ \revb{for a fixed $\delta \in \R^+$} and $\dt$ bounded by the CFL condition \eqref{eq:CFL}. 
  Moreover this solution satisfies $\rho^k \in \imgrho_k,$ ${k=1,\ldots,K}$.
\end{theorem}
We close this section with the following remark.
\begin{remark}[Dimensional splitting]
  The scheme in \eqref{eq:1stscheme} works without dimensional splitting to make it extendable to higher order schemes.
  Instead, using dimensional splitting permits CFL numbers twice that of \eqref{eq:1stscheme}.  
  Using the monotone-based flux functions from Def.~\ref{def:flux}, Thm.~\ref{thm:mainres} can easily be adapted to a dimensional splitting approach.
\end{remark} 

\subsection{Convergence proof}
The proof of Thm.~\ref{thm:mainres} is divided in several sub-results which we will establish below.
\begin{theorem}[Maximum principle]\label{eq:maxprinciple1}
  Let the Asm.~\ref{asm} hold.
  Then a solution computed by a numerical scheme \eqref{eq:1stscheme} satisfying Def.~\ref{def:flux} and the CFL condition \eqref{eq:CFL} remains in the set $\imgrho_k$, i.e. 
  \begin{equation}
    \rho_{i,j}^{k,n} \geq \rho_m^k \geq 0, \quad \forall i,j \in \Z, n \in \N, k\in \{1,\dots,K\},
  \end{equation}
  and if there additionally exists $\rho_M^k \in \R$ with $\fb^k(\cdot ,\cdot ,\rho_M^k,\cdot )=0$ and $\rho_0^{k} \leq \rho_M^k $, then
  \begin{equation}
    \rho_{i,j}^{k,n} \leq \rho_M^k, \quad \forall i,j \in \Z, n \in \N, k\in \{1,\dots,K\}.
  \end{equation}
\end{theorem}
\begin{proof}
  We prove that $\rho_{i,j}^{k,n} \geq \rho_m^k, i,j \in \Z, n \in \N$ by induction and for the sake of simplicity we omit the superscript $k$. 
  For $n=0$ it holds by the approximation of the initial data in \eqref{eq:rho0_L1}.
  Suppose it is true for a fixed $n\in \N$. Due to the monotonicity in the first two arguments we have
  \begin{align*}
    \numflux{1}{i+\frac{1}{2}}{j}{\rho_{i,j}^n}{\rho_{i+1,j}^n}
    - \numflux{1}{i-\frac{1}{2}}{j}{\rho_{i-1,j}^n}{\rho_{i,j}^n}
    \leq \numflux{1}{i+\frac{1}{2}}{j}{\rho_{i,j}^n}{\rho_m} 
    - \numflux{1}{i-\frac{1}{2}}{j}{\rho_m}{\rho_{i,j}^n} 
  .\end{align*}
  Based on the assumption $\fb(\cdot ,\cdot ,\rho_m,\cdot )=0$ and the consistency of the numerical flux we obtain that $ {\numflux{1}{i\pm \frac{1}{2}}{j}{\rho_m}{\rho_m}=0}$. Thus, these terms can be added to exploit the Lipschitz continuity 
  \begin{align*} 
     &\numflux{1}{i+\frac{1}{2}}{j}{\rho_{i,j}^n}{\rho_m} 
    - \numflux{1}{i-\frac{1}{2}}{j}{\rho_m}{\rho_{i,j}^n}\\
     \leq & \left| \numflux{1}{i+\frac{1}{2}}{j}{\rho_{i,j}^n}{\rho_m} 
    - \numflux{1}{i+\frac{1}{2}}{j}{\rho_m}{\rho_m}\right| + \left| \numflux{1}{i-\frac{1}{2}}{j}{\rho_m}{\rho_m} 
    - \numflux{1}{i-\frac{1}{2}}{j}{\rho_m}{\rho_{i,j}^n}\right| \\
     \leq &\mleft( L_{1,1} + L_{1,2} \mright)  \left| \rho_{i,j}^n - \rho_m \right| = \mleft( L_{1,1} + L_{1,2} \mright)  \mleft( \rho_{i,j}^n - \rho_m \mright).
  \end{align*}
  Similarly, for the $x_2$-direction
  $$ \numflux{2}{i}{j+\frac{1}{2}}{\rho_{i,j}^n}{\rho_{i,j+1}^n} 
    - \numflux{2}{i}{j-\frac{1}{2}}{\rho_{i,j-1}^n}{\rho_{i,j}^n}
   \leq  \mleft( L_{2,1} + L_{2,2} \mright)  \mleft( \rho_{i,j}^n - \rho_m \mright). $$
  Thus, the cell-average value at time step $n+1$ with scheme \eqref{eq:1stscheme} and CFL condition \eqref{eq:CFL} fulfills
  \begin{align*}
    \rho_{i,j}^{n+1}
    &\geq \rho_{i,j}^{n} - \lambda_1 \mleft( L_{1,1} + L_{1,2} \mright)  \mleft(  \rho_{i,j}^n - \rho_m \mright) -  \lambda_2 \mleft( L_{2,1} + L_{2,2} \mright)  \mleft(\rho_{i,j}^n - \rho_m\mright)\\
    &\geq \rho_{i,j}^n - \mleft( \rho_{i,j}^n - \rho_m \mright) = \rho_m 
  .\end{align*}
  The proof for the upper bound in the case of an existing $\rho_M^k \in \R$ is similar.
\end{proof}
Since $\rho_m^k\geq 0$ the positivity of solutions is preserved and by standard arguments, see e.g.\ \cite[Lem.~2.4]{ACT15}, we obtain the following result on the $L^1$-norm.
\begin{corollary}[$L^1$-bound]\label{cor:L1}
  Under Asm.~\ref{asm} the approximate solution computed by scheme \eqref{eq:1stscheme} with a numerical flux as defined in Def.~\ref{def:flux} and a time step size fulfilling \eqref{eq:CFL} satisfies
  $ \left\lVert \rho_{\Delta}^k\mleft( t^n,\cdot \mright) \right\rVert _{L^1} =  \left\lVert \rho_0^k \right\rVert _{L^1} $ for $k=1,\ldots,K$, where
  ${\left\lVert \rho_{\Delta}^k\mleft( t^n,\cdot \mright) \right\rVert _{L^1} = \dx \dy \sum_{i,j \in \Z} |\rho_{i,j}^n| }$.
\end{corollary}
Next, we prove that the total variation remains bounded for every finite time $T>0$.
\begin{theorem}[BV estimate in space]\label{thm:BV}
  Let Asm.~\ref{asm} hold. Using a scheme \eqref{eq:1stscheme} with a flux from {Def.~\ref{def:flux}} and a time step fulfilling the CFL conditions \eqref{eq:CFL} there exist $\mathcal{K}_1, \mathcal{K}_2 >0$ only depending on $\rhob_0, \etab, \fb^k$, $k=1,\ldots,K$, the Lipschitz constants in Def.\ \ref{def:flux}  and $\delta=\frac{\dx}{\dy}$ such that
  \begin{align*}
    \text{TV}  \mleft(\rho_{\Delta}^k\mleft( t^n,\cdot \mright)  \mright) &= \sum_{i,j\in \Z} \revb{\Bigl(} \left| \rho_{i+1,j}^{k,n} - \rho_{i,j}^{k,n} \right| \dy +  \left| \rho_{i,j+1}^{k,n} - \rho_{i,j}^{k,n} \right| \dx \revb{\Bigr)}\\
    &\leq \exp\mleft( t^{n} \mathcal{K}_1 \mright) \text{TV}\mleft(\rho^k_0\mright) + \frac{\mathcal{K}_{2}}{\mathcal{K}_1} \mleft( \exp\mleft( t^{n} \mathcal{K}_{1} \mright) -1 \mright)
  \end{align*}
  for $k=1,\dots,K$.
\end{theorem}
\begin{proof}
  We prove the bound on $\sum_{i,j\in \Z}  \left| \rho_{i+1,j}^{k,n} - \rho_{i,j}^{k,n} \right| \dy$, the remainder is analogue. The scheme \eqref{eq:1stscheme} is reformulated by defining
  \begin{equation*}
    \begin{aligned}
       H_{i,j}^{1,n} \defeq \frac12 \rho_{i,j}^{n} &- \lambda_1 \mleft(\numfluxkn{1}{i+\frac{1}{2}}{j}{\rho_{i,j}}{\rho_{i+1,j}} 
        - \numfluxkn{1}{i-\frac{1}{2}}{j}{\rho_{i-1,j}}{\rho_{i,j}}\mright) 
        \\ H_{i,j}^{2,n} \defeq \frac12 \rho_{i,j}^{n} & -\lambda_2 \mleft(\numfluxkn{2}{i}{j+\frac{1}{2}}{\rho_{i,j}}{\rho_{i,j+1}} - \numfluxkn{2}{i}{j-\frac{1}{2}}{\rho_{i,j-1}}{\rho_{i,j}}  \mright)
    \end{aligned}
  \end{equation*}
  such that
  \begin{equation}\label{eq:H1+H2}
    \begin{aligned}
        \rho_{i,j}^{n+1} = H_{i,j}^{1,n} + H_{i,j}^{2,n}
    \end{aligned}
  \end{equation}
  and
  \begin{align}
      \sum_{i,j\in \Z}\left| \rho_{i+1,j}^{n+1} - \rho_{i,j}^{n+1} \right| \dy 
      \leq  \sum_{i,j\in \Z}\left| H_{i+1,j}^{1,n} - H_{i,j}^{1,n}  \right| \dy \label{eq:diffH12a}\\
      + \sum_{i,j\in \Z}\left| H_{i+1,j}^{2,n} - H_{i,j}^{2,n}  \right| \dy.\label{eq:diffH12b}
  \end{align}
  We combine the ideas from \cite{ACG15,FSS23, RWGG20} and start with considering \eqref{eq:diffH12a} \revc{by rewriting $H_{i,j}^{1,n}$, whereby we denote ${\Delta_{i+\frac{1}{2},j}^n \defeq \rho_{i+1,j}^n - \rho_{i,j}^n}$, $i,j \in \Z$ and}
  \begin{equation}\label{eq:ab}
    \begin{aligned}
      a_{i-\frac{1}{2},j}^n & \defeq \begin{cases}
       \lambda_1  \dfrac{ \numflux{1}{i-\frac{1}{2}}{j}{\rho_{i-1,j}^n}{\rho_{i,j}^n} - \numflux{1}{i-\frac{1}{2}}{j}{\rho_{i,j}^n}{\rho_{i,j}^n}}{ \rho_{i-1,j}^n - \rho_{i,j}^n}, &\text{if  }\rho_{i-1,j}^n \neq \rho_{i,j}^n,\\ 
       0, &\text{else},
      \end{cases} \\ 
    b_{i+\frac{1}{2},j}^n & \defeq \begin{cases}
       \lambda_1 \dfrac{ \numflux{1}{i+\frac{1}{2}}{j}{\rho_{i,j}^n}{\rho_{i+1,j}^n} - \numflux{1}{i+\frac{1}{2}}{j}{\rho_{i,j}^n}{\rho_{i,j}^n}}{ \rho_{i,j}^n - \rho_{i+1,j}^n}, &\text{if  }\rho_{i,j}^n \neq \rho_{i+1,j}^n,\\ 
       0, &\text{else},
      \end{cases}
    \end{aligned}
  \end{equation}
such that
\begin{subequations}\label{eq:Hdiff}
\mathtoolsset{showonlyrefs=false}
\begin{align}
H_{i+1,j}^{1,n} - H_{i,j}^{1,n} =& \mleft( \frac12 - a_{i+\frac{1}{2},j}^n - b_{i+\frac{1}{2},j}^n \mright) \Delta_{i+\frac{1}{2},j}^n + b_{i+\frac{3}{2},j}^n \Delta_{i+\frac{3}{2},j}^n +  a_{i-\frac{1}{2},j}^n  \Delta_{i-\frac{1}{2},j}^n \hspace{1cm} \label{eq:Hdiffa}
\end{align}
\vspace{-\baselineskip}
\begin{equation}\label{eq:Hdiffb}
\begin{split}
 &\hspace{2.5cm} + \lambda_1 \Bigl(\flux{1}{i+\frac{1}{2}}{j}{\rho_{i+1,j}^{\reva{n}}}  - \flux{1}{i+\frac{3}{2}}{j}{\rho_{i+1,j}^{\reva{n}}} - \flux{1}{i-\frac{1}{2}}{j}{\rho_{i,j}^{\reva{n}}}  + \flux{1}{i+\frac{1}{2}}{j}{\rho_{i,j}^{\reva{n}}}
    \Bigr)\reva{.}
\end{split}
\end{equation}
\end{subequations}
\revc{We first determine a bound on \eqref{eq:Hdiffb} and begin by examining the differences $\flux{1}{i-\frac{1}{2}}{j}{\rho_{i,j}^{\reva{n}}}  - \flux{1}{i+\frac{1}{2}}{j}{\rho_{i,j}^{\reva{n}}}$ and $\flux{1}{i+\frac{1}{2}}{j}{\rho_{i+1,j}^{n}}  - \flux{1}{i+\frac{3}{2}}{j}{\rho_{i+1,j}^{n}}$. 
To this aim the mean value theorem, assumption $(\mathbf{f})$ and the estimates in \cite[Lem.~A.2]{ACG15} for differences in the nonlocal terms are used repeatedly such that 
\begin{align*}
  |\eqref{eq:Hdiffb}| \leq & \, \mathcal{M} \dx^2 |\rho_{i,j}^n| K_{3,1}+ \Lipf{\rho}{x_1} |\Delta_{i+\frac{1}{2},j}^n| \dx +\Lipf{\rho}{R} |\Delta_{i+\frac{1}{2},j}^n| \dx \norm{\ddx \etab}_\infty \norm{\rhob_0}_{L^1}
\end{align*} 
with $K_{3,1} \defeq \mleft( 2 + 2\norm{\ddx \etab}_\infty \norm{\rhob_0}_{L^1} + 2 \norm{\ddx \etab}_\infty^2 \norm{\rhob_0}_{L^1}^2 + 2 \norm{\ddxx \etab}_\infty \norm{\rhob_0}_{L^1} \mright)$.}\\
The terms $\mleft( \frac12 - a_{i+\frac{1}{2},j}^n - b_{i+\frac{1}{2},j}^n \mright)$, $a_{i-\frac{1}{2},j}^n$ and $b_{i+\frac{3}{2},j}^n$ are non-negative due to the Lipschitz condition on $\numflux{1}{i \pm \frac{1}{2}}{j}{}{}$ and the CFL condition \eqref{eq:CFL}. \revc{Using $\dt = \lambda_1 \dx$ leads us to
\begin{align*}
  \left|H_{i+1,j}^{1,n} - H_{i,j}^{1,n}\right|
  & \leq \, \mleft( \frac12 - a_{i+\frac{1}{2},j}^n - b_{i+\frac{1}{2},j}^n \mright) |\Delta_{i+\frac{1}{2},j}^n| + b_{i+\frac{3}{2},j}^n |\Delta_{i+\frac{3}{2},j}^n| +  a_{i-\frac{1}{2},j}^n  |\Delta_{i-\frac{1}{2},j}^n|\\
  &+ \dt \mleft(  \Lipf{\rho}{x_1} + \Lipf{\rho}{R}\norm{\ddx \etab}_\infty \norm{\rhob_0}_{L^1} \mright) |\Delta_{i+\frac{1}{2},j}^n|+ \dt \mathcal{M} \dx |\rho_{i,j}^n| K_{3,1}.
\end{align*}}
Summing over $i,j \in \Z$ and rearranging the indices leaves us with
\begin{align*}
  \sum_{i,j\in \Z}& \left|H_{i+1,j}^{1,n} - H_{i,j}^{1,n}\right| \dy \\\leq &\dt  \mathcal{M} \norm{\rhob_0}_{L^1} K_{3,1} +\sum_{i,j\in \Z} \mleft( \frac{1}{2} + \dt \mleft(  \Lipf{\rho}{x_1} + \Lipf{\rho}{R}\norm{\ddx \etab}_\infty \norm{\rhob_0}_{L^1} \mright)  \mright) \dy |\Delta_{i+\frac{1}{2},j}^n|.
 \end{align*}
Thus, we have a bound for \eqref{eq:diffH12a} such that we can proceed with \eqref{eq:diffH12b}.
\revc{We rewrite}
\begin{subequations}
\begin{align}
H_{i+1,j}^{2,n} - H_{i,j}^{2,n}
=& \Delta_{i+\frac{1}{2},j}^n 
   \mleft( \tfrac{1}{2} - \tilde{a}_{i+1,j+\frac{1}{2}} - \tilde{b}_{i+1,j-\frac{1}{2}} \mright)
   + \Delta_{i+\frac{1}{2},j-1}^n \tilde{a}_{i+1,j-\frac{1}{2}}
   + \Delta_{i+\frac{1}{2},j+1}^n \tilde{b}_{i+1,j+\frac{1}{2}}
\label{eq:H2a}
\\[0.8em]
& + \lambda_2 \Bigl(
    \numflux{2}{i+1}{j-\frac{1}{2}}{\rho_{i,j-1}^n}{ \rho_{i,j}^n} - \numflux{2}{i+1}{j+\frac{1}{2}}{\rho_{i,j}^n}{ \rho_{i,j+1}^n} \notag \\[0.3em]
&\qquad\quad
    - \numflux{2}{i}{j-\frac{1}{2}}{\rho_{i,j-1}^n}{\rho_{i,j}^n}
    + \numflux{2}{i}{j+\frac{1}{2}}{\rho_{i,j}^n}{\rho_{i,j+1}^n}
    \Bigr)
\label{eq:H2b}
\end{align}
\end{subequations}
with 
\begin{align*}
  \tilde{a}_{i+1,j+\frac{1}{2}}^n &= \begin{cases}
      \lambda_2 \dfrac{\numflux{2}{i+1}{j+\frac{1}{2}}{\rho_{i+1,j}^n}{\rho_{i+1,j+1}^n}
      - \numflux{2}{i+1}{j+\frac{1}{2}}{\rho_{i,j}^n}{\rho_{i+1,j+1}^n}}
      { \rho_{i+1,j}^n - \rho_{i,j}^n}, &\text{if  }\rho_{i+1,j}^n \neq \rho_{i,j}^n,\\ 
      0, &\text{else},
    \end{cases} \\ 
  \tilde{b}_{i+1,j+\frac{1}{2}}^n &= \begin{cases}
      \lambda_2 \dfrac{
        \numflux{2}{i+1}{j+\frac{1}{2}}{\rho_{i,j}^n}{\rho_{i,j+1}^n}
        - \numflux{2}{i+1}{j+\frac{1}{2}}{\rho_{i,j}^n}{\rho_{i+1,j+1}^n}}
        { \rho_{i+1,j+1}^n - \rho_{i,j+1}^n}, &\text{if  }\rho_{i+1,j+1}^n \neq \rho_{i,j+1}^n,\\ 
      0, &\text{else}.
    \end{cases}
\end{align*}
Considering the absolute value of \eqref{eq:H2a} and summing over $i,j \in \Z$, \eqref{eq:H2a} can be bounded in the same way as above. 
Now, we pass to \eqref{eq:H2b} and obtain
\begin{subequations}
\begin{flalign}
&\numflux{2}{i+1}{j-\frac{1}{2}}{\rho_{i,j-1}^n}{\rho_{i,j}^n}
    - \numflux{2}{i+1}{j+\frac{1}{2}}{\rho_{i,j}^n}{\rho_{i,j+1}^n}
    -\numflux{2}{i}{j-\frac{1}{2}}{\rho_{i,j-1}^n}{\rho_{i,j}^n}
    + \numflux{2}{i}{j+\frac{1}{2}}{\rho_{i,j}^n}{\rho_{i,j+1}^n} \notag\\
    =  \phantom{-} &\numflux{2}{i+1}{j-\frac{1}{2}}{\rho_{i,j-1}^n}{\rho_{i,j}^n}
    - \numflux{2}{i+1}{j-\frac{1}{2}}{\rho_{i,j}^n}{\rho_{i,j}^n}
    -  \numflux{2}{i}{j-\frac{1}{2}}{\rho_{i,j-1}^n}{\rho_{i,j}^n} 
    + \numflux{2}{i}{j-\frac{1}{2}}{\rho_{i,j}^n}{\rho_{i,j}^n} \label{eq:a1}\\
    -&\numflux{2}{i+1}{j+\frac{1}{2}}{\rho_{i,j}^n}{\rho_{i,j+1}^n}
    + \numflux{2}{i+1}{j+\frac{1}{2}}{\rho_{i,j}^n}{\rho_{i,j}^n}
    +  \numflux{2}{i}{j+\frac{1}{2}}{\rho_{i,j}^n}{\rho_{i,j+1}^n} 
    - \numflux{2}{i}{j+\frac{1}{2}}{\rho_{i,j}^n}{\rho_{i,j}^n}\label{eq:a2}\\
    +& \flux{2}{i+1}{j-\frac{1}{2}}{\rho_{i,j}^n} - \flux{2}{i}{j-\frac{1}{2}}{\rho_{i,j}^n}
  - \flux{2}{i+1}{j+\frac{1}{2}}{\rho_{i,j}^n} + \flux{2}{i}{j+\frac{1}{2}}{\rho_{i,j}^n}.\label{eq:a3}
\end{flalign}
\end{subequations}
For the terms \eqref{eq:a1} and \eqref{eq:a2} the estimate from the \revb{condition}~4 in Def.~\ref{def:flux} can be used.
\revc{Proceeding with \eqref{eq:a3}, we get
\begin{align*}
  |\eqref{eq:a3}| 
  \leq &\Bigl| \ddx \fluxself{2}{\Bigl(\widetilde{x}_1^{i+\frac{1}{2}},x_2^{j-\frac{1}{2}}  \Bigr)^T\!\!}{\rho_{i,j}^{n}}{{\Rb}_{i,j-\frac{1}{2}}^n\!} \dx - \ddx \fluxself{2}{\Bigl(\widehat{x}_1^{i+\frac{1}{2}},x_2^{j+\frac{1}{2}}  \Bigr)^T\!\!}{\rho_{i,j}^{n}}{{\Rb}_{i,j+\frac{1}{2}}^n\!} \dx \Bigr|\\
  & + \Bigl| \nabla_{\Rb} \fluxself{2}{\xb^{i+1,j-\frac{1}{2}}}{\rho_{i,j}}{\widetilde{\Rb}_{i+\frac{1}{2},j-\frac{1}{2}}^n} \mydot \mleft( {\Rb}_{i+1,j-\frac{1}{2}}^n -{\Rb}_{i,j-\frac{1}{2}}^n \mright) \\
    & \qquad- \nabla_{\Rb} \fluxself{2}{\xb^{i+1,j+\frac{1}{2}}}{\rho_{i,j}}{\widehat{\Rb}_{i+\frac{1}{2},j+\frac{1}{2}}^n} \mydot \mleft( {\Rb}_{i+1,j+\frac{1}{2}}^n -{\Rb}_{i,j+\frac{1}{2}}^n \mright) \Bigr|,
\end{align*}
where $\widetilde{x}_1^{i+\frac{1}{2}}, \widehat{x}_1^{i+\frac{1}{2}} \in [x_1^{i},x_1^{i+1}]$, 
$\widetilde{\Rb}_{i+\frac{1}{2},j-\frac{1}{2}}^n$ between $\Rb_{i+1,j-\frac{1}{2}}^n$ and $\Rb_{i,j-\frac{1}{2}}^n$, and $\widehat{\Rb}_{i+\frac{1}{2},j+\frac{1}{2}}^n$ between $\Rb_{i+1,j+\frac{1}{2}}^n$ and $\Rb_{i,j+\frac{1}{2}}^n$.
Similar to the estimates above the mean value theorem, assumption $(\mathbf{f})$ and the estimates in \cite[Lem.~A.2]{ACG15} allow us to bound the first term.
For the second term, we \revd{additionally} use 
\begin{align}\label{eq:4Rdiff}
  \norm{\Rb_{i+1,j-\frac{1}{2}}^n - \Rb_{i,j-\frac{1}{2}}^n - \Rb_{i+1,j+\frac{1}{2}}^n + \Rb_{i,j+\frac{1}{2}}^n }_2
  \leq \,  \mleft( \left\lVert \partial_{x_1 x_1}^2 \etab \right\rVert _{\infty} 2 \dx^2 + \left\lVert \partial_{x_1 x_2}^2 \etab \right\rVert _{\infty}  \dx \dy\mright)  \left\lVert \rhob_0 \right\rVert _{L^1}\!,\;\;
\end{align}
which is proven in Appendix \ref{sec:appendix}, and 
\begin{align*}
  \left\lVert \widetilde{\Rb}_{i+\frac{1}{2},j-\frac{1}{2}}^n- \widehat{\Rb}_{i+\frac{1}{2},j+\frac{1}{2}}^n \right\rVert _2
  \leq 2 \dx \norm{\ddx \etab}_\infty \norm{\rhob_0}_{L^1} + \dy \norm{\ddy \etab}_\infty \norm{\rhob_0}_{L^1}  
.\end{align*}
Note that the intermediate values for $\xb$ and $\Rb$ not only differ in the $x_2$-direction but also in the $x_1$-direction. Therefore, we use the fixed relation $\delta = \frac{\dx}{\dy}$ and thus $\dx^2= \delta \dx \dy$ such that the summation over $i,j \in \Z$ contributes to $\norm{\rhob_0}_{L^1}$ and we get
the estimate $\sum_{i,j\in \Z}\left| \eqref{eq:a3} \right| \leq \mathcal{M} \left\lVert \rhob_0 \right\rVert _{L^1} K_{4,1}$ for a constant
\begin{align*}
  K_{4,1} \defeq \Bigl(& \delta+1+ \norm{\ddx \etab}_\infty \norm{\rhob_0}_{L^1}\bigl[1 + \norm{\rhob_0}_{L^1}\mleft( 2 \delta \norm{\ddx \etab}_\infty+\norm{\ddy \etab}_\infty \mright)\bigr]\\ 
  &\quad+  \norm{\rhob_0}_{L^1} \left[\norm{\ddy \etab}_\infty + \left\lVert \partial_{x_1 x_1}^2 \etab \right\rVert _{\infty} 2 \delta + \left\lVert \partial_{x_1 x_2}^2 \etab \right\rVert _{\infty} \right] \Bigr).
\end{align*}
With this and $\lambda_2 \dy = \dt$ it follows for \eqref{eq:diffH12b} that}
\begin{align*}
  \sum_{i,j\in \Z}\left| H_{i+1,j}^{2,n} - H_{i,j}^{2,n}  \right| \dy 
  \leq & \frac{1}{2} \sum_{i,j\in \Z} \left|\Delta_{i+\frac{1}{2},j}^n \right| \dy+ \dt \mathcal{M} \norm{ \rhob_0}_{L^1} K_{4,1} \\
  &+\dt \mleft( L_{2,1}' + L_{2,2}' \mright) \Bigl[\mathcal{M} \delta \norm{ \rhob_0}_{L^1} + \sum_{i,j\in \Z} \left| \rho_{i,j+1}^n - \rho_{i,j}^n \right| \dx \Bigr].
\end{align*}
Combining the estimates for \eqref{eq:diffH12a} and \eqref{eq:diffH12b} yields a bound for $\sum_{i,j\in \Z} \left| \rho_{i+1,j}^{n+1} - \rho_{i,j}^{n+1} \right| \dy$.
Proceeding analogously for $\sum_{i,j\in \Z} \left| \rho_{i,j+1}^{n+1} - \rho_{i,j}^{n+1} \right| \dx$ we obtain similar estimates with constants $K_{3,2},K_{4,2}$. Thus,
\begin{align*}
  \text{TV}&\mleft(\rho_{\Delta}^k\mleft( t^{n+1},\cdot \mright)  \mright)\leq (1+ K_{1,1} \dt)\sum_{i,j\in \Z} \left| \rho_{i+1,j}^n - \rho_{i,j}^n \right| \dy + (1+K_{1,2}\dt) \sum_{i,j\in \Z} \left| \rho_{i,j+1}^n - \rho_{i,j}^n \right| \dx + \mathcal{K}_2 \dt
\end{align*}
with positive constants
\begin{align*}
  K_{1,\ell} =& \mleft(  \Lipf{\rho}{x_\ell} + \Lipf{\rho}{R}\norm{\partial_{x_\ell} \etab}_\infty \norm{\rhob_0}_{L^1} \mright) + \mleft( L_{\ell,1}' + L_{\ell,2}' \mright)\quad\text{for }\ell\in \{1,2\},  \\
  \mathcal{K}_2 =&  \mathcal{M} \left\lVert \rhob_0 \right\rVert _{L^1} \biggl(K_{3,1}+K_{3,2}+K_{4,1}+K_{4,2} + \mleft( L_{1,1}' + L_{1,2}' \mright) \frac{1}{\delta} + \mleft( L_{2,1}' + L_{2,2}' \mright)  \delta
  \biggr)  
.\end{align*}
We define $\mathcal{K}_1\defeq \max\{ K_{1,1}, K_{1,2} \}$ to bound the total variation recursively 
\begin{align*}
  \text{TV}\mleft(\rho_{\Delta}^k\mleft( t^{n+1},\cdot \mright)  \mright) &\leq \mleft( 1+\mathcal{K}_1 \dt \mright) \text{TV}\mleft(\rho_{\Delta}^k\mleft( t^{n},\cdot \mright)  \mright) + \mathcal{K}_{2} \dt\\
  &\leq \exp\mleft( t^{n+1} \mathcal{K}_1 \mright) \text{TV}(\rho_0) + \frac{\mathcal{K}_{2}}{\mathcal{K}_1} \mleft( \exp\mleft( t^{n+1} \mathcal{K}_{1} \mright) -1 \mright)
.\end{align*}
\end{proof}
The BV estimate above allows us to deduce a time continuity estimate which gives us an estimate for the Lipschitz continuity in time.
\begin{proposition}[Time continuity estimate]\label{prop:timecont}
  Under the same assumptions as for Thm.~\ref{thm:BV} we obtain the following estimate
  $$ \left\lVert \rho_\Delta^k\mleft( t^{n+1},\cdot  \mright) - \rho_\Delta^k \mleft( t^{n},\cdot  \mright) \right\rVert _{L^1} \leq \dt \, \mathcal{K}_3 \; \text{TV}\mleft( \rho_\Delta^k\mleft(t^n, \cdot \mright) \mright) + \dt \, \mathcal{K}_4$$
  for $k=1,\dots,K$, $\mathcal{K}_4>0$ only depending on $\rhob_0, \etab, \fb^k$ and $\mathcal{K}_3>0$ depending on the Lipschitz constants of the numerical fluxes $F_1, F_2$. 
\end{proposition}
\begin{proof}
  We use the update in time given by \eqref{eq:1stscheme}.
  By adding and subtracting the terms 
  $\lambda_1 \numflux{1}{i-\frac{1}{2}}{j}{\rho_{i,j}^n}{\rho_{i,j}^n}$, $\lambda_1 \numflux{1}{i+\frac{1}{2}}{j}{\rho_{i,j}^n}{\rho_{i,j}^n}$, $\lambda_2 \numflux{2}{i}{j-\frac{1}{2}}{\rho_{i,j}^n}{\rho_{i,j}^n}$, $\lambda_2 \numflux{2}{i}{j+\frac{1}{2}}{\rho_{i,j}^n}{\rho_{i,j}^n}$ it follows
  \begin{equation}
  \begin{aligned}
        \left| \rho_{i,j}^{n+1} - \rho_{i,j}^{n} \right| \leq  &\; \lambda_1 \Bigl( L_{1,1} \left| \rho_{i,j}^n - \rho_{i-1,j}^n \right| + L_{1,2} \left| \rho_{i+1,j}^n - \rho_{i,j}^n \right| + \left| \flux{1}{i+\frac{1}{2}}{j}{\rho_{i,j}^n}- \flux{1}{i-\frac{1}{2}}{j}{\rho_{i,j}^n} \right| \Bigr) \\
        & + \lambda_2  \Bigl( L_{2,1} \left| \rho_{i,j}^n - \rho_{i,j-1}^n \right| + L_{2,2} \left| \rho_{i,j+1}^n - \rho_{i,j}^n \right| + \left| \flux{2}{i}{j+\frac{1}{2}}{\rho_{i,j}^n}- \flux{2}{i}{j-\frac{1}{2}}{\rho_{i,j}^n} \right| \Bigr).
    \end{aligned}
\end{equation}
With the same inequalities as in the proof of Thm.~\ref{thm:BV} we get
\begin{align*}
  \dx \dy \sum_{i,j\in \Z}  &\left| \rho_{i,j}^{n+1} - \rho_{i,j}^{n}\right| \\
  \leq\; & \dt \mleft( L_{1,1}+L_{1,2} \mright) \sum_{i,j\in \Z} \left| \rho_{i+1,j}^n - \rho_{i,j}^n \right| \dy+ \dt \mleft( L_{2,1}+L_{2,2} \mright) \sum_{i,j\in \Z} \left| \rho_{i,j+1}^n - \rho_{i,j}^n \right| \dx\\
  &+ \dt \dy \sum_{i,j\in \Z} \mathcal{M} |\rho_{i,j}^n| \mleft( \norm{\Rb_{i+\frac{1}{2},j}^n - \Rb_{i-\frac{1}{2},j}^n}_2 + \dx \mright) \\
  &+ \dt \dx \sum_{i,j\in \Z} \mathcal{M} |\rho_{i,j}^n| \mleft( \norm{\Rb_{i,j+\frac{1}{2}}^n - \Rb_{i,j-\frac{1}{2}}^n}_2 + \dy \mright)
.\end{align*}
Thus, with $L \defeq \max\{ L_{1,1}+L_{1,2}, L_{2,1}+L_{2,2}\}$ we conclude
\begin{align*}
  \dx \dy \sum_{i,j\in \Z}  \left| \rho_{i,j}^{n+1} - \rho_{i,j}^{n}\right|
  \leq & \dt \, L \; \text{TV}\mleft( \rho_\Delta\mleft(t^n, \cdot \mright) \mright)\\
  &+ \dt \mathcal{M}  \mleft(2 + \norm{\ddx \etab}_\infty \norm{\rhob_0}_{L^1} + \norm{\ddy \etab}_\infty \norm{\rhob_0}_{L^1}\mright) \sum_{i,j\in \Z} \left| \rho_{i,j}^n \right| \dx \dy\\
  &= \dt \, \revb{\mathcal{K}_3} \; \text{TV}\mleft( \rho_\Delta\mleft(t^n, \cdot \mright) \mright) + \dt \revb{\mathcal{K}_4}
\end{align*}
\revb{with $\mathcal{K}_3 = L$ and $\mathcal{K}_4 = \mathcal{M} \mleft(2 + \mleft( \norm{\ddx \etab}_\infty + \norm{\ddy \etab}_\infty  \mright) \norm{\rhob_0}_{L^1}\mright) \norm{\rhob_0}_{L^1}$.}
\end{proof}
For any finite time we can also prove an upper bound on the numerical solution which is necessary for the case that $\imgrho_k=[\rho_m^k,\infty)$ since otherwise the maximum principle in Thm.~\ref{eq:maxprinciple1} already holds.
\begin{lemma}[L$^\infty$-bound]\label{lemma:Linf}
  Let the Asm.~\ref{asm} hold.
  Then for a solution computed by a numerical scheme \eqref{eq:1stscheme} satisfying Def.\ \ref{def:flux} and the CFL condition \eqref{eq:CFL} there exists a constant $\mathcal{K}_5>0$ depending on $\rhob_0, \etab$ and $\fb^k$ such that
  $$ \norm{\rho_{\Delta}^k\mleft( t^n,\cdot \mright)}_\infty \leq \exp\mleft( \mathcal{K}_5 t^n \mright)  \norm{\rhob_0}_\infty  $$
for $k=1,\dots,K$.
\end{lemma}
\begin{proof}
      Using \eqref{eq:H1+H2} we determine estimates on $H_{i,j}^{1,n}$ and $H_{i,j}^{2,n}$. For this we \revc{rewrite $H_{i,j}^{1,n}$ as in the proof of Thm.~\ref{thm:BV} and use similar bounds on the difference $\flux{1}{i-\frac{1}{2}}{j}{\rho_{i,j}^n}  - \flux{1}{i+\frac{1}{2}}{j}{\rho_{i,j}^n$}} to obtain
  \begin{equation*}
    \begin{aligned}
    \left| H_{i,j}^{1,n} \right| \leq \mleft( \frac12 - a_{i-\frac{1}{2},j}^n - b_{i+\frac{1}{2},j}^n \mright) \norm{\rho_{\Delta}^k\mleft( t^n,\cdot \mright)}_{\infty}  + a_{i-\frac{1}{2},j}^n \norm{\rho_{\Delta}^k\mleft( t^n,\cdot \mright)}_{\infty} + b_{i+\frac{1}{2},j}^n \norm{\rho_{\Delta}^k\mleft( t^n,\cdot \mright)}_{\infty}\\
    + \lambda_1 \mathcal{M} \mleft| \rho_{i,j}^n \mright| \dx \mleft( 1 + \norm{\ddx \etab}_\infty \norm{\rhob_0}_{L^1} \mright)\\
    \leq \mleft( \frac12 + \dt \mathcal{M} (1 + \norm{\ddx \etab}_\infty \norm{\rhob_0}_{L^1}) \mright) \norm{\rho_{\Delta}^k\mleft( t^n,\cdot \mright)}_{\infty}.
    \end{aligned}
  \end{equation*} 
  With an analogous bound on $\left| H_{i,j}^{2,n} \right|$ we are left with
  \begin{align*}
    \norm{\rho_{\Delta}^k\mleft( t^{n+1},\cdot \mright)}_{\infty} \leq \mleft( 1 + \dt  \mathcal{K}_5 \mright) \norm{\rho_{\Delta}^k\mleft( t^n,\cdot \mright)}_{\infty}
    \leq \exp\mleft( t^{n+1} \mathcal{K}_5 \mright) \norm{\rhob_0}_{\infty}
  \end{align*}
  for $ \mathcal{K}_5 = \mathcal{M} (2 + (\norm{\ddx \etab}_\infty+\norm{\ddy \etab}_\infty) \norm{\rhob_0}_{L^1})$.
\end{proof}
The last ingredient to prove the convergence is a discrete analogue for the entropy conditions. 
Hence, we define the discrete entropy fluxes
  \begin{align*}
    \entropyfluxkn{1}{i+\frac{1}{2}}{j}{u}{w} &\defeq
    \numfluxkn{1}{i+\frac{1}{2}}{j}{u \land \kappa}{w \land \kappa} - \numfluxkn{1}{i+\frac{1}{2}}{j}{u \lor \kappa}{w \lor \kappa}\\
    \entropyfluxkn{2}{i}{j+\frac{1}{2}}{u}{w}
    &\defeq \numfluxkn{2}{i}{j+\frac{1}{2}}{u \land \kappa}{w \land \kappa} - \numfluxkn{2}{i}{j+\frac{1}{2}}{u \lor \kappa}{w \lor \kappa}
  \end{align*}
  with $a \land b \defeq \max\mleft( a,b \mright) $ and $a \lor b \defeq \min \mleft( a,b \mright)$.
  This allows us to show the following discrete entropy condition.
We note that the proof is analogue to \cite{ACT15,FSS23} and we do not go into detail here.
\begin{proposition}[Discrete entropy condition]
  Under the Asm.~\ref{asm} the numerical scheme \eqref{eq:1stscheme} with a numerical flux from Def.\ \ref{def:flux} fulfills the discrete entropy condition
  \begin{equation}\label{eq:discr entropy}
  \begin{aligned}
    \left| \rho_{i,j}^{k,n+1} - \kappa \right| &- \left| \rho_{i,j}^{k,n} - \kappa \right| 
    + \lambda_1 \mleft( \entropyfluxkn{1}{i+\frac{1}{2}}{j}{\rho_{i,j}^{k,n}}{\rho_{i+1,j}^{k,n}} - \entropyfluxkn{1}{i-\frac{1}{2}}{j}{\rho_{i-1,j}^{k,n}}{\rho_{i,j}^{k,n}} \mright)\\ 
     &+ \lambda_1 \textnormal{sgn}\mleft( \rho_{i,j}^{n+1}-\kappa \mright) \mleft( \fluxk{1}{i+\frac{1}{2}}{j}{\kappa} -\fluxk{1}{i-\frac{1}{2}}{j}{\kappa}\mright)\\
    &+ \lambda_2 \mleft(\entropyfluxkn{2}{i}{j+\frac{1}{2}}{\rho_{i,j}^{k,n}}{\rho_{i,j+1}^{k,n}} - \entropyfluxkn{2}{i}{j-\frac{1}{2}}{\rho_{i,j-1}^{k,n}}{\rho_{i,j}^{k,n}} \mright)\\
    &+ \lambda_2 \textnormal{sgn}\mleft( \rho_{i,j}^{n+1}-\kappa \mright)
    \mleft( \fluxk{2}{i}{j+\frac{1}{2}}{\kappa} - \fluxk{2}{i}{j-\frac{1}{2}}{\kappa} \mright)  \leq 0,
  \end{aligned}
\end{equation}
for $k=1,\dots,K$.
\end{proposition}
Now, we are ready to prove the main result:
\begin{proof}[Proof of Thm.~\ref{thm:mainres}]
  Let Asm.~\ref{asm} and $\rhob_0 \in \text{BV}(\R^2)$ hold and $\rhob_{\Delta \xb}$ be a sequence where $\dx$ tends to zero and consequently also $\dy$ by the fixed relation $\dx = \delta \dy$ and $\dt$ by the CFL condition \eqref{eq:CFL}. 
With the bound on the $L^1$-norm (Cor.~\ref{cor:L1}), the total variation in space (Thm.\ \ref{thm:BV}) and the bound on the time continuity (Prop.\ \ref{prop:timecont})  there exists for each $\rho^k_{\Delta \xb}$, $k=1,\ldots,K$ a uniformly convergent subsequence in $L_{loc}^1(\R^2)$ on every bounded interval $[0,T]$ \cite[Lem.~1]{San83}. The approximated convolution terms \eqref{eq:Ri} and \eqref{eq:Rj} converge for kernels fulfilling $(\etab)$ to the integral terms. Thus, with a Lax-Wendroff-type argument we can prove that
the limit of each subsequence is a weak entropy solution of \eqref{eq:system} in the sense of \ref{Def:entropysol}. Later, in Thm.~\ref{thm:uniqueness} we will prove the uniqueness of the weak entropy solution and together with Cor.~\ref{cor:L1} we obtain that the entire sequence $\rhob_{\Delta \xb}$ converges to a unique weak entropy solution in $\mleft( L_{loc}^1([0,T] \times \R^2) \mright)^K$.
\end{proof}

\section{Error estimate}\label{sec:errorestimate}
    After having established the convergence of the monotone-based numerical schemes, we will investigate the rate of convergence via error estimates. 
In addition, we provide a uniqueness result for weak entropy solutions of \eqref{eq:system}.
In \cite{aggarwal2024well}, the authors prove a Kuznetsov-type lemma as well as error estimates for a similar nonlocal system as \eqref{eq:system} in one spatial dimension. 
Although the result of the Kuznetsov-type lemma is still valid in several space \reva{dimensions} \cite[Rem.~5.3]{aggarwal2024well},
we consider a more general flux function and hence, we will discuss the necessary modifications of the proof as well as the error estimate in this section. 
We underline that we do not use dimensional splitting such that the proof of the error estimate in \cite{aggarwal2024well} cannot be directly applied.

As aforementioned, we follow \cite{aggarwal2024well} and define $Q_T=[0,T]\times \R^2$ and $\phi:Q_T^2\to \R$ by
$$ \phi(t,\xb,s,\yb)\coloneq \omega_\epsilon(x_1-y_1) \omega_\epsilon(x_2-y_2) \omega_{\epsilon_0}(t-s), $$
where $\omega_a(x)=\omega(x/a)/a,\ a>0$, and $\omega$ is a standard symmetric mollifier
with compact support on $[-1,1]$. 
In addition, we assume that $\int_{\R}\omega_a(x) dx=1$ and $\int_{\R} |\omega_a'(x)| dx=1/a$.
We note that $\phi$ is symmetric, $\nabla_\xb \phi= -\nabla_\yb \phi$ and $\phi_t=-\phi_s$. 

For $\varrhob,\ub\in L^1\cap L^\infty (Q_T;\R^K) $, $\phi\in \mathbf{C}_c^\infty(Q_T;\R^+)$ and for each $k\in \{1,\dots,K\}$, we define the entropy functionals by
\begin{align*}
    \Lambda_T^k(\varrho^k,\phi,\kappa)\coloneq & \int_{Q_T} \mleft[ \left| \varrho^k - \kappa \right| \ddt \phi + \textnormal{sgn}(\varrho^k - \kappa)\mleft( \fb^k\mleft( t, \xb, \varrho^k, \contconv{\varrhob} \mright) - \fb^k\mleft( t, \xb, \kappa, \contconv{\varrhob} \mright)
    \mright) \mydot \nabla_\xb \phi \mright] d\xb \; dt \\
                                            & - \int_{Q_T} \textnormal{sgn}(\varrho^k-\kappa) \Div \fb^k\mleft( t, \xb, \kappa, \contconv{\varrhob} \mright)   \phi d\xb \; dt \\
                                            & - \int_{\R^2} \left| \varrho^k(T,\xb)-\kappa \right| \phi(T,\xb) d\xb + \int_{\R^2} \left| \varrho_0^k(\xb)-\kappa \right| \phi(0,\xb) d\xb\revb{,} \\
    \Lambda_{\epsilon,\epsilon_0}^k(\varrho^k,u^k)\coloneq& \int_{Q_T} \Lambda_T^k\mleft( \varrho^k(\cdot,\cdot),\phi(\cdot,\cdot,s,\yb),u^k(s,\yb)\mright) d \yb\; ds\revb{,}
\end{align*}
\revb{\text{and}}
\begin{align*}
    \mathcal{U}\coloneq&\{ \ub:\, Q_T\to \R^K: \ub(t,\cdot)\in (L^\infty\cap BV)(\R^2;\R^K),\ \|u^k(t,\cdot)\|_{L^1(\R^2)}=\|u^k(0,\cdot)\|_{L^1(\R^2)}\\
    &\text{ for }t\in [0,T]\text{ and }k=1,\dots,K\}\revb{.}  
\end{align*}

\begin{lemma}[A Kuznetsov-type lemma for two-dimensional nonlocal systems of conservation laws]\label{lem:Kuznetsov}
    Let $\varrhob$ be an entropy solution and $\ub\in \mathcal{U}$, then for any finite time $T>0$ it holds
    \begin{equation}
      \| \varrhob(T,\cdot) - \ub(T,\cdot) \|_{L^1(\R^2;\R^K)} \leq \mathcal{K}_6 \left(\sum_{k=1}^{K} \mleft(\gamma(u^k,\epsilon_0)-\Lambda_{\epsilon,\epsilon_0}(u^k,\varrho^k)\mright)
      +\| \varrhob_0-\ub_0\|_{L^1(\R^2;\R^K)} +K(\epsilon+\epsilon_0)\right),
    \end{equation}
    where $\gamma(u^k,\epsilon_0)\coloneq \displaystyle{\sup_{\substack{|t_1-t_2|<\epsilon_0,\\ 0\leq t_1\leq t_2 \leq T}}}
    \|u^k(t_1,\cdot )-u^k(t_2,\cdot )\|_{L^1(\R^2)}$ and $\mathcal{K}_6$ depends on $T, \fb, \mathcal{M}, \etab, \varrhob, \ub$.
\end{lemma}
\begin{proof}
    The proof is very similar to the one of \cite[Lem.~4.1]{aggarwal2024well}.
    Hence, we focus on the most significant differences which occur due to the different flux function and two spatial dimensions.
    For any ${k\in \{1,\dots,K\}}$, we add the entropy functionals and use the symmetry of $\phi$ to get 
        $$\Lambda_{\epsilon,\epsilon_0}^k(\varrho^k,u^k)+\Lambda_{\epsilon,\epsilon_0}^k(u^k,\varrho^k)=I_0^k-I_T^k+I_{\nabla_\xb \phi}^k+I_{\phi}^k$$
    with
    \begin{align*}
        I_0^k\coloneq&\int_{Q_T} \int_{\R^2}\left( \left| \varrho_0^k(\xb)-u^k(s,\yb) \right| + \left| u_0^k(\yb)-\varrho^k(s,\xb) \right| \right)\phi(t,\xb,0,y)d\yb \; d\xb \; dt\\
        I_T^k\coloneq&\int_{Q_T} \int_{\R^2}\left( \left| \varrho^k(T,\xb)-u^k(s,\yb) \right| + \left| u^k(T,\yb)-\varrho^k(s,\xb) \right| \right)\phi(t,\xb,T,y)d\yb \; d\xb \; dt\\
        I_{\nabla_\xb \phi}^k\coloneq& \int_{Q_T^2} \textnormal{sgn}(\varrho^k(t,\xb) - u^k(s,\yb)) \mleft( \fb^k\mleft( t, \xb, \varrho^k(t,\xb), \contconv{\varrhob} \mright) 
                                - \fb^k\mleft( t, \xb, u^k(s,\yb), \contconv{\varrhob} \mright) \mright. \\
                               &\mleft. + \fb^k\mleft( s, \yb, u^k(s,\yb), \contconarg{\ub}{s}{\yb} \mright) 
                                -\fb^k\mleft( s, \yb, \varrho^k(t,\xb),\contconarg{\ub}{s}{\yb} \mright) \mright) \mydot \nabla_\xb \phi(t,\xb,s,\yb) dt \; d\xb \; ds \; d\yb\\
        I_{\phi}^k\coloneq&  - \int_{Q_T^2} \textnormal{sgn}(\varrho^k(t,\xb)-u^k(s,\yb))\mleft( \Div \fb^k\mleft( t, \xb, u^k(s,\yb), \contconv{\varrhob}\mright)\mright.\\
                          & \qquad  \mleft.- \textnormal{div}_\yb \, \fb^k\mleft( s, \yb, \varrho^k(t,\xb), \contconarg{\ub}{s}{\yb} \mright)  \mright) \phi(t,\xb,s,\yb) dt \; d\xb \; ds \; d\yb.
    \end{align*}
    Since $\varrhob$ is an entropy solution of \eqref{eq:system}, $\Lambda_{\epsilon,\epsilon_0}^k(\varrho^k,u^k)\geq 0$ holds, such that
    \begin{equation}\label{eq:Kuznetsov}
        I_T^k \leq -\Lambda_{\epsilon,\epsilon_0}^k(u^k,\varrho^k)+I_0^k+I_{\nabla_\xb \phi}^k+I_{\phi}^k .
    \end{equation}
    The terms $I_T^k$ and $I_0^k$ can be estimated following \cite{ghoshal2021godunov}, see Claim 1 and Claim 2 in the proof of Lem.~4.1.
    Hence, we obtain similar estimates as in \cite{aggarwal2024well}. 
    The most significant differences to the proof of \cite[Lem.~4.1]{aggarwal2024well} \reva{appear} for the estimation of $I_{\nabla_\xb \phi}^k+I_{\phi}^k$. 
    Integration by parts yields
    \begin{align*}
        I_{\nabla_\xb \phi}^k+I_{\phi}^k=& - \int_{Q_T^2} \textnormal{sgn}(\varrho^k(t,\xb)-u^k(s,\yb))
        \mleft( \Div \mleft[\fb^k\mleft( t, \xb, \varrho^k(t,\xb), \contconv{\varrhob}\mright)\mright.\mright.\\
        &\mleft.\mleft.-\fb^k\mleft( s, \yb, \varrho^k(t,\xb), \contconarg{\ub}{s}{\yb}\mright)\mright]
        - \textnormal{div}_\yb \, \fb^k\mleft( s, \yb, \varrho^k(t,\xb), \contconarg{\ub}{s}{\yb} \mright)  \mright) \phi(t,\xb,s,\yb) dt \; d\xb \; ds \; d\yb.
    \end{align*}
    In the following the partial derivative of $\varrho$ needs to be understood in the sense of measures. We can compute 
    \begin{align*}
        &I_{\nabla_\xb \phi}^k+I_{\phi}^k\\
        =& - \int_{Q_T^2} \textnormal{sgn}(\varrho^k(t,\xb)-u^k(s,\yb))
        \bigg(\sum_{\ell=1}^2 (\partial_{x_\ell} f_\ell^k)(t,\xb,\varrho^k(t,\xb),\contconv{\varrhob})-(\partial_{x_\ell} f_\ell^k)(s,\yb,\varrho^k(t,\xb),\contconarg{\ub}{s}{\yb})\\
        &+(\partial_{x_\ell} \varrho(t,\xb))\mleft((\partial_{\varrho} f_\ell^k)(t,\xb,\varrho^k(t,\xb),\contconv{\varrhob})-(\partial_{\varrho} f_\ell^k)(s,\yb,\varrho^k(t,\xb),\contconarg{\ub}{s}{\yb})\mright)\\
        &+(\nabla_{R} f_\ell^k)(t,\xb,\varrho^k(t,\xb),\contconv{\varrhob})\;\mydot\; \partial_{x_\ell} \contconv{\varrhob}
        -(\nabla_{R}  f_\ell^k)(s,\yb,\varrho^k(t,\xb),\contconarg{\ub}{s}{\yb})\;\mydot\; \partial_{y_\ell} \contconarg{\ub}{s}{\yb})\bigg) \\
        &\quad \phi(t,\xb,s,\yb) dt \; d\xb \; ds \; d\yb.
    \end{align*}
    For the last term we add and subtract $(\nabla_{R}  f_\ell^k)(s,\yb,\varrho^k(t,\xb),\contconarg{\ub}{s}{\yb})\;\mydot\; \partial_{x_\ell} \contconv{\varrhob}$,
    such that we can estimate (by using the Asm.~\ref{asm})
    \begin{align}\label{eq:UkUxk}
        &I_{\nabla_\xb \phi}^k+I_{\phi}^k\\
        \leq &\; \mathcal{C}_1\int_{Q_T^2} \bigg(\sum_{\ell=1}^2 \mleft( (1+|\partial_{x_\ell} \contconv{\varrhob}|)|\varrho^k(t,\xb)|+ |\partial_{x_\ell}\varrho^k(t,\xb)| \mright)\\
        & \qquad \qquad \qquad \mleft(  |t-s|+\|\xb-\yb\|_2+\| \contconv{\varrhob}- \contconarg{\ub}{s}{\yb}\|_2 \mright)\\
        &\qquad \qquad \qquad + |\varrho^k(t,\xb)| \| \partial_{x_\ell} \contconv{\varrhob}- \partial_{y_\ell} \contconarg{\ub}{s}{\yb}\|_2
        \bigg)\phi(t,\xb,s,\yb) dt \; d\xb \; ds \; d\yb.
    \end{align}
    Here, $\mathcal{C}_1$ depends on $\mathcal{M}$ and $\fb$. 
    Further, we can use $|\partial_{x_\ell} \contconv{\varrhob}|\leq \|\partial_{x_\ell}\eta\|_{L^\infty(\R^2)}\displaystyle{\sup_{\tau \in [0,T]}}\|\varrhob(\tau,\cdot)\|_{L^1(\R^2;\R^K)}$ in the first line.
    The remaining terms can be estimated by using either the BV-norm or $L^1$-norm of $\varrho^k$. 
    In particular, for $\ell\in \{1,2\}$
    \begin{align*}
    \int_{Q_T^2} &\mleft( |\varrho^k(t,\xb)|+ |\partial_{x_\ell}\varrho^k(t,\xb)| \mright) 
        \mleft(  |t-s|+\|\xb-\yb\|_2\mright)   \phi(t,\xb,s,\yb) dt \; d\xb \; ds \; d\yb\\
        & \leq (\|\varrho^k\|_{L^1(Q_T)}+|\varrho^k|_{L^\infty([0,T];BV(\R^2))}) (\epsilon+\epsilon_0).
    \end{align*}
    The differences in the nonlocal terms can be estimated as in \cite{aggarwal2024well}, such that we obtain\footnote{In particular, we refer for the remaining terms in the first two lines in \eqref{eq:UkUxk} to the estimation of $I_{\mathcal{U}^k}$ on \cite[pp.\ 3369 -- 3371]{aggarwal2024well}. The steps are completely analogue, using one time the BV-norm and one time the $L^1$-norm of $\varrho^k$. The last line of \eqref{eq:UkUxk} is as well analogue to $I_{\mathcal{U}_x^k}$ in \cite{aggarwal2024well}.} 
    \begin{align*}
        I_{\nabla_\xb \phi}^k+I_{\phi}^k\leq \mathcal{C}_2 (\epsilon+\epsilon_0) +\mathcal{C}_3 \int_0^T\| \varrhob(t,\cdot)-\ub(t,\cdot)\|_{L^1(\R;\R^K)} dt
    \end{align*}
    for constants $\mathcal{C}_2,\mathcal{C}_3>0$ depending on $T, \fb, \etab, \varrhob, \ub$.
    The lemma now follows from substituting all the estimates into \eqref{eq:Kuznetsov} and applying Gronwall's lemma.
\end{proof}
Similar to \cite{aggarwal2024well} the uniqueness of weak entropy solutions follow from Lem.~\ref{lem:Kuznetsov}.
The proof is the same as for \cite[Thm.~4.2]{aggarwal2024well}:
\begin{theorem}\label{thm:uniqueness}
 Let $\varrhob,\ub$ be two entropy solutions of the system \eqref{eq:system} with initial data
    $\varrhob_0,\ \ub_0$ respectively, both fulfilling assumption $(\rhob_0)$.
    For any finite time $T>0$ the following holds
    \begin{align*}   
           \|\varrhob(T,\cdot)-\ub(T,\cdot)\|_{L^1(\R^2;\R^K)}\leq \mathcal{K}_6 \|\varrhob_0-\ub_0\|_{L^1(\R^2;\R^K)}
    \end{align*}
    \revb{with $\mathcal{K}_6$ as in Lem.~\ref{lem:Kuznetsov}.}
    In particular, a weak entropy solution to \eqref{eq:system} is unique.   
\end{theorem}

  The Kuznetsov-type lemma is not only used to obtain uniqueness of the weak entropy solution but also for an error estimate for the numerical approximation to this solution. We will obtain a convergence rate of 0.5. In the local case, it is already proven \cite{Sab97} that this is the optimal rate that can be achieved for nonlinear (local) fluxes.  
For this we first consider an estimate on the relative entropy functional $\Lambda_{\epsilon,\epsilon_0}^k(\rho_\Delta^k,\varrho^k)$ that can be inserted into the estimate of Lem.~\ref{lem:Kuznetsov}.

\begin{lemma}
    For a numerical solution $\rhob_\Delta$ computed by a scheme \eqref{eq:1stscheme} with a flux function fulfilling Def.~\ref{def:flux} using a grid that satisfies the fixed relation $\dx = \delta \dy$ and the CFL condition \eqref{eq:CFL} the relative entropy functional satisfies
    \begin{align*}
        -\Lambda_{\epsilon,\epsilon_0}^k(\rho_\Delta^k,\varrho^k) \leq  \mathcal{K}_7 \mleft( \frac{\dx}{\varepsilon} + \frac{\dt}{\varepsilon_0} + \dt + \dx \mright) 
    \end{align*}
    for every $k=1,\ldots,K$ and a constant $\mathcal{K}_7$ depending on $\rhob_0, \etab, \fb, \delta$, and the Lipschitz constants in \mbox{Def.~\ref{def:flux}}.
\end{lemma}
\begin{proof}
    \revb{We present a sketch of the proof here, and provide details in Appendix~\ref{sec:appendix}. 
  The proof is again based on the one of \cite[Lem.~5.1]{aggarwal2024well}. We introduce the notation of time-space cells $C_{i,j}^n = [t^n,t^{n+1}) \times C_{i,j}$. 
    Further, we define for $n \in \{0,\ldots,N_T \}, i,j \in \Z,$ ${k=1,\ldots,K}, (t,\xb) \in \overline{Q}_T$
    \begin{align*}
        p_{i,j}^{k,n}(\kappa, R) \coloneq \textnormal{sgn}\mleft( \rho_{i,j}^{k,n} - \kappa \mright) \mleft( \fb^k\mleft( t,\xb,\rho_{i,j}^{k,n},R \mright) - \fb^k\mleft( t,\xb,\kappa,R \mright)  \mright)
    .\end{align*}
    We consider $-\Lambda_{\epsilon,\epsilon_0}^k(\rho_\Delta^k,\varrho^k)$ for the discrete $\rhob_\Delta$, perform summation by parts and add a zero to obtain
    \begin{align*}
        -\Lambda_{\epsilon,\epsilon_0}^k&(\rho_\Delta^k,\varrho^k) =\int_{Q_T} \sum_{i,j \in \Z} \sum_{n=0}^{N_T-1} \mleft[ \left| \rho_{i,j}^{k,n+1} - \varrho^k(s,\yb) \right| - \left| \rho_{i,j}^{k,n} - \varrho^k(s,\yb) \right| \mright] \int_{C_{i,j}}\phi\mleft( s,\yb,t^{n+1},\xb \mright) d\xb \, ds \, d\yb\\
        &- \int_{Q_T} \sum_{i,j \in \Z} \sum_{n=0}^{N_T-1} \int_{C_{i,j}^n} p_{i,j}^{k,n}(\varrho^k(s,\yb), \Rb_{i+\frac{1}{2},j}^n)  \mydot \nabla_\xb \phi dt \, d\xb \, ds \, d\yb  \\
        &- \int_{Q_T} \sum_{i,j \in \Z} \sum_{n=0}^{N_T-1} \int_{C_{i,j}^n} \mleft[ p_{i,j}^{k,n}(\varrho^k(s,\yb), \contconv{\rhob_\Delta}) - p_{i,j}^{k,n}(\varrho^k(s,\yb), \Rb_{i+\frac{1}{2},j}^n)\mright] \mydot \nabla_\xb \phi dt \, d\xb \, ds \, d\yb  \\
        & + \int_{Q_T} \sum_{i,j \in \Z} \sum_{n=0}^{N_T-1} \int_{C_{i,j}^n}\textnormal{sgn}\mleft( \rho_{i,j}^{k,n} - \varrho^k(s,\yb) \mright) \Div \fb^k\mleft( t, \xb, \varrho^k(s,\yb), \contconv{\rhob_\Delta} \mright)   \phi \, dt \, d\xb \, ds \, d\yb \\
        & \coloneq \biglambda_1 + \biglambda_2' + \bigeps_{2} + \biglambda_3.
    \end{align*}
    We note that we adapted here the notation of \cite{aggarwal2024well} as far as possible\footnote{The corresponding terms to $\biglambda_1, \biglambda_2', \biglambda_3$ in \cite{aggarwal2024well} are denoted by $\lambda_1, \lambda_2', \lambda_3$.} to increase the clarity when referring to steps in \cite{aggarwal2024well} in the appendix.
    The spatial derivatives in the terms above can be expressed as a sum of terms that involve either the first or the second component of the flux function.
    
    The main goal of the proof is to show 
    $$ -\Lambda_{\epsilon,\epsilon_0}^k(\rho_\Delta^k,\varrho^k)= \mathcal{O}\mleft(\frac{\dx}{\varepsilon} + \frac{\dt}{\varepsilon_0} + \dx + \dt  \mright).$$
    In comparison to local conservation laws, the main differences in the proof arise from the additionally appearing nonlocal terms. 
    Those need to be estimated using similar techniques as before.

    In a first step we want to consider $\bigeps_{2}$ and  $\biglambda_3$ treating each component \revd{separately}, while the other terms follow analogously.
    \begin{itemize}
      \item An estimate on $\bigeps_{2}^1$ can be provided using the differences of fluxes as usual, i.e.\ mean value theorem, assumption $(\mathbf{f})$, estimates for the differences in the nonlocal terms, and exploiting the fixed relation $\delta = \frac{\dx}{\dy}$.
      \item A direct estimate on $\biglambda_3^1$ cannot be established such that it is rewritten into $\biglambda_3^1 \eqcolon \; \biglambda_3'^1 + \bigeps_{31} + \bigeps_{32}$. These terms include differences either in (derivatives of) the flux functions or the test functions.   
    Note that due to the two spatial dimensions additional terms, compared to the proof in \cite{aggarwal2024well}, arise. 
    Using the properties of the test function, e.g. $\left| \omega_\eps \right|_{BV(\R)} = \frac{1}{\varepsilon}$, and similar arguments as before, $\bigeps_{31}$ and $\bigeps_{32}$ can be shown to be $\mathcal{O}\mleft( \frac{\dx}{\varepsilon} \mright)$.
    \end{itemize}
    Thus, we have proven 
    \begin{align*}
        -\Lambda_{\epsilon,\epsilon_0}^k(\rho_\Delta^k,\varrho^k) = \biglambda_1 + \biglambda_2' +  \biglambda_3' + \mathcal{O}\mleft( \frac{\dx}{\varepsilon} \mright) 
    .\end{align*}
    Next, we follow the steps in \cite{aggarwal2024well} and exploit the discrete entropy inequality for $\biglambda_1$ to split it up\footnote{The terms $A_2, A_3$ in \cite{aggarwal2024well} are here again expressed as two terms each, according to the two directions.} into $A_2^1+A_2^2+A_3^1+A_3^2$. We proceed to show $A_2^1 + \biglambda_2'^1= \mathcal{O}\mleft( \frac{\dt}{\varepsilon_0} + \frac{\dx}{\varepsilon} + \dt^2 + \dx^2  \mright)$.
    \begin{itemize}
      \item With the fundamental theorem of calculus and other standard arguments
      we obtain various differences that need to be bounded. Here, a difference in $p_{i,j}^{k,n}$ introduces additional terms, compared to \cite{aggarwal2024well}, due to the general form of the flux function $\fb$. Therefore, the subsequent estimates must be adapted. 
      Because some terms which are integrated still depend on time and space, 
      we obtain not only continuous norms of the test functions, as in \cite{aggarwal2024well}, but also discrete ones.
      \item Furthermore, we add some zeros that have to be chosen carefully in the two-dimensional case, e.g.\ including an integration over $x_2$ while fixing $x_1$.
    \end{itemize}
    Then, $A_3^1 + \biglambda_3'^1= \mathcal{O}\mleft(\frac{\dx}{\varepsilon} + \frac{\dt}{\varepsilon_0} + \dx + \dt  \mright)$ is shown using similar ideas as before. Consequently, with analogue estimates for the \mbox{$x_2$-directions}, we have achieve the result that needed to be established. 
    For further details on the computations, we refer to Appendix \ref{sec:appendix}.}
\end{proof}

Finally, the result on the relative entropy functional can be utilized to further estimate the claim in the \reva{Kuznetsov-type} Lemma \ref{lem:Kuznetsov}. Following the arguments in \cite{aggarwal2024well}, we obtain an estimate on the convergence rate of the numerical approximations. \reva{This constitutes the final main result of this work.}

\begin{lemma}[Error estimate]\label{lem:errorestimate}
    Let $\varrhob$ be a weak entropy solution and let $\rhob_{\Delta}$ be a numerical approximation computed by a scheme \eqref{eq:1stscheme} with a flux function fulfilling Def.\ \ref{def:flux} using a grid that satisfies the fixed relation $\dx = \delta \dy$ and the CFL condition \eqref{eq:CFL}. For any finite time $T>0$, the convergence rate can be estimated by
    \begin{align*}
        \| \varrhob(T,\cdot ) - \rhob_{\Delta}(T,\cdot )\|_{L^1(\R^2;\R^K)} = \mathcal{O}(\sqrt{\dt})
    .\end{align*}
\end{lemma}

\section{Numerical examples}\label{sec:experiments}
    In this section, we validate the theoretical results with numerical experiments. In particular, we consider a model that is reversible in time \cite{CG25}. This provides us with an exact solution given by the initial data that can be used for a convergence study. We will consider two settings, one examining the rate for discontinuous solutions with respect to the error estimate in Section \ref{sec:errorestimate} and one with smooth initial data showing the first order convergence of the applied numerical scheme. Moreover, we deal with a multi-population model, which is nonlinear in $\rho$. We compare the solutions computed with different numerical schemes as well as the associated errors to a reference solution.
    
   \revb{
   Both examples fall into the model class which has a multiplicative form as discussed in Rem.~\ref{rem:multflux}.
   In fact, most of the models considered in the literature, e.g. \cite{CR19,BGIV20,goatin2025pedestrians,GR24,CGL12,colombo2018nonlocal,aggarwal2016crowd,gottlich2014modeling,GGZ25,CG25,keimer2018multi}, belong to this class such that we have chosen the above test cases as representative numerical examples.
   Nevertheless, we note that there exist models, like the material flow models in \cite{RWGG20} and \cite{colombo2015nonlocal}, that do not fit into this special framework and require the general form of the flux.
   The numerical schemes in \cite{RWGG20} and \cite{colombo2015nonlocal} fit into our framework. For these schemes convergence has already been proved in the cited references, while the corresponding error estimates have not been established but are now guaranteed. Our framework also yields both convergence and error estimates for alternative schemes.
   }
\subsection{Encryption-decryption}
We consider the time reversible model proposed in \cite[Section 3.4]{CG25}
\begin{equation}\label{eq:encr-decr}
    \begin{cases}
        \begin{aligned}
            &\ddt \rho + \Div [\rho \; \nub\left( t, \xb, \nabla (\tilde \eta * \rho) \right)] = 0&&, (t,\xb) \in \R^+ \times \R^2\\
            &\rho(0,\xb) = \rho_0(\xb)&&, \xb \in  \R^2. 
        \end{aligned}
    \end{cases}
  \end{equation}
with 
$$\nub\left( t, \xb, \nabla \tilde \eta *  \rho \right) =   \begin{bmatrix}
0 & -1 \\
1 & 0
\end{bmatrix}  \frac{\nabla \left(  \tilde\eta * \rho \right)}{\sqrt{1+ \left\lVert \nabla \left( \tilde\eta * \rho \right) \right\rVert _2^2}}$$
for a kernel $\tilde{\eta} \in (C^3 \cap W^{3,1} \cap W^{3,\infty})(\R^2;\R)$.
This scalar conservation law fits into the setting of \eqref{eq:system} and into \eqref{eq:conv} by $\boldsymbol{\eta} = \left( \ddx \widetilde{\eta}, \ddy \widetilde{\eta} \right)^T$.
The above equation \eqref{eq:encr-decr} is reversible in time \cite[Thm. 2.2]{CG25} and thus, can be used for encrypting and decrypting data.
We consider the $L^1$-errors of the decrypted solution to the initial data and with that the convergence rates for a discontinuous and a smooth solution. 

Note that for this model the Godunov-type flux corresponds to the Upwind-type flux due to the linearity in $\rho$ of the reduced flux function. 
Moreover, the less diffusive Lax-Friedrichs-type flux that we proposed in \eqref{eq:LxFnew} results in the Upwind-type scheme too if we set the viscosity coefficient $\alpha=1$. 
Thus, we solely compare the results for the Lax-Friedrichs-type flux from \cite{ACG15} with $\alpha=1$, cf.\ \eqref{eq:LxF}, and the Upwind-type flux.

We employ periodic boundary conditions to examine the problem on a bounded domain $[-L,L]^2$ for $L \in \R^+$ while retaining all information to be reversed.
We choose $\dx = \dy = 2L/N$ and the time step size $\dt$ is set by the CFL condition \eqref{eq:CFL} with Lipschitz constants given by one, i.e.\ $\dt = 0.25\dx$.

\paragraph{Non-smooth initial data}
We consider a first numerical example to observe the convergence rates for initial data obtained via midpoint-rule approximation of the function
\begin{align}\label{eq:nonsmoothinitial}
    \rho_0\mleft(x_1,x_2  \mright) = 1 + \mleft( 4 \sin^2\left(x_1 \right) + 3\sin^2\left( x_2 \right) \mright) \raisebox{0.6ex}{\scalebox{1.2}{$\chi$}}_{[0,3]}\mleft( \left\lVert \xb \right\rVert _2 \mright)    
\end{align}
on the domain $\xb \in [-6,6]^2$. This example is similar to the one presented in \cite[Sec.~3.4.2]{CG25}. 
Here, we use a slightly different kernel, i.e.
$$\tilde\eta_\ell\mleft( \xb \mright) = \cos^\revd{5}\mleft( \frac{\pi}{2\ell^2} \left\lVert \xb \right\rVert _2^2 \mright) \raisebox{0.6ex}{\scalebox{1.2}{$\chi$}}_{B_\ell(0)}\mleft( \left\lVert \xb \right\rVert _2 \mright)\text{ with } \ell=2,$$
where $B_\ell(\mathbf{z}) \coloneq \{\xb \in \R^2: |\xb - \mathbf{z}| < \ell \}$.

The encrypted data is decrypted at time $t=0.75$, yielding an approximate numerical reconstruction of the initial data. Figure \ref{fig:nonsmooth} illustrates the corresponding densities obtained using the Upwind-type numerical flux function on a grid with $N=6400$ cells in each direction.
The left plot presents the initial data, while the middle plot depicts the solution of the encryption at $t=0.75$. This solution is subsequently utilized for decryption, with the resulting density illustrated in the right-hand plot.
The errors are given by the discrete $L^1$-differences between the initial data and the decrypted density at $t=0$. These are summarized in Figure \ref{tab:nonsmooth} using the Lax-Friedrichs and the Upwind-type numerical flux functions for different grid sizes. Moreover, on the left side in Figure \ref{tab:nonsmooth} the errors in a logarithmic plot are depicted for the Upwind-type flux. 
For both flux functions the convergence rates are close to the \textit{worst case} rate of 0.5 proven by the error estimate in Lem.~\ref{lem:errorestimate}. Moreover, as expected, the total errors using the Upwind-type flux are smaller than those of the Lax-Friedrichs-type flux from \cite{ACG15}.

\begin{figure}[t]
    \centering
    \begin{subfigure}[b]{0.325\textwidth}
        \centering
        \includegraphics[width=\textwidth]{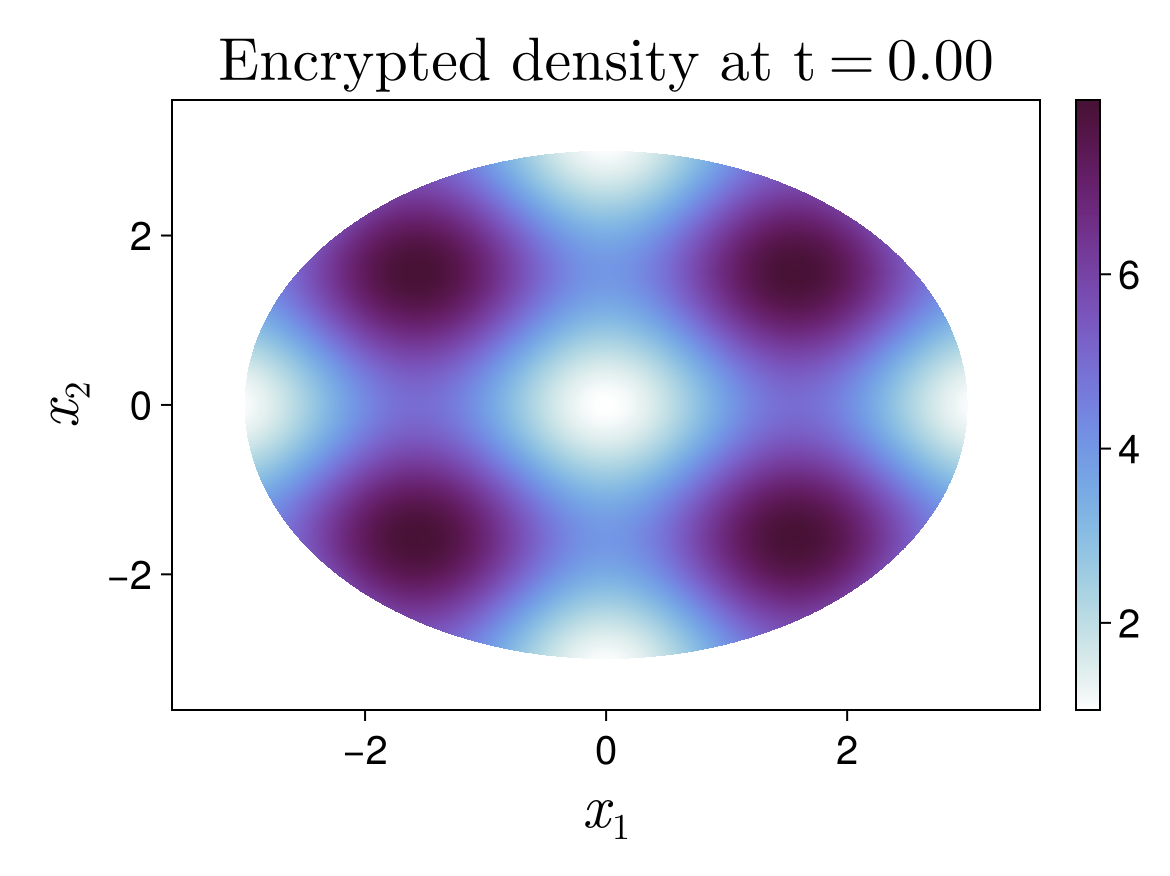}
    \end{subfigure}
    \hfill
    \begin{subfigure}[b]{0.325\textwidth}
        \centering
        \includegraphics[width=\textwidth]{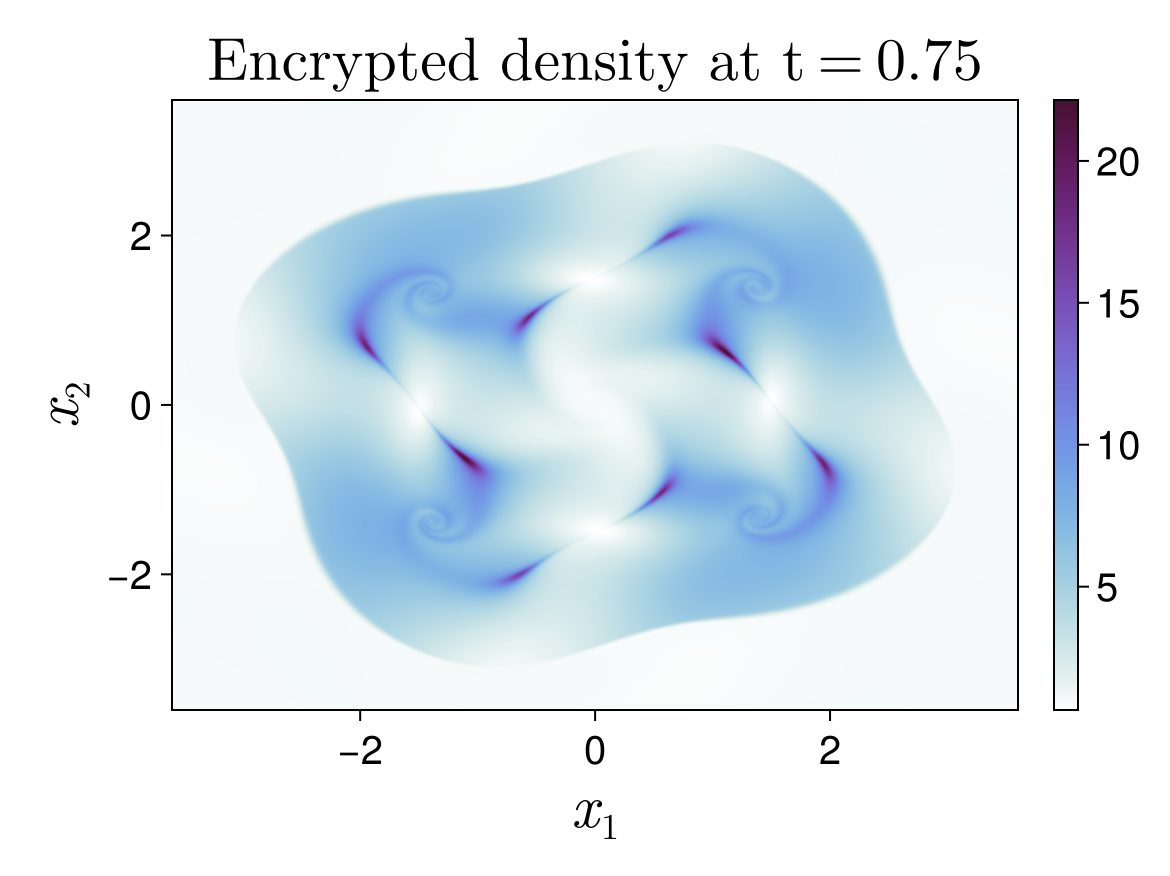}
    \end{subfigure}
    \hfill
    \begin{subfigure}[b]{0.325\textwidth}
        \centering
        \includegraphics[width=\textwidth]{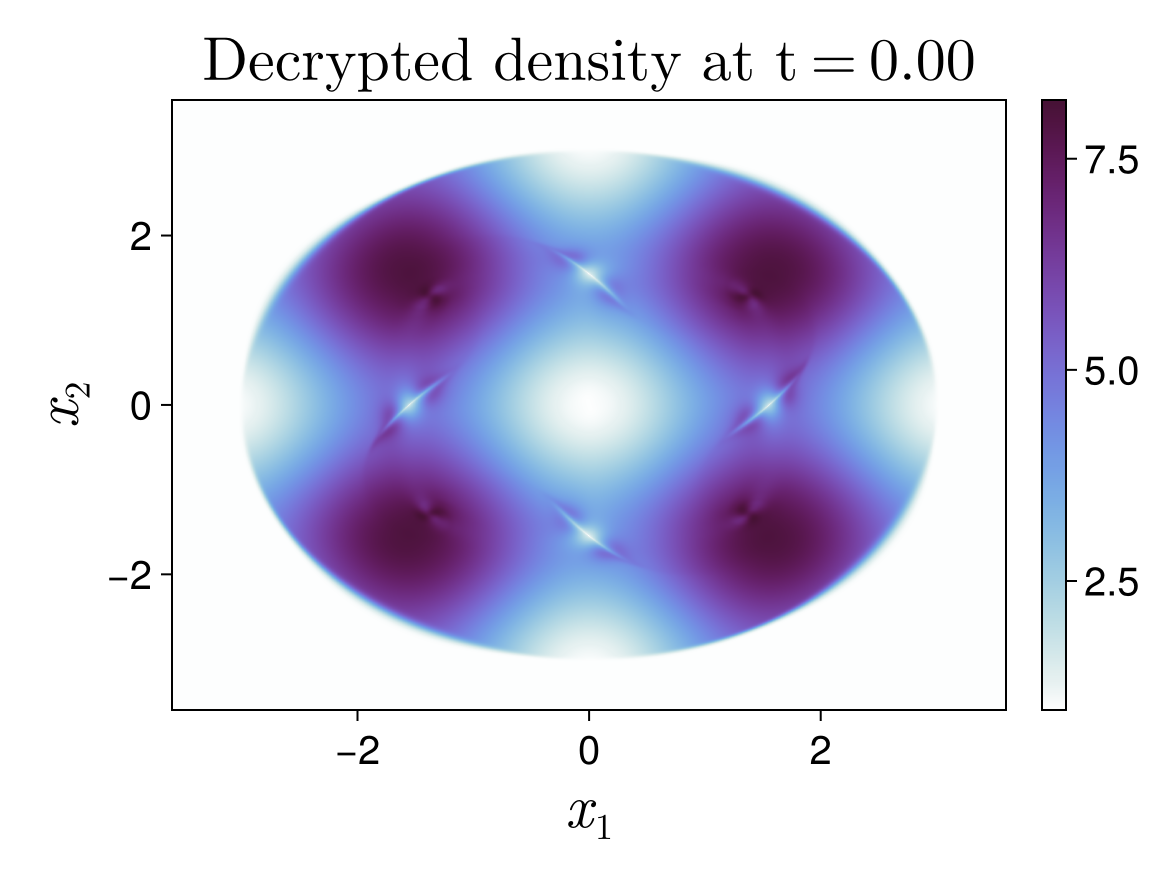}
    \end{subfigure}

    \caption{Discontinuous initial data (left) being encrypted until $t=0.75$ (middle) and decrypted (right) each plotted in the domain $[-4,4]^2$. Here an Upwind-type scheme was used with $N=6400$ cells in each direction.}
    \label{fig:nonsmooth}
\end{figure}

\begin{figure}[ht]
\centering
\begin{minipage}{0.45\textwidth}
  \centering
  \includegraphics[width=\linewidth]{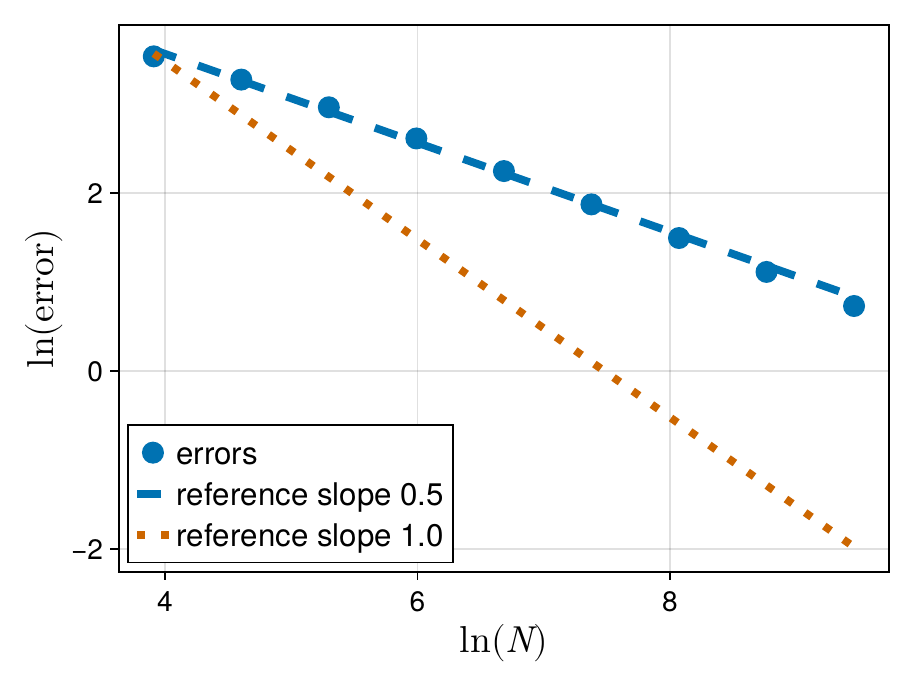}
\end{minipage}
\hfill
\begin{minipage}{0.45\textwidth}
  \centering
  \begin{tabular}{c|cc|cc}
    &\multicolumn{2}{c|}{LxF\cite{ACG15}}&\multicolumn{2}{c}{Upwind}\\
    $N$ & error & c.r. & error & c.r. \\
    \hline
    50 & 47.8 & - & 34.2 & - \\
    100 & 35.6 & 0.425 & 26.3 & 0.377 \\
    200 & 26.0 & 0.451 & 19.3 & 0.449 \\
    400 & 18.6 & 0.48 & 13.6 & 0.505 \\
    800 & 13.2 & 0.502 & 9.43 & 0.529 \\
    1600 & 9.22 & 0.514 & 6.49 & 0.540 \\
    3200 & 6.42 & 0.521 & 4.44 & 0.546 \\
    6400 & 4.45 & 0.528 & 3.04 & 0.546 \\
    12800 & 3.08 & 0.534 & 2.07 & 0.552
\end{tabular}
\end{minipage}
\caption{Convergence analysis for discontinuous initial data \eqref{eq:nonsmoothinitial}. The left side displays the decrease of the error employing Upwind-type flux on different grids with logarithmic axes. On the right, a table summarizes the errors and convergence rates for different grid sizes and using both the Lax-Friedrichs-type \cite{ACG15} and Upwind-type fluxes.}
\label{tab:nonsmooth}
\end{figure}

\paragraph{Smooth initial data}
The second numerical test deals with a convergence study demonstrating the first order accuracy of a scheme \eqref{eq:1stscheme} for smooth initial data. We consider $\xb \in [-1,1]^2$ with initial data obtained from the function
\begin{align}\label{eq:smoothinitial}
    \rho_0\mleft(x_1,x_2  \mright) = \reva{0.5 \sin\mleft(\pi x_1 + \frac{\pi}{3} \mright) \sin\mleft(\pi x_2 + \frac{\pi}{3} \mright)+0.5}
\end{align}
and the kernel is set to
$$\tilde\eta_\ell\mleft( \xb \mright) = 5 \cos^\revd{5}\mleft( \frac{\pi}{2\ell^2} \left\lVert \xb \right\rVert _2^2 \mright)\raisebox{0.6ex}{\scalebox{1.2}{$\chi$}}_{B_\ell(0)}\left( \left\lVert \xb \right\rVert _2 \right)\text{ with } \ell=0.8.$$
Figure \ref{fig:smooth} illustrates on the left side the encrypted density at time \reva{$t=0.2$} obtained using the Upwind-type numerical flux function on a grid with $N=6400$ cells for each direction. The errors and convergence rates when decrypting this density are given in the table on the right of Figure \ref{fig:smooth} for both numerical flux functions, the Lax-Friedrichs \cite{ACG15} and the Upwind-type flux, for different grid sizes. We again observe smaller errors for the Upwind scheme. Nonetheless, the convergence rates for this example with smooth initial data suggest the expected order of one.

\begin{figure}[ht]
\centering
\begin{minipage}{0.5\textwidth}
  \centering
  \includegraphics[width=\linewidth]{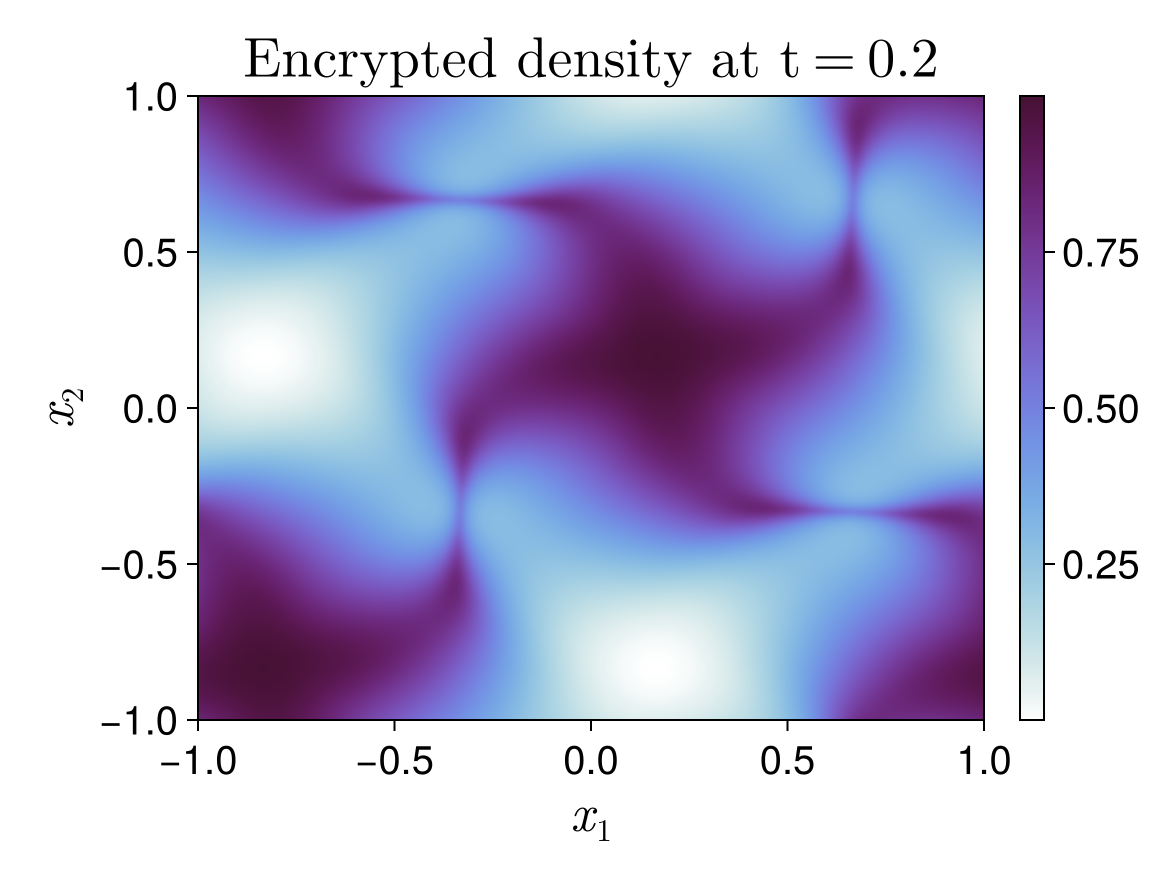}
\end{minipage}
\hfill
\begin{minipage}{0.45\textwidth}
  \centering
  \reva{
  \begin{tabular}{c|cc|cc}
    &\multicolumn{2}{c|}{LxF\cite{ACG15}}&\multicolumn{2}{c}{Upwind}\\
    $N$ & error & c.r. & error & c.r. \\
    \hline
    50 & 1.55e-1 & - & 7.66e-2 & - \\
    100 & 9.19e-2 & 0.749 & 4.36e-2 & 0.812 \\
    200 & 5.3e-2 & 0.796 & 2.36e-2 & 0.889 \\
    400 & 2.96e-2 & 0.839 & 1.23e-2 & 0.935 \\
    800 & 1.61e-2 & 0.879 & 6.33e-3 & 0.962 \\
    1600 & 8.53e-3 & 0.915 & 3.21e-3 & 0.979 \\
    3200 & 4.43e-3 & 0.945 & 1.62e-3 & 0.989 \\
    6400 & 2.27e-3 & 0.968 & 8.12e-4 & 0.994 \\
    12800 & 1.15e-3 & 0.982 & 4.07e-4 & 0.997 \\
\end{tabular}}
\end{minipage}
\caption{\reva{\revb{Encrypted} density at \reva{$t=0.2$} of the smooth \revb{initial} function \eqref{eq:smoothinitial} using an Upwind-type flux for $N=6400$ cells in each direction (left) and a table with errors and convergence rates for different grid sizes using the Lax-Friedrichs and Upwind-type numerical flux function.}}
\label{fig:smooth}
\end{figure}

\subsection{Crowd movements}
The system \eqref{eq:system} can describe crowd movements for two populations on bounded domains $\Omega \subset \R^2$, defined by the presence of walls and obstacle, for instance. Following \cite{BGIV20,GR24} we include an additional stationary density $\rho^3 = R_c \chi_{\Omega^c}$, with sufficiently large $R_c>0$, representing the domain boundaries as a high and constant value in $\R^2 \setminus \Omega$. Hence, a problem on the bounded domain $\Omega$ with boundary conditions $\rhob(t,\xb)=0$ on $\partial \Omega$ can be regarded on $\R^2$ without posing the boundary conditions directly.
We specifically consider, similar to \cite{ACG15,BGIV20,goatin2025pedestrians,GR24,CGL12},
\begin{align}\label{eq:pedmodel}
    \begin{cases*}
        \ddt \rho^1 + \Div \Bigl[ \rho^1 \, v_{\max} \; (1-\rho^1)  \mleft( \mathbf{w}^1(\xb) - \beta \frac{  \nabla \tilde \eta_{\ell} * (\rho^2 + \rho^3) }{\sqrt{1+ ||\nabla \tilde \eta_{\ell} * (\rho^2 + \rho^3) ||^2}} \mright) \Bigr] = 0,\\
        \ddt \rho^2 + \Div \Bigl[ \rho^2 \, v_{\max} \; (1-\rho^2) \mleft( \mathbf{w}^2(\xb) - \beta \frac{  \nabla \tilde \eta_{\ell} * (\rho^1 + \rho^3) }{\sqrt{1+ ||\nabla \tilde \eta_{\ell} * (\rho^1 + \rho^3) ||^2}} \mright) \Bigr] = 0.
    \end{cases*}
\end{align}
Note that this flux ensures non-negative solutions bounded above by one, cf. Thm.~\ref{eq:maxprinciple1}.

To guarantee well-posedness and to ensure that high densities in $\rhob$ are prevented from leaving the domain by sufficiently strong repulsive effects, we assume that the domain $\Omega$ satisfies assumptions $(\Omega.1)$ and $(\Omega.2)$ in \cite{GR24}, and that the support of the initial data is contained in the domain, i.e., $\operatorname{supp}(\rhob_0) \subset \Omega$. 
Moreover, the parameter $R_c > 0$ must be chosen sufficiently large such that the vector fields point inward along the boundary of the domain.
An explicit condition guaranteeing this is provided in \cite[Remark 1]{BGIV20}.

In our experiment, the domain $\Omega$ describes two crossed corridors with an obstacle, as in \cite{GR24}. More precisely, $\Omega = \R \times (-1,1) \cup (-1,1) \times \R \setminus \Omega_1^c$ for the obstacle $\Omega_1^c = [2.35,2.5]\times [2.85,3]$, where we focus the experiment on the part included in $[-3,3]^2$.
The vector field $\mathbf{w}^k$, $k=1,2$, describes the space dependent target direction of each population. The first population is initialized by \mbox{$\rho_0^1(\xb) = 0.4 \, \raisebox{0.6ex}{\scalebox{1.2}{$\chi$}}_{(-2.35,-1.45)\times (-0.75,0.75)}$} and seeks to exit the domain through $\Gamma_1 = \{3\} \times [-1,1]$. The second population, with initial density $\rho_0^2(\xb) = 0.5 \, \raisebox{0.6ex}{\scalebox{1.2}{$\chi$}}_{(-0.75,0.75)\times (-2.35,-1.25)}$, targets the exit $\Gamma_2 = [-1,1] \times \{3\}$. Thus, we impose absorbing boundary conditions at these exits.
Moreover, we choose 
\begin{equation}
    R_c = 3,\quad \beta = 0.7, \quad v_{\max}=4.5 ,\quad \tilde{\eta}_\ell(\xb) = \frac{315}{128 \pi \ell^{18}} (\ell^4 - ||\xb||_2^4)^4 \; \raisebox{0.6ex}{\scalebox{1.2}{$\chi$}}_{B_\ell(0)}\mleft( \left\lVert \xb \right\rVert _2 \mright) \text{ with } \ell=\reva{0.2}.
\end{equation}
The model fits into the setting of \eqref{eq:system} with $K=3$, $\fb^3 \equiv 0$ and 
$$ \etab = \begin{bmatrix}
0 & \ddx\tilde{\eta}_{r_2} & \ddx\tilde{\eta}_{r_2}\\
0 & \ddy\tilde{\eta}_{r_2} & \ddy\tilde{\eta}_{r_2}\\
\ddx\tilde{\eta}_{r_2} & 0 & \ddx\tilde{\eta}_{r_2}\\
\ddy\tilde{\eta}_{r_2} & 0 & \ddy\tilde{\eta}_{r_2}
\end{bmatrix}.$$

\paragraph{Direction field}
To compute the direction field $\mathbf{w}^k$, $k=1,2$, we follow \cite{PT09}. First, we solve  a Laplace's equation $\Delta \mathbf{u}^k(\xb) = 0$ with suitable Dirichlet boundary conditions. 
The outer boundary of $\Omega$ is defined by $\Omega_\text{out}$, which includes the subsets $\Gamma_k$, $k=1,2$ that denote the targeted exits. Then, the boundary conditions are given by
\begin{align}\label{eq:dirfieldbc}
  \begin{cases}
    \mathbf{u}^k(\xb) = 0 \quad, \xb \in \Omega_\text{out}\setminus\Gamma_k,\\
    \mathbf{u}^k(\xb) = 1 \quad, \xb \in \Gamma_k.
  \end{cases}
\end{align}
Since strong and broad repulsion at the outer boundary, i.e., the walls, may appear, we weaken this effect and solve a modified Helmholtz equation
\begin{align*}
  \mleft[ \Delta - \frac{1}{\delta} \mright]  \mathbf{u}^k(\xb) = 0
\end{align*}
with $\delta = 10^{-5}$ and the boundary conditions \eqref{eq:dirfieldbc}. 
Now, the direction field $\mathbf{w}^k$, $k=1,2$, is obtained by computing the normalized gradient of $\mathbf{u}^k$. 
This direction field points along the preferred path to the exit and assigns a small weight to the domain boundaries. 
We note that the boundaries are already included in the convolutions.

\begin{figure}[ht]
    \centering
    \hspace{-0.5cm}
    \begin{subfigure}[b]{0.34\textwidth}
        \centering
        \includegraphics[width=\textwidth]{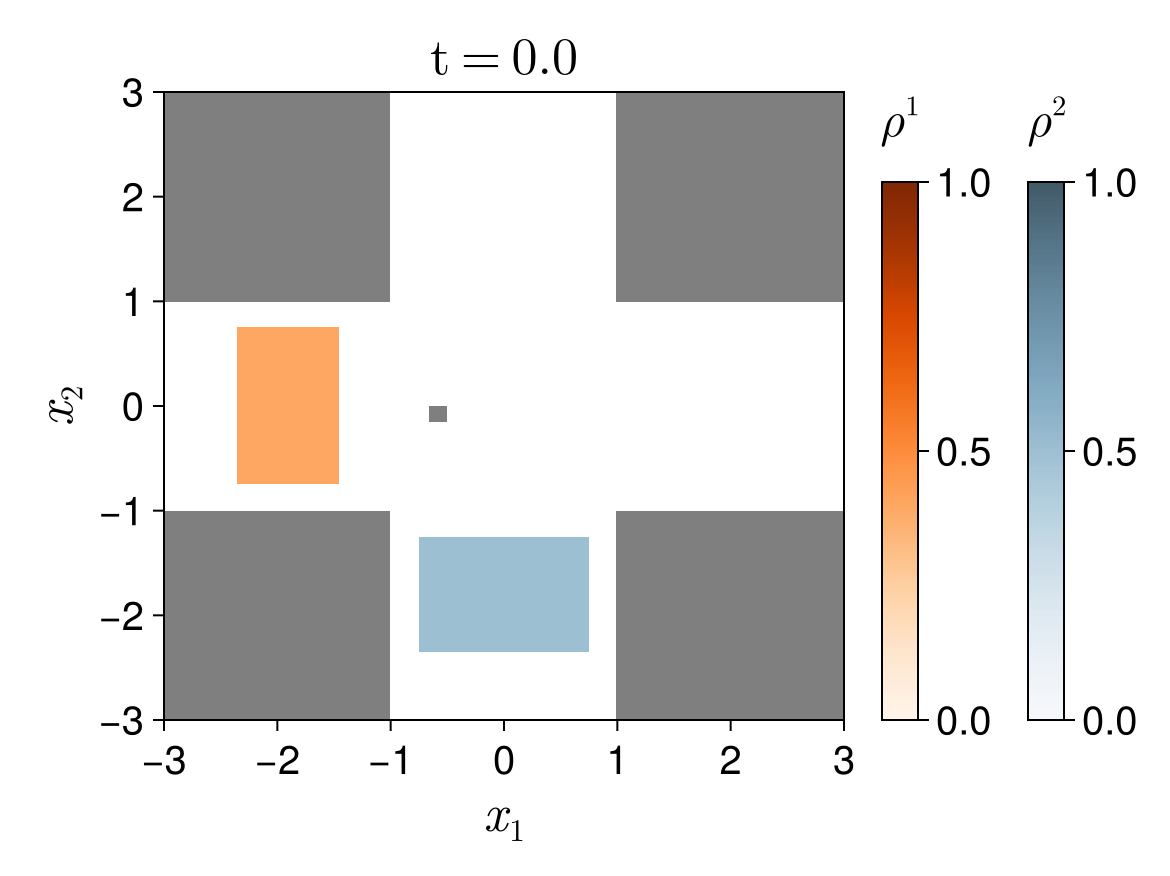}
    \end{subfigure}
    \hspace{-0.35cm}
    \begin{subfigure}[b]{0.34\textwidth}
        \centering
        \includegraphics[width=\textwidth]{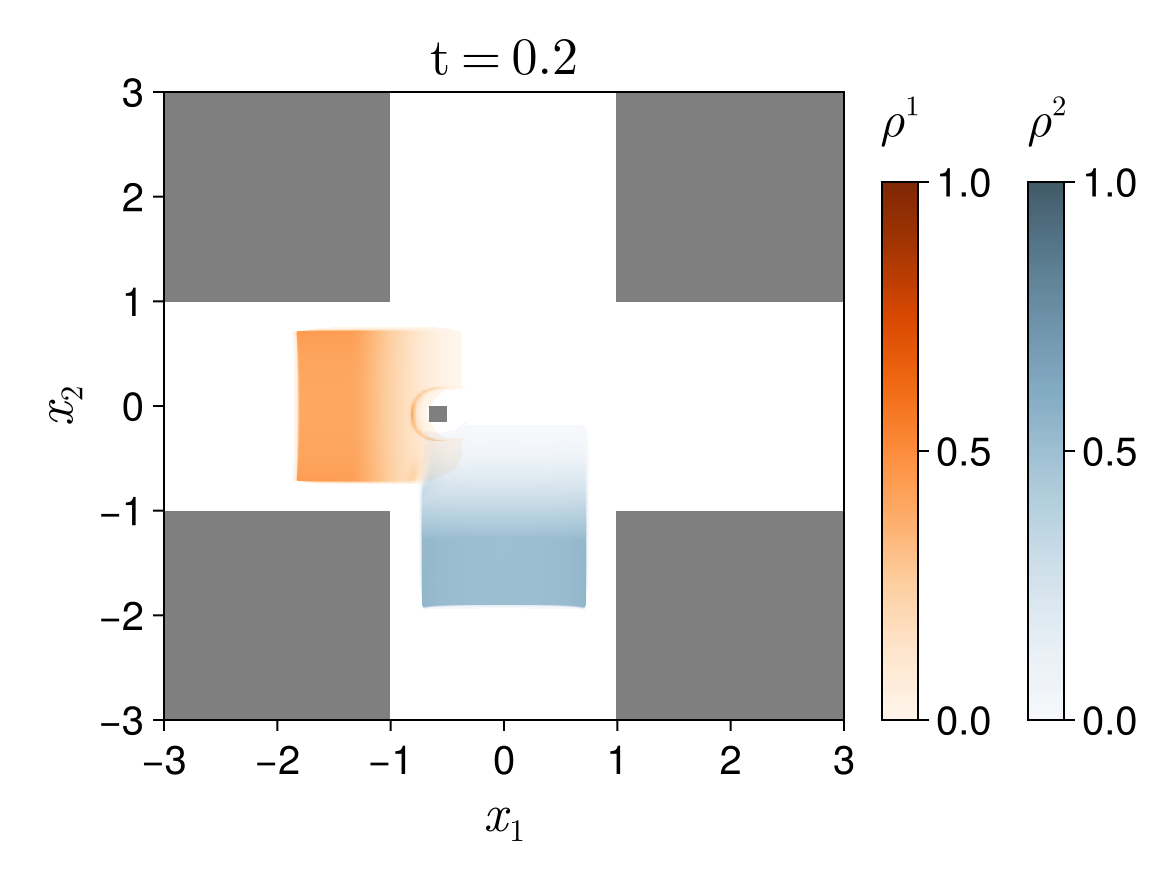}
    \end{subfigure}
    \hspace{-0.35cm}
    \begin{subfigure}[b]{0.34\textwidth}
        \centering
        \includegraphics[width=\textwidth]{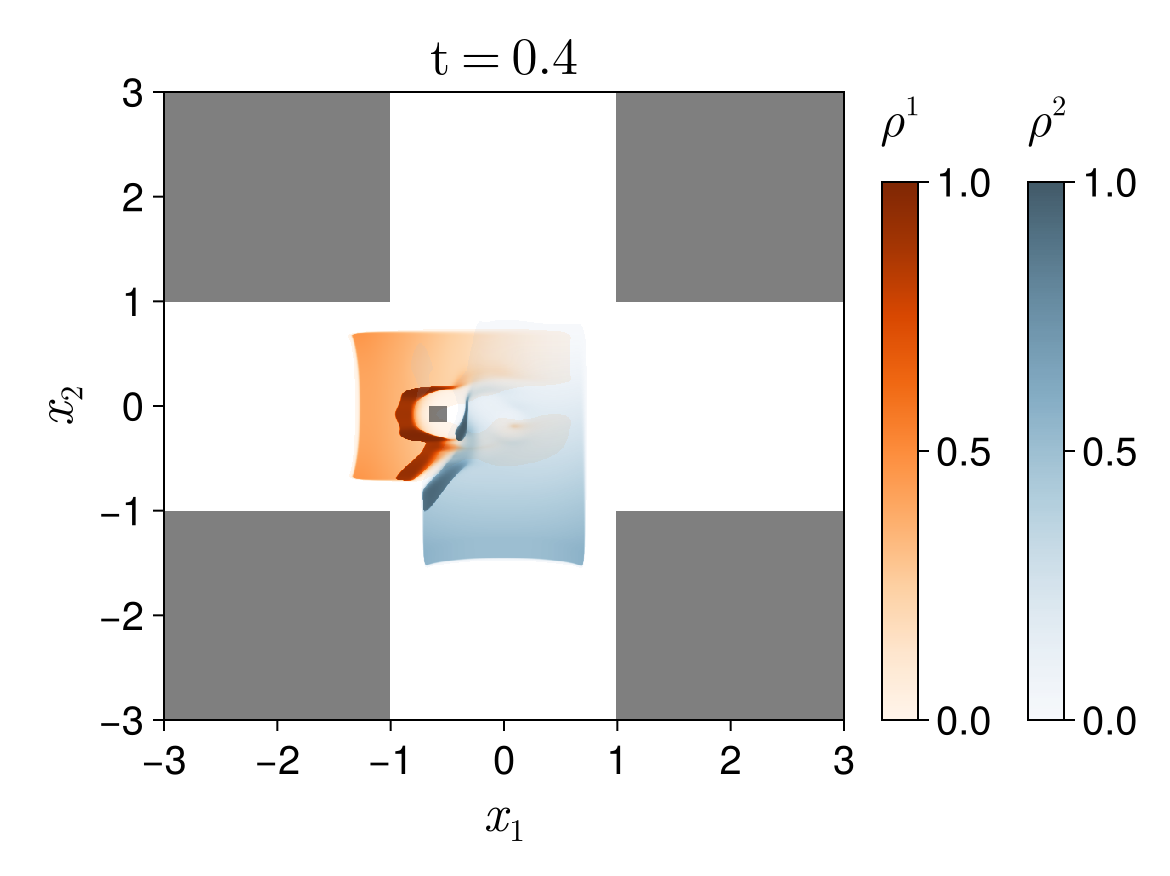}
    \end{subfigure}\\
    \hspace{-0.5cm}
    \begin{subfigure}[b]{0.34\textwidth}
        \centering
        \includegraphics[width=\textwidth]{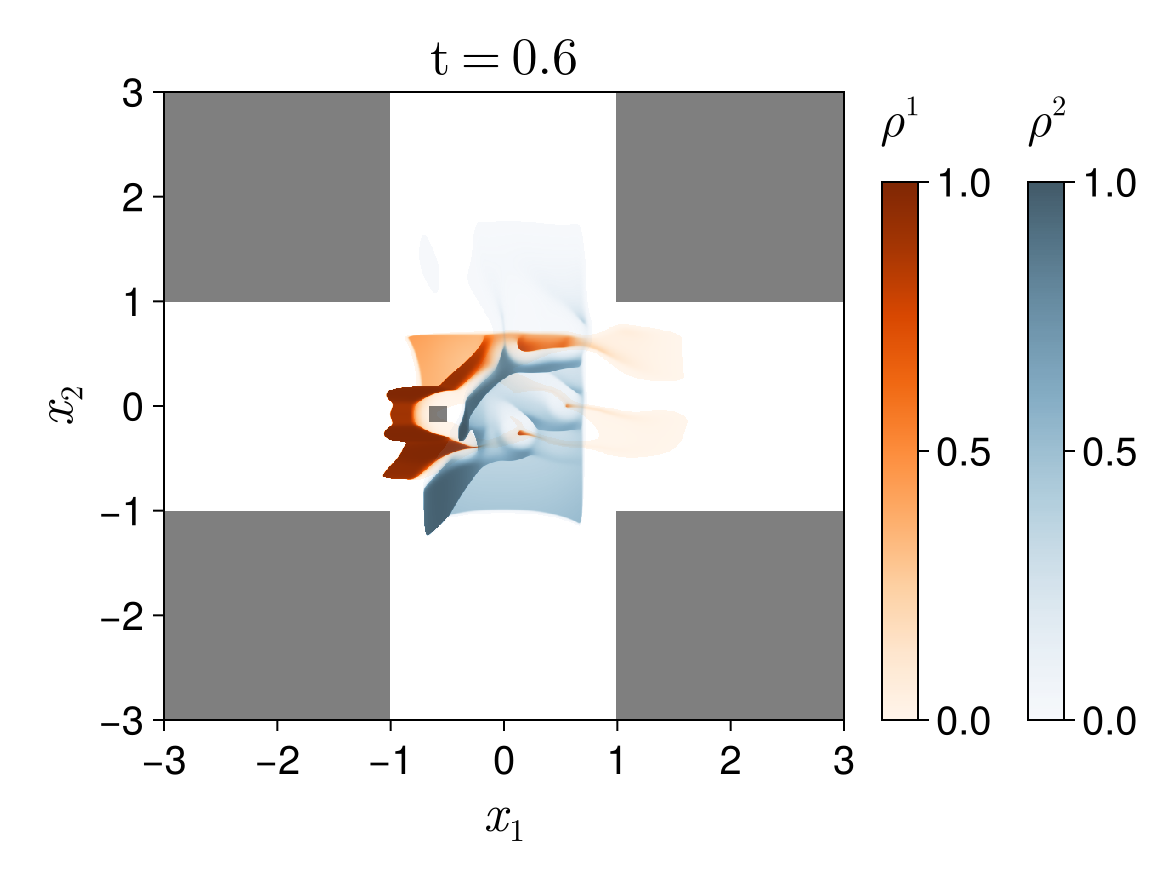}
    \end{subfigure}
    \hspace{-0.35cm}
    \begin{subfigure}[b]{0.34\textwidth}
        \centering
        \includegraphics[width=\textwidth]{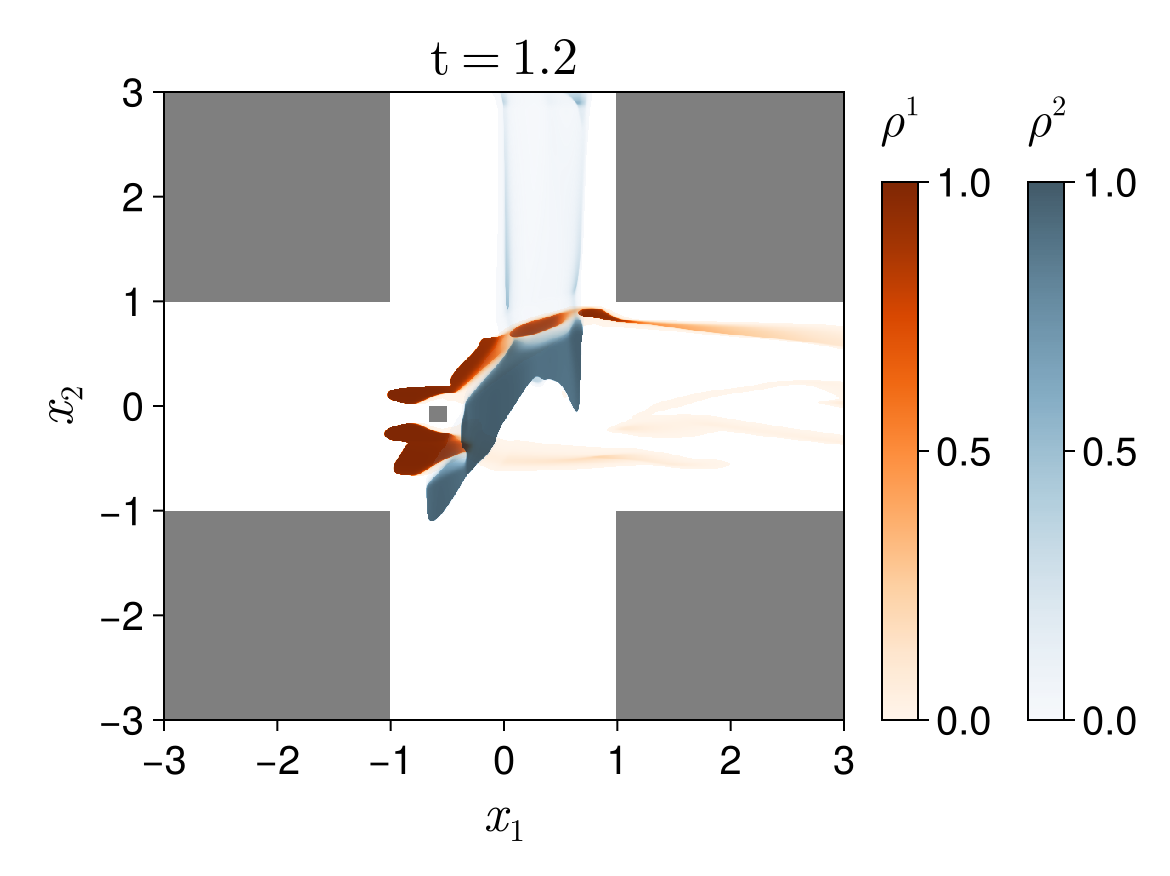}
    \end{subfigure}
    \hspace{-0.35cm}
    \begin{subfigure}[b]{0.34\textwidth}
        \centering
        \includegraphics[width=\textwidth]{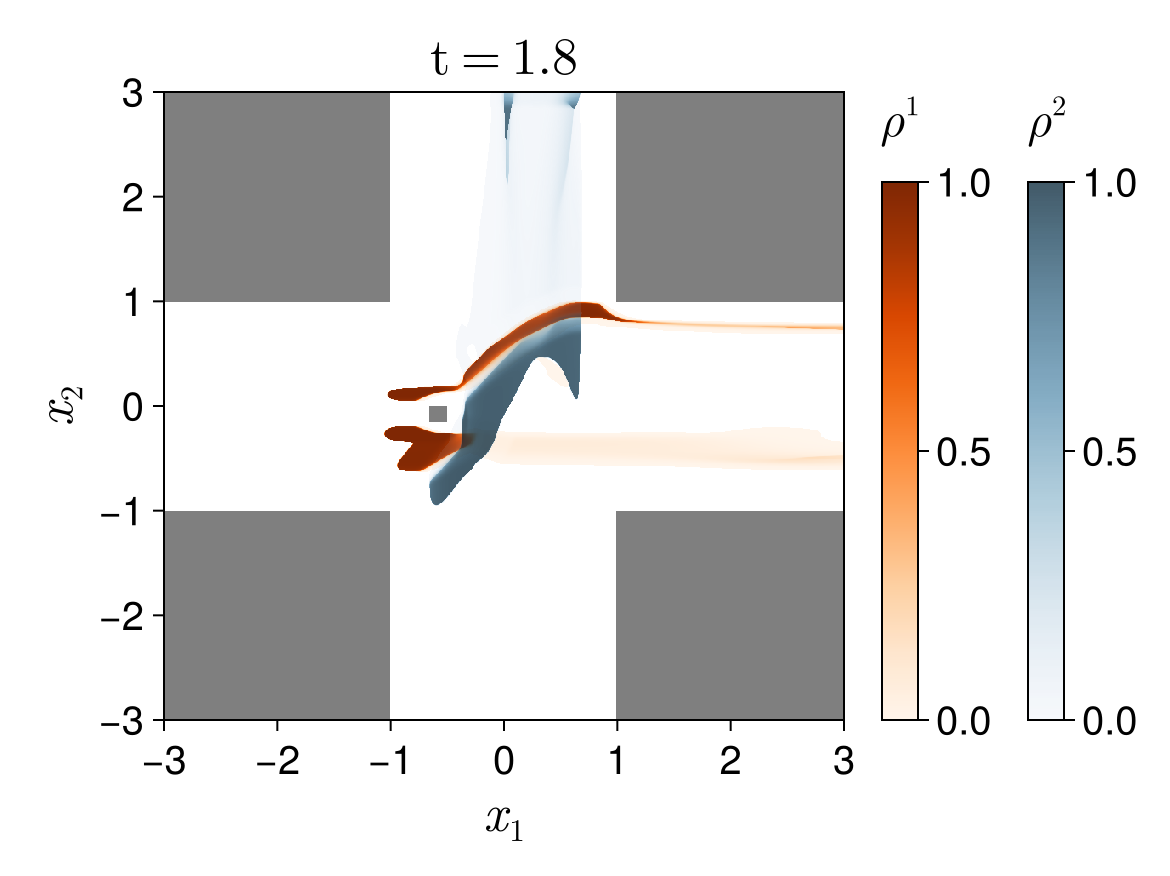}
    \end{subfigure}
    \caption{Solution at different time steps $t \in  \{0.0,0.2,0.4,0.6,1.2,1.8\}$ computed with the Godunov-type scheme on a mesh with $N=800$. The gray regions indicate the points outside the domain where the density is prescribed as $R_c = 3$.} 
    \label{fig:corridor}
\end{figure}

\paragraph{Comparison of different schemes}
Due to the multiplicative form of the model \eqref{eq:pedmodel} and the nonlinear flux in $\rho$, the Godunov-type scheme and the Lax-Friedrichs-type scheme \eqref{eq:LxFnew} do not coincide as in the previous subsection. 
Thus, we compare the following schemes: The Lax-Friedrichs-type scheme \cite{ACG15}, the less diffusive Lax-Friedrichs-type scheme \eqref{eq:LxFnew} and the Godunov-type scheme \eqref{eq:God}. 
Since the absolute value of the vector field does not exceed $1 + \beta$ the viscosity coefficient for the Lax-Friedrichs-type scheme \cite{ACG15} is set to $\alpha=v_{\max}(1+\beta)=7.65$, whereas the coefficient for the version proposed in \eqref{eq:LxFnew} can be chosen as $\alpha = v_{\max}=4.5$. 
The Lipschitz constants for all schemes are given by $L=v_{\max}(1+\beta)$, such that $\dt = \frac{\dx}{4L}$ with $\dx=\dy$.

Figure \ref{fig:corridor} displays the solution of the Godunov-type scheme on a grid with $N=800$ cells in each direction for different time steps to illustrate the general development of the two population densities $\rho^1$ and $\rho^2$. 
The first population needs to pass the obstacle, which is located close to the center. 
Here, the two populations collide which creates lanes such that the populations can pass each other.

Next, we fix the time $t=0.6$ and compare the three different schemes. The associated plots are displayed in Figure~\ref{fig:t06}. 
The Lax-Friedrichs-type scheme \cite{ACG15} (on the left) is more diffusive than the other two schemes. 
In contrast, the Lax-Friedrichs-type scheme \eqref{eq:LxFnew} (in the middle) shows significantly improved performance.
The differences to the Godunov-type scheme (on the right) seem to be negligible, although the latter still produces slightly sharper lines. 
Note that in all cases and at all time steps in Figure \ref{fig:corridor}, the two densities remain between 0 and 1, as required by the maximum principle.

\begin{figure}[b]
    \centering
    \hspace{-0.5cm}
    \begin{subfigure}[b]{0.34\textwidth}
        \centering
        \includegraphics[width=\textwidth]{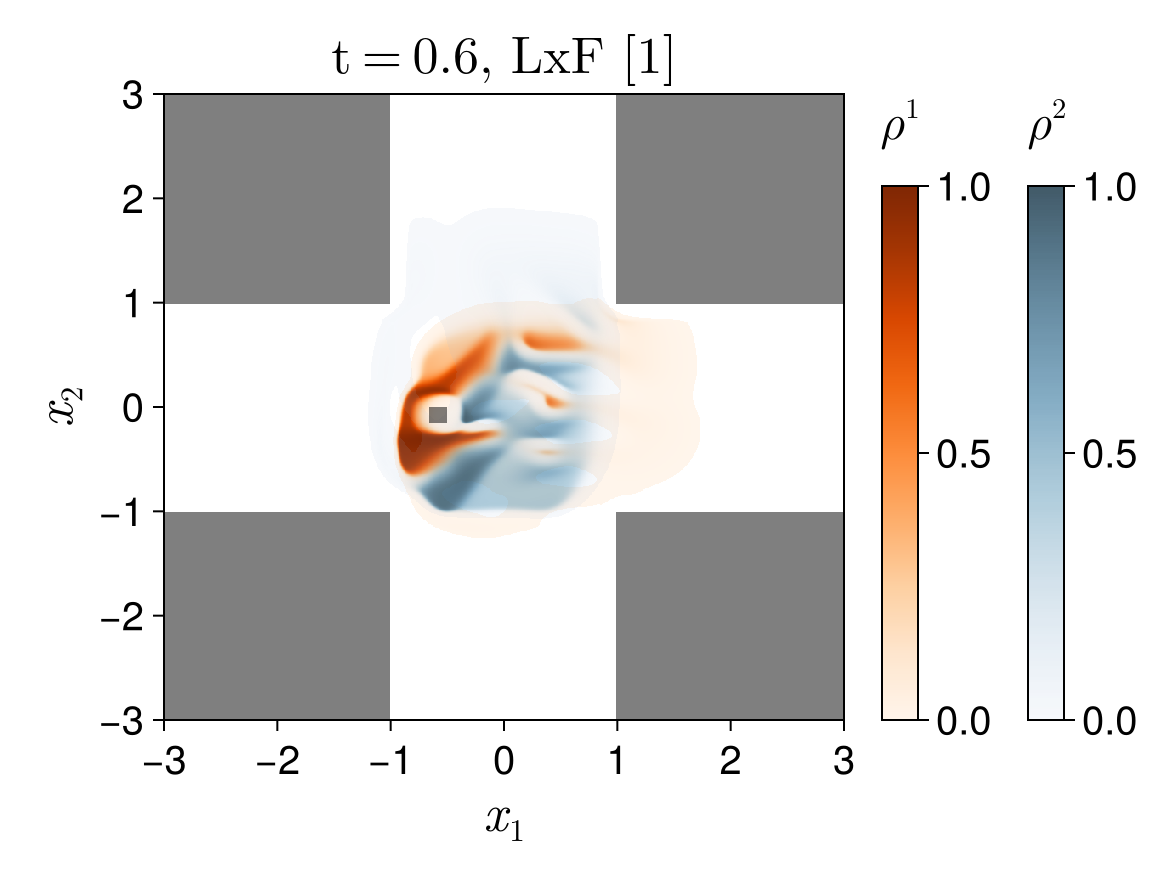}
    \end{subfigure}
    \hspace{-0.35cm}
    \begin{subfigure}[b]{0.34\textwidth}
        \centering
        \includegraphics[width=\textwidth]{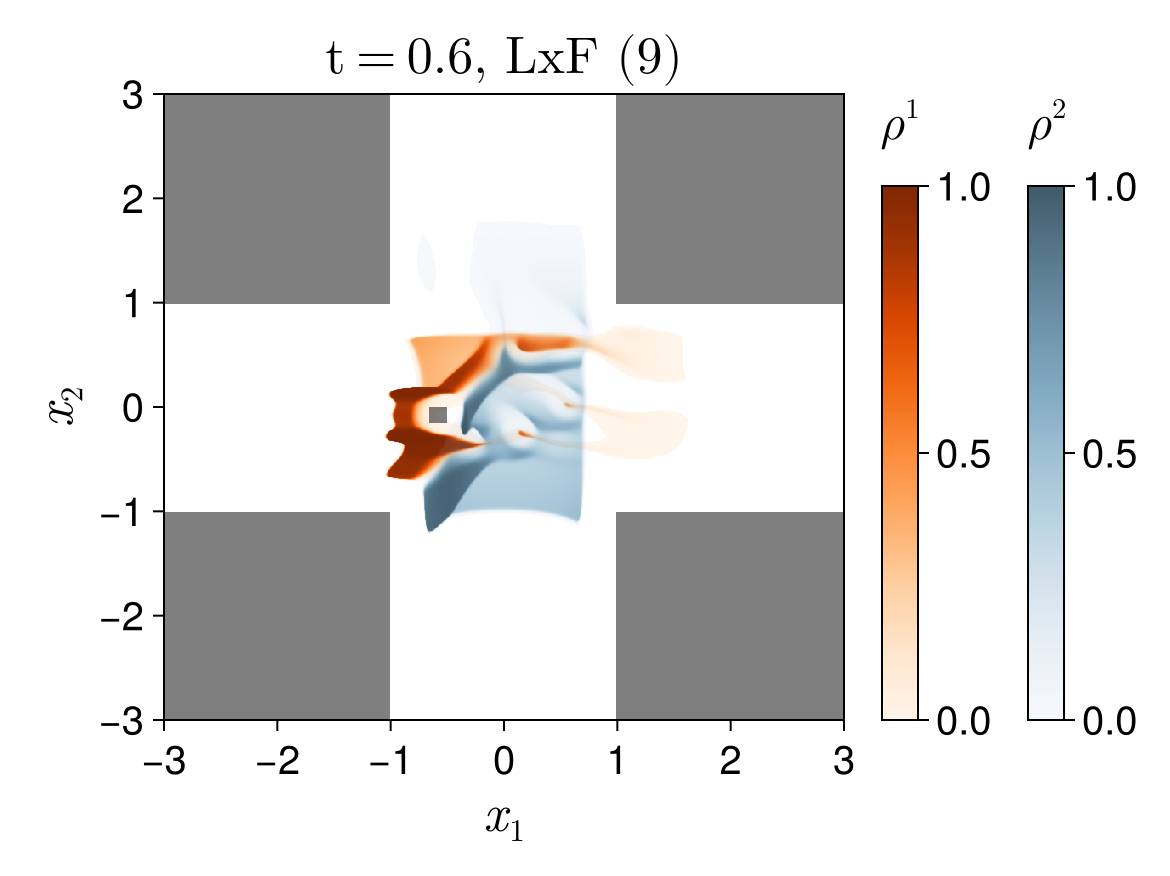}
    \end{subfigure}
    \hspace{-0.35cm}
    \begin{subfigure}[b]{0.34\textwidth}
        \centering
        \includegraphics[width=\textwidth]{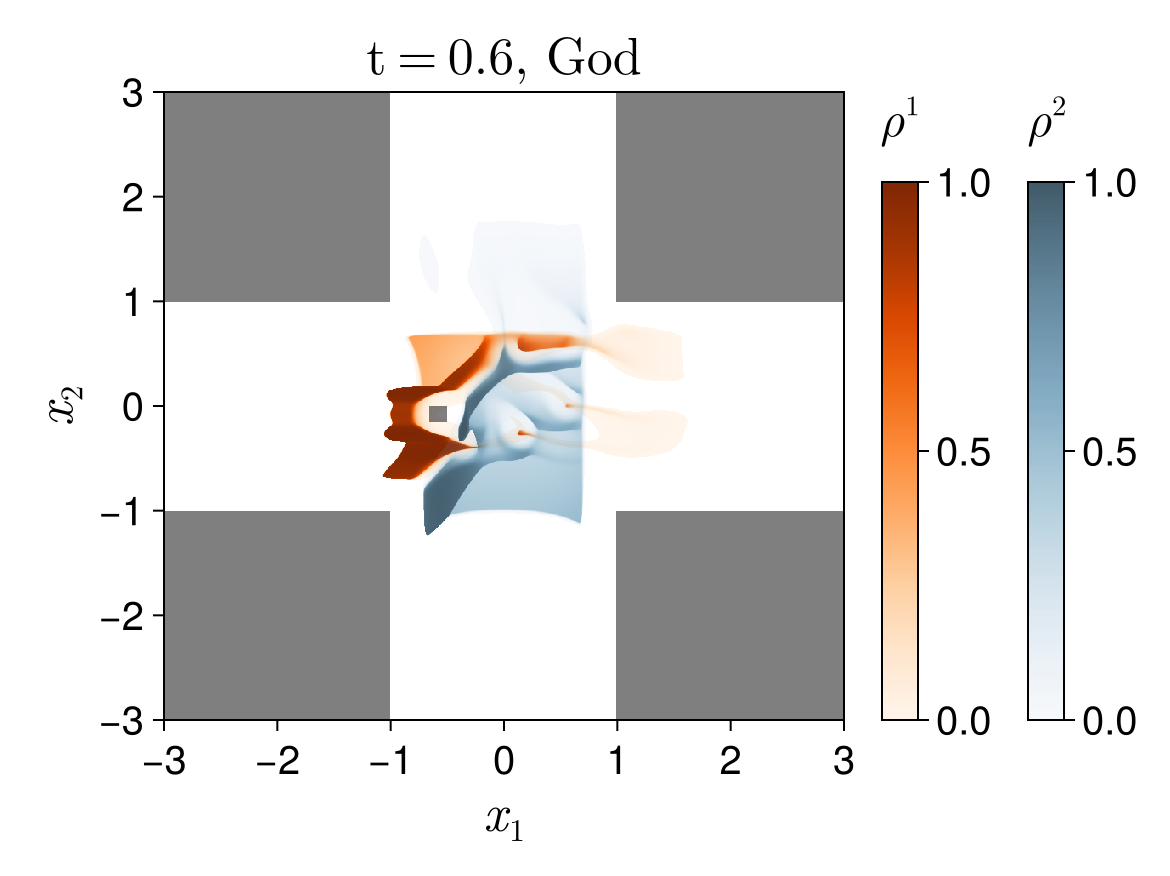}
    \end{subfigure}
    \caption{Solution at $t=0.6$ computed with different schemes on a mesh with $N=800$: the Lax-Friedrichs-type scheme \cite{ACG15} on the left, the Lax-Friedrich-type scheme \eqref{eq:LxFnew} in the middle and the Godunov-type scheme \eqref{eq:God} on the right.}
    \label{fig:t06}
\end{figure}

For a better comparison, in particular concerning the Godunov-type scheme and Lax-Friedrichs-type scheme \eqref{eq:LxFnew}, we compute the errors for different meshes to a reference solution that is computed with the Godunov-type scheme for $N=3200$.
As shown in Table~\ref{tab:corridor_errors} the errors produced by the Lax-Friedrichs-type scheme \eqref{eq:LxFnew} are only slightly larger than with the Godunov-type scheme \eqref{eq:God}. 
The largest errors are obtained by using the Lax–Friedrichs-type scheme from \cite{ACG15}.
Hence, the Lax-Friedrichs-type scheme \eqref{eq:LxFnew} provides a good compromise between accuracy and simplicity of implementation, since no Riemann problems of the reduced flux need to be solved as required by the Godunov-type scheme \eqref{eq:God}.
We note that only the Godunov-type scheme approaches the convergence rate of one and the worst case error estimate does not seem to apply.

\begin{table}
  \centering
  \begin{tabular}{c|cc|cc|cc}
  &\multicolumn{2}{c|}{LxF\cite{ACG15}}&\multicolumn{2}{c|}{LxF \eqref{eq:LxFnew}}& \multicolumn{2}{c}{God \eqref{eq:God}}\\
  $N$ & error & c.r. & error & c.r. & error & c.r. \\
  \hline
  50 & \reva{1.82} & - & \reva{1.14} & - & \reva{1.10} & - \\
  100 & \reva{1.51} & 0.269 & \reva{0.829} & 0.456 & \reva{0.686} & 0.678 \\
  200 & \reva{1.23} & 0.292 & \reva{0.669} & 0.310 & \reva{0.474} & 0.533 \\
  400 & \reva{1.03} & 0.261 & \reva{0.463} & 0.532 & \reva{0.309} & 0.615 \\
  800 & \reva{0.801} & 0.360 & \reva{0.267} & 0.791 & \reva{0.156} & 0.985 \\
  \end{tabular}  
  \caption{Errors and convergence rates for various grid sizes using the Lax-Friedrichs-type scheme from \cite{ACG15}, the Lax-Friedrichs-type scheme proposed in \eqref{eq:LxFnew} and the Godunov-type scheme \eqref{eq:God}. The errors are calculated with respect to a reference solution obtained using the Godunov-type scheme on a grid with $N=3200$.}
  \label{tab:corridor_errors}
\end{table}

\section{Conclusion}\label{sec:conclusion}
In this work, we have \revc{established the convergence for a general class of numerical} schemes for two-dimensional nonlocal conservation laws. 
In particular, we have proven that well-known monotone schemes for local conservation laws can be applied such as Godunov- or Lax-Friedrichs-type schemes. 
It should be noted that the resulting schemes are not monotone in every argument due to the nonlocal term. 
In addition, we have provided an error estimate for \revc{these monotone-based} schemes which is of magnitude $\mathcal{O}(\sqrt{\Delta t})$. 

Even though our results can be adapted to a scheme with dimensional splitting, we concentrated on those without dimensional splitting because this allows for a straightforward extension to higher order schemes. 
In particular, to guarantee the positivity preservation and, possibly, a maximum principle for suitable nonlocal conservation laws
it is important that the numerical fluxes \reva{satisfy} those for a first order scheme, as we have shown here. 
We will investigate higher order approximations based on the here presented class of numerical schemes in a follow-up work.

\section*{Declarations}
\subsection*{Availability of data and materials}
Not applicable.

\subsection*{Competing interest}
The authors declare that they have no competing interests.

\subsection*{Funding}
A.~B. and J.~F. are supported by the German Research Foundation (DFG) through SPP 2410 `Hyperbolic Balance Laws in Fluid Mechanics: Complexity, Scales, Randomness' under grant FR 4850/1-1. Moreover, A.~B. is partially funded by the DFG project 320021702/GRK2326 ’Energy, Entropy, and Dissipative Dynamics (EDDy)’.

\subsection*{Authors' contribution}
A.~B. implemented the code and prepared the numerical examples examined in Section \ref{sec:experiments}, while both authors contributed equally to the other sections. All authors read and approved the final manuscript.

\subsection*{Acknowledgment}
Both authors thank Michael Herty for fruitful discussions during the completion of this work.
\reva{Furthermore, we would like to thank the two anonymous referees for investing time in reviewing this work, which helped improve it further.}

  \bibliographystyle{siam} 
  \bibliography{references.bib}

  \appendix
\section{Appendix}\label{sec:appendix}

We prove that the Lax-Friedrichs-type numerical flux \eqref{eq:LxF} satisfies \revb{condition} 4 in Def.\ \ref{def:flux}. We concentrate on the estimate
\begin{align}\label{eq:dblLipschitzproof}
  \left| \numfluxkn{2}{i+1}{j-\frac{1}{2}}{\rho_{i,j-1}}{\rho_{i,j}}
                - \numfluxkn{2}{i+1}{j-\frac{1}{2}}{\rho_{i,j}}{\rho_{i,j}}
                -  \numfluxkn{2}{i}{j-\frac{1}{2}}{\rho_{i,j-1}}{\rho_{i,j}} 
                + \numfluxkn{2}{i}{j-\frac{1}{2}}{\rho_{i,j}}{\rho_{i,j}} \right|\\
                 \leq (\dx \left| \rho_{i,j}^n - \rho_{i,j-1}^n \right|+ \mathcal{M} |\rho_{i,j}^n| \dx^2) L_{2,1}'
,\end{align}
since the remaining ones follow similarly.
\begin{proof}[Proof of \eqref{eq:dblLipschitzproof}]
  Using the Lax-Friedrichs-type flux function \eqref{eq:LxF} gives us
  \begin{align}\label{eq:LxF4}
    &\numfluxkn{2}{i+1}{j-\frac{1}{2}}{\rho_{i,j-1}}{\rho_{i,j}}
                - \numfluxkn{2}{i+1}{j-\frac{1}{2}}{\rho_{i,j}}{\rho_{i,j}}
                -  \numfluxkn{2}{i}{j-\frac{1}{2}}{\rho_{i,j-1}}{\rho_{i,j}} 
                + \numfluxkn{2}{i}{j-\frac{1}{2}}{\rho_{i,j}}{\rho_{i,j}}\\
    =\,& \frac{1}{2} \mleft( \flux{2}{i+1}{j-\frac{1}{2}}{\rho_{i,j-1}^{k,n}} - \flux{2}{i}{j-\frac{1}{2}}{\rho_{i,j-1}^{k,n}} - \flux{2}{i+1}{j-\frac{1}{2}}{\rho_{i,j}^{k,n}}  + \flux{2}{i}{j-\frac{1}{2}}{\rho_{i,j}^{k,n}} \mright) 
  .\end{align}
  Applying the mean value theorem on the first two and the last two flux terms yields
  \begin{align*}
    &\frac{1}{2}\Biggl( \ddx \fluxself{2}{\widetilde{\xb}^{i+\frac{1}{2},j-\frac{1}{2}}}{\rho_{i,j-1}^{k,n}}{\Rb_{i+1,j-\frac{1}{2}}^{n}} \dx - \ddx \fluxself{2}{\widehat{\xb}^{i+\frac{1}{2},j-\frac{1}{2}}}{\rho_{i,j}^{k,n}}{\Rb_{i+1,j-\frac{1}{2}}^{n}} \dx\\  
    + & \mleft( \nabla_{\Rb} \fluxself{2}{\xb^{i+\frac{1}{2},j-\frac{1}{2}}}{\rho_{i,j-1}^{k,n}}{\widetilde{\Rb}_{i+\frac{1}{2},j-\frac{1}{2}}^{n}}\! - \! \nabla_{\Rb} \fluxself{2}{\xb^{i+\frac{1}{2},j-\frac{1}{2}}}{\rho_{i,j-1}^{k,n}}{\widehat{\Rb}_{i+\frac{1}{2},j-\frac{1}{2}}^{n}} \! \mright) \! \mleft( \Rb_{i+1,j-\frac{1}{2}}^n - \Rb_{i,j-\frac{1}{2}}^n \mright) \!\!
    \Biggr) 
  \end{align*}
  for $\widetilde{\xb}^{i+\frac{1}{2},j-\frac{1}{2}} = (\widetilde{x}_1^{i+\frac{1}{2}}, x_2^{j-\frac{1}{2}})^T$ and $\widehat{\xb}^{i+\frac{1}{2},j-\frac{1}{2}} = (\widehat{x}_1^{i+\frac{1}{2}}, x_2^{j-\frac{1}{2}})^T$ with 
  $\widetilde{x}_1^{i+\frac{1}{2}}, \widehat{x}_1^{i+\frac{1}{2}} \in [x_1^{i},x_1^{i+1}]$ and $\widetilde{\Rb}_{i+\frac{1}{2},j-\frac{1}{2}}^{n}$, $\widehat{\Rb}_{i+\frac{1}{2},j-\frac{1}{2}}^{n} \in  I[{\Rb}_{i,j-\frac{1}{2}}^{n}, {\Rb}_{i+1,j-\frac{1}{2}}^{n}]$. 
  Hence, we again apply the mean value theorem and utilize assumption $(\fb)$ to obtain the following estimate on the absolute value
  \begin{align*}
    |\eqref{eq:LxF4}| \leq &\frac{1}{2}\Bigl( \Lipf{\rho}{x_1} \dx |\rho_{i,j}^{k,n}- \rho_{i,j-1}^{k,n}| + \mathcal{M} |\rho_{i,j}^{k,n}| \dx^2\\
    &\quad+  \Lipf{\rho}{\Rb} |\rho_{i,j}^{k,n}- \rho_{i,j-1}^{k,n}| \| \Rb_{i+1,j-\frac{1}{2}}^n - \Rb_{i,j-\frac{1}{2}}^n \|_2 + \mathcal{M} |\rho_{i,j}^{k,n}| \| \Rb_{i+1,j-\frac{1}{2}}^n - \Rb_{i,j-\frac{1}{2}}^n \|_2^2 \Bigr)\\
    \leq & \mleft( \dx \left| \rho_{i,j}^n - \rho_{i,j-1}^n \right|+ \mathcal{M} |\rho_{i,j}^n| \dx^2 \mright) L_{2,1}'
  \end{align*}
  with $L_{2,1}' = \frac{1}{2}\max\{\Lipf{\rho}{x_1} + \Lipf{\rho}{\Rb}\norm{\ddx \etab}_\infty \norm{\rhob_0}_{L^1}, 1 + \norm{\ddx \etab}_\infty^2 \norm{\rhob_0}_{L^1}^2  \}$. In the last inequality we employed \mbox{\cite[Lem.~A.2]{ACG15}}.
\end{proof}

\begin{proposition}\label{prop:dblLipmult}
For a multiplicative flux function $f^k\mleft( t,\xb, \rho^k, \etab * \rhob \mright) =g^k\mleft( \rho^k \mright)  \nub^k\mleft( t,\xb,\etab * \rhob \mright)$ as described in Rem.~\ref{rem:multflux} we consider the class of numerical fluxes given by \eqref{eq:F=GV}. Then there exist constants $L_{\ell,2}^G, L_{\ell,2}^G > 0$ such that the estimates
  \begin{align*}
  \Bigl| &F_\ell^k \mleft( a, c, V \mright)
    - F_\ell^k \mleft( b, c, V \mright)- F_\ell^k \mleft( a, c, W \mright)
    + F_\ell^k \mleft( b, c, W \mright) \Bigr| \leq L_{\ell,1}^G \left| a-b \right| \left| V-W \right|\\
  \Bigl| &F_\ell^k \mleft( c, a, V \mright)
    - F_\ell^k \mleft( c, b , V \mright)- F_\ell^k \mleft( c, a, W \mright)
    + F_\ell^k \mleft( c, b, W \mright) \Bigr| \leq L_{\ell,2}^G \left| a-b \right| \left| V-W \right|
\end{align*}
hold for $a,b,c \in \imgrho_k$, $V,W \in \R^M$ and $\ell=1,2$. This immediately implies assumption~4 in Def.\ \ref{def:flux}.
\end{proposition}
\begin{proof}
  We only prove the first inequality, since the second one is analogous. In the following we use the identity
  \begin{align*}
    &F_\ell \mleft( a, c, V \mright)
    - F_\ell \mleft( b, c, V \mright)- F_\ell \mleft( a, c, W \mright)
    + F_\ell \mleft( b, c, W \mright)\\
    & =\bigl( G_\ell( a, c, \textnormal{sgn}(V) ) - G_\ell( b, c, \textnormal{sgn}(V))\bigr)  |V| 
      - \bigl( G_\ell( a, c, \textnormal{sgn}(W) ) - G_\ell( b, c, \textnormal{sgn}(W))\bigr)  |W|
  .\end{align*}
We consider two cases depending on the signs of $V$ and $W$.

\begin{enumerate}
  \item[\textit{1.}] \textit{case:} sgn($V$)=sgn($W$). We have
        \begin{align*}
          \Bigl| &F_\ell \mleft( a, c, V \mright)- F_\ell \mleft( b, c, V \mright)
                - F_\ell \mleft( a, c, W \mright)+ F_\ell \mleft( b, c, W \mright) \Bigr|\\
                &= \bigl| \bigl( G_\ell( a, c, \textnormal{sgn}(V) ) - G_\ell( b, c, \textnormal{sgn}(V))\bigr) \left( |V|-|W| \right) \bigr| \\
                &\leq  \bigl| G_\ell( a, c, \textnormal{sgn}(V) ) - G_\ell( b, c, \textnormal{sgn}(V))\bigr| \big| |V|-|W| \big|.
        \end{align*}
        With the Lipschitz continuity of $G_\ell$ and $\big| |V|-|W| \big| = \left| V-W\right| $ we obtain 
        \begin{align*}
          \Bigl| &F_\ell \mleft( a, c, V \mright)- F_\ell \mleft( b, c, V \mright)
                - F_\ell \mleft( a, c, W \mright)+ F_\ell \mleft( b, c, W \mright) \Bigr| \leq  L_{\ell,1}^G |a-b| |V-W|.
        \end{align*}
  \item[\textit{2.}] \textit{case:} sgn$(V)\neq \text{sgn}(W)$.
        Without loss of generality we assume sgn$(V)=1$ and sgn$(W)=-1$, such that
        \begin{align*}
          \Bigl| &F_\ell \mleft( a, c, V \mright)- F_\ell \mleft( b, c, V \mright)
                - F_\ell \mleft( a, c, W \mright)+ F_\ell \mleft( b, c, W \mright) \Bigr|\\
                  & \leq \bigl| G_\ell( a, c, \textnormal{sgn}(V) ) - G_\ell( b, c, \textnormal{sgn}(V))\bigr|  |V| 
                  + \bigl| G_\ell( a, c, \textnormal{sgn}(W) ) - G_\ell( b, c, \textnormal{sgn}(W))\bigr|  |W|\\
                  & \leq L_{\ell,1}^G |a-b| \mleft( |V| + |W| \mright)
        .\end{align*}
        Due to the signs of $V$ and $W$ it holds $(|V| + |W|) = (V - W) \leq |V-W|$.
\end{enumerate}
Hence, in both cases the claim holds.
\end{proof}

\begin{proof}[Proof of \eqref{eq:4Rdiff}]
  \begin{align}\label{eq:4R}
    \norm{R_{i+1,j-\frac{1}{2}}^n &- R_{i,j-\frac{1}{2}}^n - R_{i+1,j+\frac{1}{2}}^n + R_{i,j+\frac{1}{2}}^n }_2 \\
    &\leq \dx \dy \sum_{p,q\in\Z} \left\lVert \mleft( \etab_{p+1,q-\frac{1}{2}} - \etab_{p,q-\frac{1}{2}} - \etab_{p+1,q+\frac{1}{2}} + \etab_{p,q+\frac{1}{2}}  \mright) \rhob^{n}_{i-p,j-q}  \right\rVert _2 \notag\\
    &\leq \dx \dy \sum_{p,q\in\Z} \left\lVert \etab_{p+1,q-\frac{1}{2}} - \etab_{p,q-\frac{1}{2}} - \etab_{p+1,q+\frac{1}{2}} + \etab_{p,q+\frac{1}{2}}   \right\rVert _2 \left\lVert \rhob^{n}_{i-p,j-q}  \right\rVert _2.
  \end{align}
  Applying the mean value theorem componentwise on the matrices leaves us with
  \begin{align*}
    &\eta^{m,k}_{p+1,q-\frac{1}{2}} - \eta^{m,k}_{p,q-\frac{1}{2}} - \eta^{m,k}_{p+1,q+\frac{1}{2}} + \eta^{m,k}_{p,q+\frac{1}{2}}\\
    =& \eta^{m,k}( x_1^{p+1},x_2^{q-\frac{1}{2}}) - \eta^{m,k}( x_1^{p},x_2^{q-\frac{1}{2}})  - \eta^{m,k}( x_1^{p+1},x_2^{q+\frac{1}{2}})  + \eta^{m,k}( x_1^{p},x_2^{q+\frac{1}{2}})\\ 
    =&\ddx \eta^{m,k}( \tilde{x}_1^{p+\frac{1}{2}},x_2^{q-\frac{1}{2}}) \dx - \ddx \eta^{m,k}( \hat{x}_1^{p+\frac{1}{2}},x_2^{q+\frac{1}{2}}) \dx
  \end{align*}
  for $\tilde{x}_1^{p+\frac{1}{2}}, \hat{x}_1^{p+\frac{1}{2}} \in [x_1^{p},x_1^{p+1}]$ varying for each $m=1,\ldots,M, \; k=1,\ldots,K$.
  Next, 
  \begin{align*}
    &\ddx \eta^{m,k}( \tilde{x}_1^{p+\frac{1}{2}},x_2^{q-\frac{1}{2}}) \dx - \ddx \eta^{m,k}( \hat{x}_1^{p+\frac{1}{2}},x_2^{q+\frac{1}{2}}) \dx\\
    =& \partial_{x_1 x_1}^2 \eta^{m,k}( \breve{x}_1^{p+\frac{1}{2}},x_2^{q-\frac{1}{2}}) ( \tilde{x}_1^{p+\frac{1}{2}} - \hat{x}_1^{p+\frac{1}{2}}) \dx - \partial_{x_1 x_2}^2 \eta^{m,k}(\hat{x}_1^{p+\frac{1}{2}}, \overline{x}_2^{q}) \dx \dy
  \end{align*}
  for $ \breve{x}_1^{p+\frac{1}{2}} \in I[\hat{x}_1^{p+\frac{1}{2}}, \tilde{x}_1^{p+\frac{1}{2}}]$ and $\overline{x}_2^{q} \in [x_2^{q-\frac{1}{2}},x_2^{q+\frac{1}{2}}]$.
  Turning back to \eqref{eq:4R} we have
  \begin{align*}
    \norm{R_{i+1,j-\frac{1}{2}}^n &- R_{i,j-\frac{1}{2}}^n - R_{i+1,j+\frac{1}{2}}^n + R_{i,j+\frac{1}{2}}^n }_2\\
    &\leq \mleft( \left\lVert \partial_{x_1 x_1}^2 \etab \right\rVert _{\infty} 2 \dx^2 + \left\lVert \partial_{x_1 x_2}^2 \etab \right\rVert _{\infty}  \dx \dy\mright)  \dx \dy \sum_{p,q\in\Z} \left\lVert \rhob^{n}_{i+p,j+q}  \right\rVert _2\\
    &\leq  \mleft( \left\lVert \partial_{x_1 x_1}^2 \etab \right\rVert _{\infty} 2 \dx^2 + \left\lVert \partial_{x_1 x_2}^2 \etab \right\rVert _{\infty}  \dx \dy\mright)  \left\lVert \rhob_0 \right\rVert _{L^1}
  .\end{align*}
\end{proof}

\begin{proof}
    We again focus on the significant differences compared to the proof of \cite[Lem.~5.1]{aggarwal2024well}. We introduce the notation of time-space cells $C_{i,j}^n = [t^n,t^{n+1}) \times C_{i,j}$ and we will denote various constants $c_\ell \geq 0$, $\ell \in \N$ that depend on $\rhob_0, \etab, \fb$, the Lipschitz constants in Def.\ \ref{def:flux}  and $\delta=\frac{\dx}{\dy}$.
    Further, we define for $n \in \{0,\ldots,N_T \}, i,j \in \Z, k=1,\ldots,K, (t,\xb) \in \overline{Q}_T$
    \begin{align*}
        p_{i,j}^{k,n}(\kappa, R) \coloneq \textnormal{sgn}\mleft( \rho_{i,j}^{k,n} - \kappa \mright) \mleft( \fb^k\mleft( t,\xb,\rho_{i,j}^{k,n},R \mright) - \fb^k\mleft( t,\xb,\kappa,R \mright)  \mright)
    .\end{align*}
    We consider $-\Lambda_{\epsilon,\epsilon_0}^k(\rho_\Delta^k,\varrho^k)$ for the discrete $\rhob_\Delta$, perform summation by parts and add a zero to obtain
    \begin{align*}
        -\Lambda_{\epsilon,\epsilon_0}^k&(\rho_\Delta^k,\varrho^k) =\int_{Q_T} \sum_{i,j \in \Z} \sum_{n=0}^{N_T-1} \mleft[ \left| \rho_{i,j}^{k,n+1} - \varrho^k(s,\yb) \right| - \left| \rho_{i,j}^{k,n} - \varrho^k(s,\yb) \right| \mright] \int_{C_{i,j}}\phi\mleft( s,\yb,t^{n+1},\xb \mright) d\xb \, ds \, d\yb\\
        &- \int_{Q_T} \sum_{i,j \in \Z} \sum_{n=0}^{N_T-1} \int_{C_{i,j}^n} p_{i,j}^{k,n}(\varrho^k(s,\yb), \Rb_{i+\frac{1}{2},j}^n)  \mydot \nabla_\xb \phi dt \, d\xb \, ds \, d\yb  \\
        &- \int_{Q_T} \sum_{i,j \in \Z} \sum_{n=0}^{N_T-1} \int_{C_{i,j}^n} \mleft[ p_{i,j}^{k,n}(\varrho^k(s,\yb), \contconv{\rhob_\Delta}) - p_{i,j}^{k,n}(\varrho^k(s,\yb), \Rb_{i+\frac{1}{2},j}^n)\mright] \mydot \nabla_\xb \phi dt \, d\xb \, ds \, d\yb  \\
        & + \int_{Q_T} \sum_{i,j \in \Z} \sum_{n=0}^{N_T-1} \int_{C_{i,j}^n}\textnormal{sgn}\mleft( \rho_{i,j}^{k,n} - \varrho^k(s,\yb) \mright) \Div \fb^k\mleft( t, \xb, \varrho^k(s,\yb), \contconv{\rhob_\Delta} \mright)   \phi \, dt \, d\xb \, ds \, d\yb \\
        & \coloneq \biglambda_1 + \biglambda_2' + \bigeps_{2} + \biglambda_3.
    \end{align*}
    We note that we adapted here the notation of \cite{aggarwal2024well} as far as possible\footnote{The corresponding terms to $\biglambda_1, \biglambda_2', \biglambda_3$ in \cite{aggarwal2024well} are denoted by $\lambda_1, \lambda_2', \lambda_3$.} to increase the clarity when referring to steps in \cite{aggarwal2024well}.
    The derivatives in $\biglambda_2',\bigeps_{2}, \biglambda_3$ can be expressed as a sum of terms that involve either the first or the second component of the flux function. In the following we will focus solely on the terms for the first components denoted by $\biglambda_{2}'^1, \bigeps_{2}^1, \biglambda_{3}^1$, while the other terms follow analogously. Thus, we also denote the first component of $p_{i,j}^{k,n}$ by $p_{i,j,1}^{k,n}$.
    For an estimate on $\bigeps_{2}^1$ we need to bound 
    \begin{align*}
        & \left| p_{i,j,1}^{k,n}\mleft( \varrho^k(s,\yb),\contconv{\rhob_\Delta} \mright) - p_{i,j,1}^{k,n}\mleft( \varrho^k(s,\yb), \Rb_{i+\frac{1}{2},j}^n \mright) \right|\\
        = & \Bigl| f_1^k\mleft( t,\xb,\rho_{i,j}^{k,n},\contconv{\rhob_\Delta} \mright) - f_1^k\mleft( t,\xb,\varrho^k(s,\yb),\contconv{\rhob_\Delta} \mright)\\
        & - f_1^k\mleft( t,\xb,\rho_{i,j}^{k,n},\Rb_{i+\frac{1}{2},j}^n \mright) + f_1^k\mleft( t,\xb, \varrho^k(s,\yb),\Rb_{i+\frac{1}{2},j}^n \mright) \Bigr|
    .\end{align*}
    Hence, we consider $\rho \in \imgrho_k$ for both, replacing $\varrho^k(s,\yb)$ and $\rho_{i,j}^{k,n}$, respectively, and $(t,\xb) \in C_{i,j}^n$
    \begin{align}\label{eq:difffconv}
        &\left| f_1^k\mleft( t,\xb,\rho,\contconv{\rhob_\Delta} \mright) - f_1^k\mleft( t,\xb^\reva{i+\frac{1}{2},j},\rho,\Rb_{i+\frac{1}{2},j}^n\mright) \right| \\
        \leq & \left| f_1^k\mleft( t,\xb,\rho,\contconv{\rhob_\Delta} \mright) - f_1^k\mleft( t,\xb,\rho,\contconarg{\rhob_\Delta}{t^n}{\xb} \mright) \right| \\
        & + \left|f_1^k\mleft( t,\xb,\rho,\contconarg{\rhob_\Delta}{t^n}{\xb} \mright) - f_1^k\mleft( t^n,\xb,\rho,\contconarg{\rhob_\Delta}{t^n}{\xb} \mright) \right| \\
        & + \left| f_1^k\mleft( t^n,\xb,\rho,\contconarg{\rhob_\Delta}{t^n}{\xb} \mright) - f_1^k\mleft( t^n,\xb,\rho,\contconarg{\rhob_\Delta}{t^n}{\xb^{i+\frac{1}{2},j}} \mright) \right| \\
        & + \left| f_1^k\mleft( t^n,\xb,\rho,\contconarg{\rhob_\Delta}{t^n}{\xb^{i+\frac{1}{2},j}} \mright) -  f_1^k\mleft( t^n,\xb^{i+\frac{1}{2},j},\rho,\contconarg{\rhob_\Delta}{t^n}{\xb^{i+\frac{1}{2},j}} \mright)\right| \\
        & + \left|  f_1^k\mleft( t^n,\xb^{i+\frac{1}{2},j},\rho,\contconarg{\rhob_\Delta}{t^n}{\xb^{i+\frac{1}{2},j}} \mright) -  f_1^k\mleft( t^n,\xb^{i+\frac{1}{2},j},\rho, \Rb_{i+\frac{1}{2},j}^n \mright)\right| 
    .\end{align}
    In each case, the estimation of these five terms begins with using the mean value theorem and bounding the derivatives of $f_1^k$ by $\mathcal{M}|\rho|$ using assumption $(\fb)$. Next, the differences of the arguments have to be considered and for the first, third and last line we can proceed as in \cite[p.\ 3374]{aggarwal2024well} giving us (in two space dimensions) bounds consisting of a constant times $(\dx + \dy)$. We simplify this using $\dy = \dx/\delta$. For the second and fourth line we estimate $|t-t^n|$ and $\| \xb - \xb^{i+\frac{1}{2},j} \|_2$. Overall, we have
    \begin{align}
        \left| f_1^k\mleft( t,\xb,\rho,\contconv{\rhob_\Delta} \mright) - f_1^k\mleft( t,\xb^\reva{i+\frac{1}{2},j},\rho,\Rb_{i+\frac{1}{2},j}^n\mright) \right| 
        \leq c_1 |\rho| \dx
    \end{align}
    such that we obtain
    \begin{align*}
        \left| \bigeps_2^1 \right|
        \leq & c_1 \dx \int_{Q_T} \sum_{i,j \in \Z} \sum_{n=0}^{N_T-1} \mleft( |\rho_{i,j}^{k,n}| + \left| \varrho^k(s,\yb) \right|  \mright) \int_{C_{i,j}^n} \left| \ddx \phi (s,\yb,t,\xb) \right| dt \; d\xb \; ds \; d\yb,
    \end{align*}
    which can be estimated analogously to \cite[p.\ 3374]{aggarwal2024well} by $\frac{\dx}{\varepsilon}$ times \reva{$c_1$ and} a constant depending on the $L^1$-norms of $\rho^k_{\Delta}$ and $\varrho^k$.
    Next, the term $\biglambda_3^1$ is estimated, which can be rewritten into
    \begin{align*}
        \biglambda_3^1  = & \int_{Q_T} \sum_{i,j \in \Z} \sum_{n=0}^{N_T-1} \textnormal{sgn}\mleft( \rho_{i,j}^{k,n} - \varrho^k(s,\yb) \mright) \Bigl( f_1^k\mleft( t^n, \xb^{i+\frac{1}{2},j}, \varrho^k(s,\yb), \Rb_{i+\frac{1}{2},j}^n \mright)\\
        & \qquad \qquad - f_1^k\mleft( t^n, \xb^{i-\frac{1}{2},j}, \varrho^k(s,\yb), \Rb_{i-\frac{1}{2},j}^n \mright) \Bigr) \times    \int_{C^n} \int_{x_2^{j-\frac{1}{2}}}^{x_2^{j+\frac{1}{2}}}  \phi(s, \yb, t, (x_1^{i+\frac{1}{2}},x_2)) dx_2 \, dt \, ds \, d\yb  
    \end{align*}
    \begin{align*}
        &+\int_{Q_T} \sum_{i,j \in \Z} \sum_{n=0}^{N_T-1}\textnormal{sgn}\mleft( \rho_{i,j}^{k,n} - \varrho^k(s,\yb) \mright) \int_{C_{i,j}^n} \frac{d}{dx_1} f_1^k\mleft( t, \xb, \varrho^k(s,\yb), \contconv{\reva{\rhob_\Delta}} \mright)  \\
        &  \hspace{6cm}\times  \mleft( \phi(s, \yb, t, \xb) - \phi(s, \yb, t, (x_1^{i+\frac{1}{2}},x_2)) \mright) \reva{dt} \, d\xb \, ds \, d\yb \\
        &+\int_{Q_T} \sum_{i,j \in \Z} \sum_{n=0}^{N_T-1}\textnormal{sgn}\mleft( \rho_{i,j}^{k,n} - \varrho^k(s,\yb) \mright) \int_{C_{i,j}^n} \phi(s, \yb, t, (x_1^{i+\frac{1}{2}},x_2))  \\
        & \hspace{1cm}\times   \Biggl(\frac{d}{dx_1} f_1^k\mleft( t, \xb, \varrho^k(s,\yb), \contconv{\reva{\rhob_\Delta}} \mright)\\
        &\hspace{2cm} - \frac{f_1^k\mleft( t^n, \xb^{i+\frac{1}{2},j}, \varrho^k(s,\yb), \Rb_{i+\frac{1}{2},j}^n \mright) - f_1^k\mleft( t^n, \xb^{i-\frac{1}{2},j}, \varrho^k(s,\yb), \Rb_{i-\frac{1}{2},j}^n \mright)}{\dx} \Biggr) dt \, d\xb \, ds \, d\yb \\
            \eqcolon &\; \biglambda_3'^1 + \bigeps_{31} + \bigeps_{32}
    .\end{align*}
    Now, for $\left| \bigeps_{31} \right|$  we can use analogous inequalities as in \cite[p.\ 3375]{aggarwal2024well} with first applying 
    $$\left| \frac{d}{dx_1} f_1^k\mleft( t, \xb, \varrho^k(s,\yb), \contconv{\reva{\rhob_\Delta}} \mright) \right| \leq \mathcal{M} |\varrho^k(s,\yb)| (1+ \| \partial_{x_1} \revb{\etab} \|_{L^\infty} \| \rhob_0 \|_{L^1} \reva{)}$$
    such that $\left| \bigeps_{31} \right| = \mathcal{O}\mleft(  \frac{\dx}{\varepsilon} \mright)$. The estimates for
    $\left| \bigeps_{32} \right|$ are as well similar to \cite[p.\ 3376]{aggarwal2024well}, but with the differences of $f_1^k$ instead of the convolution terms, and we obtain the additional term 
    \begin{align*}
        &\int_{Q_T} \sum_{i,j \in \Z} \sum_{n=0}^{N_T-1} \int_{C^n} \int_{x_2^{j-\frac{1}{2}}}^{x_2^{j+\frac{1}{2}}} \phi(s, \yb, t, (x_1^{i+\frac{1}{2}},x_2))
        \Bigl(  f_1^k\mleft( t^n, (x_1^{i+\frac{1}{2}},x_2)^T, \varrho^k(s,\yb), \contconarg{\rhob_\Delta}{t^n}{(x_1^{i+\frac{1}{2}},x_2)} \mright)\\
        & - f_1^k\mleft( t^n, \xb^{i+\frac{1}{2},j}, \varrho^k(s,\yb), \contconarg{\rhob_\Delta}{t^n}{\xb^{i+\frac{1}{2},j}} \mright)
        -f_1^k\mleft( t^n, (x_1^{i-\frac{1}{2}},x_2)^T, \varrho^k(s,\yb), \contconarg{\rhob_\Delta}{t^n}{(x_1^{i-\frac{1}{2}},x_2)} \mright)\\
        & + f_1^k\mleft( t^n, \xb^{i-\frac{1}{2},j}, \varrho^k(s,\yb), \contconarg{\rhob_\Delta}{t^n}{\xb^{i-\frac{1}{2},j}} \mright)  \Bigr) dx_2\; dt\; ds\; d\yb
    \end{align*}
    that can be treated as the second summand of $\left|  \bigeps_{32} \right| $ in \cite{aggarwal2024well}: We perform summation by parts in $i$ to derive a difference of the test functions and to reduce the four flux terms to two
    and use the bounds for the differences in $f_1^k$, where in our setting the assumption $(\fb)$ and \eqref{eq:difffconv} \reva{are} used. Based on this, the integrals over the difference of the test functions in $x_1$ can be estimated by the BV-norm of $\omega_{\epsilon}$, i.e.\
    \begin{align*}
        \sum_{i,j \in \Z} \sum_{n=0}^{N_T-1} \int_{C^n} \int_{x_2^{j-\frac{1}{2}}}^{x_2^{j+\frac{1}{2}}} \left|\phi(s, \yb, t, (x_1^{i+\frac{1}{2}},x_2)) - \phi(s, \yb, t, (x_1^{i-\frac{1}{2}},x_2)) \right| \leq \left| \omega_\eps \right|_{BV(\R)} = \frac{1}{\varepsilon}
    .\end{align*}
    Thus, we have proven 
    \begin{align*}
        -\Lambda_{\epsilon,\epsilon_0}^k(\rho_\Delta^k,\varrho^k) = \biglambda_1 + \biglambda_2' +  \biglambda_3' + \mathcal{O}\mleft( \frac{\dx}{\varepsilon} \mright) 
    .\end{align*}
    The next steps are very similar to those in \cite[pp.\ 3378-3381]{aggarwal2024well}, i.e.\ we \revb{first} exploit the discrete entropy inequality for $\biglambda_1$ to split it up\footnote{The terms $A_2, A_3$ in \cite{aggarwal2024well} are here again expressed as two terms each, according to the two directions.} into $A_2^1+A_2^2+A_3^1+A_3^2$,
    \revb{where
    \begin{align*}
      A_2^1 = -\lambda_1 \int_{Q_T} \sum_{i,j \in \Z} \sum_{n=0}^{N_T-1}  &\mleft( \entropyfluxknkappa{1}{i+\frac{1}{2}}{j}{\rho_{i,j}^{k,n}}{\rho_{i+1,j}^{k,n}}{\varrho^k(s,\yb)} -\entropyfluxknkappa{1}{i-\frac{1}{2}}{j}{\rho_{i-1,j}^{k,n}}{\rho_{i,j}^{k,n}}{\varrho^k(s,\yb)}\mright) \\
         & \hspace{5cm}  \times \mleft( \int_{C_{i,j}}\phi(s,\yb,t^{n+1},\xb) d\xb \mright) ds \, d\yb,
    \end{align*}
    \begin{align*}
      A_3^1 = -\lambda_1 \int_{Q_T} \sum_{i,j \in \Z} \sum_{n=0}^{N_T-1} &\textnormal{sgn}\mleft( \rho_{i,j}^{k,n+1} - \varrho^k(s,\yb) \mright)
      \Bigl( f_1^k\mleft( t^n, \xb^{i+\frac{1}{2},j}, \varrho^k(s,\yb), \Rb_{i+\frac{1}{2},j}^n \mright)\\
       & - f_1^k\mleft( t^n, \xb^{i-\frac{1}{2},j}, \varrho^k(s,\yb), \Rb_{i-\frac{1}{2},j}^n \mright)\Bigr)
       \times \mleft( \int_{C_{i,j}}\phi(s,\yb,t^{n+1},\xb) d\xb \mright) ds \, d\yb
    .\end{align*}
    Then, we}
    apply the fundamental theorem of calculus and do summation by parts for $\biglambda_2'^1$. We also follow the first steps in \cite[Claim 1]{aggarwal2024well} to bound $A_2^1 + \biglambda_2'^1$, i.e.\ adding and subtracting a term and performing summation by parts to obtain
    \begin{align*}
        A_2^1 + \biglambda_2'^1 =&\; \lambda_1 \int_{Q_T} \sum_{i,j \in \Z} \sum_{n=0}^{N_T-1} \mleft( \entropyfluxknkappa{1}{i+\frac{1}{2}}{j}{\rho_{i,j}^{k,n}}{\rho_{i+1,j}^{k,n}}{\varrho^k(s,\yb)} - p_{i,j}^{k,n}(\varrho^k(s,\yb), R_{i+\frac{1}{2},j}^{k,n})\mright) \\
         & \hspace{2cm}  \times  \mleft( \int_{C_{i+1,j}} \phi(s,\yb,t^{n+1},\xb) d\xb - \int_{C_{i,j}} \phi(s,\yb,t^{n+1},\xb) d\xb \mright) ds \, d\yb\\
        &+ \int_{Q_T} \sum_{i,j \in \Z} \sum_{n=0}^{N_T-1} \mleft( p_{i,j\reva{,1}}^{k,n}(\varrho^k(s,\yb), R_{i+\frac{1}{2},j}^{k,n}) - p_{i-1,j\reva{,1}}^{k,n}(\varrho^k(s,\yb), R_{i-\frac{1}{2},j}^{k,n}) \mright)\\
        & \hspace{2cm} \times  \mleft( \int_{C^n} \int_{x_2^{j-\frac{1}{2}}}^{x_2+\frac{1}{2}} \phi\mleft( s,\yb,t,\mleft( x_1^{\reva{i-\frac{1}{2}}},x_2 \mright)^T  \mright) dt - \lambda_1 \int_{C_{i,j}} \phi\mleft( s,\yb,t^{n+1},\xb \mright) d\xb \mright) 
    ,\end{align*}
    where 
    \begin{align*}
        \mleft| \entropyfluxknkappa{1}{i+\frac{1}{2}}{j}{\rho_{i,j}^{k,n}}{\rho_{i+1,j}^{k,n}}{\varrho^k(s,\yb)} - p_{i,j,1}^{k,n}(\varrho^k(s,\yb), R_{i+\frac{1}{2},j}^{k,n})\mright| \leq c_2 \left| \rho_{i+1,j}^{k,n} - \rho_{i,j}^{k,n} \right| 
    .\end{align*}
    Moreover, we need an estimate on
    \begin{subequations}
    \begin{align}
        & \mleft| p_{i,j,1}^{k,n}(\varrho^k(s,\yb), R_{i+\frac{1}{2},j}^{k,n}) - p_{i-1,j,1}^{k,n}(\varrho^k(s,\yb), R_{i-\frac{1}{2},j}^{k,n}) \mright|\hspace{7cm}
    \end{align}
    \vspace*{-1.5em}
    \begin{equation} \label{eq:pdiffa}
        \begin{aligned}
        \leq  & \Bigl| \textnormal{sgn}\mleft( \rho_{i,j}^{k,n} - \varrho^k(s,\yb)\mright) \Bigl[ f_1^k\mleft( t^n, \xb^{i+\frac{1}{2},j}, \rho_{i,j}^{k,n}, \Rb_{i+\frac{1}{2},j}^n \mright) - f_1^k\mleft( t^n, \xb^{i+\frac{1}{2},j}, \varrho^k(s,\yb), \Rb_{i+\frac{1}{2},j}^n \mright)\\
                & \hspace{3cm}\reva{-} f_1^k\mleft( t^n, \xb^{i-\frac{1}{2},j}, \rho_{i-1,j}^{k,n}, \Rb_{i-\frac{1}{2},j}^n \mright) \reva{+} f_1^k\mleft( t^n, \xb^{i-\frac{1}{2},j}, \varrho^k(s,\yb), \Rb_{i-\frac{1}{2},j}^n \mright) \Bigr] \Bigr|
    \end{aligned}
    \end{equation}
    \vspace*{-0.5em}
    \begin{equation} \label{eq:pdiffb}
        \begin{aligned}
        & + \Bigl| \left[ \textnormal{sgn}\mleft( \rho_{i,j}^{k,n} - \varrho^k(s,\yb)\mright) - \textnormal{sgn}\mleft( \rho_{i-1,j}^{k,n} - \varrho^k(s,\yb)\mright)\right]\\
        & \hspace{2.5cm} \mleft( f_1^k\mleft( t^n, \xb^{i-\frac{1}{2},j}, \rho_{i-1,j}^{k,n}, \Rb_{i-\frac{1}{2},j}^n \mright) - f_1^k\mleft( t^n, \xb^{i-\frac{1}{2},j}, \varrho^k(s,\yb), \Rb_{i-\frac{1}{2},j}^n \mright) \mright)
                \Bigr| 
    .\end{aligned}
    \end{equation}
    \end{subequations}
    The second term \eqref{eq:pdiffb} can be bounded by $2 L_\rho \left| \rho_{i,j}^{k,n} - \rho_{i-1,j}^{k,n} \right| $ because either the two $\textnormal{sgn}$ terms are equal, causing the term to vanish, or we have that $\varrho^k(s,\yb)$ is between $\rho_{i-1,j}^{k,n}$ and $\rho_{i,j}^{k,n}$, i.e.\ after employing the Lipschitz continuity of $f_1^k$ we use $\left| \rho_{i-1,j}^{k,n} -  \varrho^k(s,\yb) \right| \leq  \left| \rho_{i-1,j}^{k,n} - \rho_{i,j}^{k,n} \right| $.
    For the first term \eqref{eq:pdiffa} successively applying the mean value theorem yields 
    \begin{align*}
        \left| \eqref{eq:pdiffa} \right| \leq &\left| \partial_{x_1x_1}^2 f_1^k\mleft( t^n, \breve{\xb}^{i,j}, \rho_{i,j}^{k,n},\Rb_{i+\frac{1}{2},j}^{k,n} \mright)  \right| \dx^2 + \Lipf{\rho}{x_1} \left| \rho_{i,j}^{k,n} - \varrho^k(s,\yb) \right| \dx + \reva{3}\Lipf{\rho}{} |\rho_{i-1,j}^{k,n} - \rho_{i,j}^{k,n} |\\
        &+ \Lipf{\rho}{\Rb} \left| \rho_{\reva{i-1},j}^{k,n} - \varrho^k(s,\yb) \right| \left\| \Rb_{i+\frac{1}{2},j}^{k,n} - \Rb_{i-\frac{1}{2},j}^{k,n} \right\|_{2}\\
        &+ \left| \nabla_{\Rb} \mydot \nabla_{\Rb} f_1\mleft( t^n, \xb^{i-\frac{1}{2},j}, \varrho^k\mleft( s,\yb\mright), \breve{\Rb}_{i,j}  \mright) \right| \left\| \Rb_{i+\frac{1}{2},j}^{k,n} - \Rb_{i-\frac{1}{2},j}^{k,n} \right\|_{2}^2 
    .\end{align*}
    Thus, using again assumption $(\fb)$ and \reva{the estimates in \cite[Lem.~A.2]{ACG15} for differences in the nonlocal terms} we are left with 
    \begin{align*}
         & \mleft| p_{i,j,1}^{k,n}(\varrho^k(s,\yb), R_{i+\frac{1}{2},j}^{k,n}) - p_{i-1,j,1}^{k,n}(\varrho^k(s,\yb), R_{i-\frac{1}{2},j}^{k,n}) \mright| \\
        \leq & c_3 \mleft[ \left| \rho_{i-1,j}^{k,n} - \rho_{i,j}^{k,n} \right| + \left| \rho_{i-1,j}^{k,n} - \varrho^{k}(s,\yb) \right| \dx \reva{+\left| \rho_{i,j}^{k,n} - \varrho^{k}(s,\yb) \right| \dx} + \mleft( \left| \rho_{i,j}^{k,n} \right| + \left| \varrho^k(s,\yb) \right| \mright) \dx^2  \mright] 
    .\end{align*}
    To proceed with the estimate of $A_2^1+\biglambda_2'^1$, we need estimates on the differences of the integrals over the test function $\phi$.
    For the first term we consider 
    \begin{align}\label{eq:42}
        & \left| \int_{Q_T} \int_{C_{i+1,j}} \phi\mleft( s,\yb,t^{n+1},\xb  \mright) dt - \int_{C_{i,j}} \phi\mleft( s,\yb,t^{n+1},\xb \mright) ds \, d\yb \right|\\
        & = \left| \int_\R \int_{C_{i,j}} \omega_{\varepsilon}(y_1 - x_1 - \dx)  - \omega_{\varepsilon}(y_1 - x_1) d\xb \, dy_1 \right| \\
        & \leq \int_\R \int_{C_{i,j}} \int_{-\dx}^{0} \left| \omega_{\varepsilon}'(y_1-x_1 + \xi) \right| d\xi \, d\xb \, dy_1 = \frac{\dx^2 \dy}{\varepsilon}.
    \end{align}
    The second one is considered using Taylor expansions
    \begin{align}\label{eq:41}
        & \int_{C^n} \int_{x_2^{j-\frac{1}{2}}}^{x_2^{j+\frac{1}{2}}} \phi\mleft( s,\yb,t,\mleft( x_1^{\reva{i-\frac{1}{2}}},x_2 \mright)^T  \mright) dt - \lambda_1 \int_{C_{i,j}} \phi\mleft( s,\yb,t^{n+1},\xb \mright) d\xb\\
        \leq & \int_{x_2^{j-\frac{1}{2}}}^{x_2^{j+\frac{1}{2}}} \left| \ddt \phi\mleft( s,\yb,\widetilde{t}^{n+\frac{1}{2}},\mleft( x_1^{\reva{i-\frac{1}{2}}},x_2 \mright)^T  \mright) \right| \dt^2 + \left| \ddx \phi\mleft( s,\yb,t^{n+1},\mleft( \widetilde{x}_1^{i},x_2 \mright)^T  \mright) \right| \dt \dx dx_2
    \end{align}
    with $\widetilde{t}^{n+\frac{1}{2}} \in [t^n,t^{n+1}]$ and $\widetilde{x}_1^{i} \in  [x_1^{i-\frac{1}{2}},x_1^{i+\frac{1}{2}}]$. Thus, we have
    \begin{align}\label{eq:A2lamb2s}
        A_2^1 &+ \biglambda_2'^1 \leq \, c_2 \lambda_1 \sum_{i,j \in \Z} \sum_{n=0}^{N_T-1} \left| \rho_{i+1,j}^{k,n}-\rho_{i,j}^{k,n} \right| \frac{\dx^2 \dy}{\varepsilon} \\
        &+ c_3 \int_{Q_T} \sum_{i,j \in \Z} \sum_{n=0}^{N_T-1} \Bigl[ \left| \rho_{i-1,j}^{k,n} - \rho_{i,j}^{k,n} \right| + \left| \rho_{i-1,j}^{k,n} - \varrho^{k}(s,\yb) \right| \dx\\
         &\hspace*{4.3cm} \reva{+ \left| \rho_{i,j}^{k,n} - \varrho^{k}(s,\yb) \right| \dx} + \mleft( \left| \rho_{i,j}^{k,n} \right| + \left| \varrho^k(s,\yb) \right| \mright) \dx^2  \Bigr] \\
        & \qquad \left(\int_{x_2^{j-\frac{1}{2}}}^{x_2^{j+\frac{1}{2}}} \left| \ddt \phi\mleft( s,\yb,\widetilde{t}^{n+\frac{1}{2}},\mleft( x_1^{i+\frac{1}{2}},x_2 \mright)^T  \mright) \right| \dt^2 + \left| \ddx \phi\mleft( s,\yb,t^{n+1},\mleft( \widetilde{x}_1^{i},x_2 \mright)^T  \mright) \right| \dt \dx dx_2 \right) ds \, d\yb
    .\end{align}
    The terms including the differences between two cell-average values of $\rho^k_\Delta$ are similar to those estimated in \cite[pp.\ 3380-3381]{aggarwal2024well}. The remaining are considered now, where we 
    insert \mbox{$\left| \rho_{\reva{i},j}^{k,n} - \varrho^{k}(s,\yb) \right| \linebreak[0] {\leq \left| \rho_{\reva{i},j}^{k,n} \right| + \left| \varrho^{k}(s,\yb) \right|}$}. This results in two terms \reva{(as analogously for $\bigl| \rho_{i-1,j}^{k,n} - \varrho^{k}(s,\yb)\bigr|$)}: One can be estimated by
    \begin{equation}
        \begin{aligned}
            &\sum_{i,j \in \Z} \sum_{n=0}^{N_T-1} \left| \rho_{\reva{i},j}^{k,n} \right| \dx \int_{x_2^{j-\frac{1}{2}}}^{x_2^{j+\frac{1}{2}}} \int_{Q_T} \left| \ddt \phi\mleft( s,\yb,\widetilde{t}^{n+\frac{1}{2}},\mleft( x_1^{i+\frac{1}{2}},x_2 \mright)^T  \mright) \right| \dt^2\\
            & \hspace{5cm} + \left| \ddx \phi\mleft( s,\yb,t^{n+1},\mleft( \widetilde{x}_1^{i},x_2 \mright)^T  \mright) \right| \dt \dx   ds \, d\yb dx_2 \\
            &\leq \sum_{i,j \in \Z} \sum_{n=0}^{N_T-1} \left| \rho_{\reva{i},j}^{k,n} \right| \dx \int_{x_2^{j-\frac{1}{2}}}^{x_2^{j+\frac{1}{2}}} \mleft( \frac{\dt^2}{\varepsilon_0} + \frac{\dt \dx}{\varepsilon} \mright) dx_2,
        \end{aligned}
    .\end{equation}
    and the other by
    \begin{align*}
        &\int_{Q_T} \left| \varrho^{k}(s,\yb) \right|\dx \sum_{i,j \in \Z} \sum_{n=0}^{N_T-1}\int_{x_2^{j-\frac{1}{2}}}^{x_2^{j+\frac{1}{2}}} \left| \ddt \phi\mleft( s,\yb,\widetilde{t}^{n+\frac{1}{2}},\mleft( x_1^{i+\frac{1}{2}},x_2 \mright)^T  \mright) \right| \dt^2  \\
        & \hspace{6cm} +\left| \ddx \phi\mleft( s,\yb,t^{n+1},\mleft( \widetilde{x}_1^{i},x_2 \mright)^T  \mright) \right| \dt \dx  ds \, d\yb dx_2\\
        \leq & \int_{Q_T} \left| \varrho^{k}(s,\yb) \right| \mleft(  \| \omega_{\varepsilon} \|_{L^1(\R)} \dt \left(\sum_{i \in \Z} \left| \omega_\varepsilon\mleft(y_1 - x_1^{i+\frac{1}{2}}\mright) \right| \dx\right) \left(\sum_{n=0}^{N_T-1} \left| \omega_{\varepsilon_0}'\mleft( s- \widetilde{t}^{n+\frac{1}{2}} \mright) \right| \dt \right)\mright) ds \, d\yb  \\
        & + \int_{Q_T} \left| \varrho^{k}(s,\yb) \right| \mleft(  \| \omega_{\varepsilon} \|_{L^1(\R)} \dx \left(\sum_{i \in \Z} \left| \omega_\varepsilon' \mleft(y_1 - \widetilde{x}_1^{i+\frac{1}{2}}\mright) \right| \dx\right) \left(\sum_{n=0}^{N_T-1} \left| \omega_{\varepsilon_0}\mleft( s- t^{n} \mright) \right| \dt \right)\mright) ds \, d\yb 
    .\end{align*}
    Due to the Lipschitz continuity of $\omega_{\varepsilon_0}'$ we can write 
    \begin{align}\label{eq:discrnormomega}
        \sum_{n=0}^{N_T-1} \left| \omega_{\varepsilon_0}'\mleft( s- \widetilde{t}^{n+\frac{1}{2}}\mright) \right| \dt = \int_{0}^{T} \left| \omega_{\varepsilon_0}'\mleft( s- t\mright) \right| dt + \mathcal{O}(\dt) = \left| \omega_{\varepsilon_0} \right|_{BV(\R)} + \mathcal{O}(\dt)
    \end{align}
    and analogue for $\omega_\varepsilon'$, as well as for $\omega_{\varepsilon_0}, {\omega}_\varepsilon$ using the $L^1(\R)$-norm instead.
    Thus, the term above can be estimated by 
    \begin{align*}
        \| \rho^k_\Delta \|_{L^1(Q_T)} \mleft( \frac{\dt}{\varepsilon_0} + \frac{\dx}{\varepsilon} \mright)
        + \| \varrho^k \|_{L^1(Q_T)} \mleft( \frac{\dt}{\varepsilon_0} + \frac{\dx}{\varepsilon} + \mathcal{O}\mleft( \dt^2 + \dx^2  \mright)  \mright)
    .\end{align*}
    The summand in \eqref{eq:A2lamb2s} including $\mleft( \left| \rho_{i,j}^{k,n} \right| + \left| \varrho^k(s,\yb) \right| \mright) \dx^2 $ can be estimated analogously with an additional $\dx$ such that, overall, we obtain
    \begin{align*}
        \left| \eqref{eq:A2lamb2s} \right| = \mathcal{O}\mleft( \frac{\dt}{\varepsilon_0} + \frac{\dx}{\varepsilon} + \dt^2 + \dx^2  \mright) 
    .\end{align*}
    The next step is to prove $A_3^1 + \biglambda_3'^1 = \mathcal{O}\mleft(\frac{\dx}{\varepsilon} + \frac{\dt}{\varepsilon_0} + \dx + \dt  \mright) $. We adhere to the calculation in \cite[pp.\ 3382-3383]{aggarwal2024well}, meaning we add a zero to obtain the sum of terms denoted by $\widetilde{\bigeps}_1$ and $\widetilde{\bigeps}_2$. We start with
    \begin{align*}
         \widetilde{\bigeps}_1 = \int_{Q_T}& \sum_{i,j \in \Z} \sum_{n=0}^{N_T-1}\mleft( f_1^k\mleft( t^n, \xb^{i+\frac{1}{2},j}, \varrho^k(s,\yb), \Rb_{i+\frac{1}{2},j}^n \mright) - f_1^k\mleft( t^n, \xb^{i-\frac{1}{2},j}, \varrho^k(s,\yb), \Rb_{i-\frac{1}{2},j}^n \mright) \mright)\\
         &\reva{\times \; \textnormal{sgn}\mleft( \rho_{i,j}^{k,n} - \varrho^k(s,\yb) \mright) } \mleft( \int_{C^n} \int_{x_2^{j-\frac{1}{2}}}^{x_2+\frac{1}{2}} \phi\mleft( s,\yb,t,\mleft( x_1^{i+\frac{1}{2}},x_2 \mright)^T  \mright) dt - \lambda_1 \int_{C_{i,j}} \phi\mleft( s,\yb,t^{n+1},\xb \mright) d\xb ds \, d\yb \mright)
    .\end{align*}
    In the absolute value, we first utilize assumption $(\fb)$ and \eqref{eq:41} to obtain
    \begin{align*}
        \left| \widetilde{\bigeps}_1 \right| \leq  c_4 \int_{Q^T} \left| \varrho^k(s,\yb) \right| \dx \sum_{i,j \in \Z} \sum_{n=0}^{N_T-1}\int_{x_2^{j-\frac{1}{2}}}^{x_2^{j+\frac{1}{2}}} &\left| \ddt \phi\mleft( s,\yb,\widetilde{t}^{n+\frac{1}{2}},\mleft( x_1^{i+\frac{1}{2}},x_2 \mright)^T  \mright) \right| \dt^2\\
        & + \left| \ddx \phi\mleft( s,\yb,t^{n+1},\mleft( \widetilde{x}_1^{i},x_2 \mright)^T  \mright) \right| \dt \dx  dx_2
    ,\end{align*}
    which equals one of the terms in $A_2^1+\biglambda_2'^1$ up to a constant. Hence, $\left| \widetilde{\bigeps}_1 \right| = \mathcal{O}\mleft( \frac{\dt}{\varepsilon_0} + \frac{\dx}{\varepsilon} + \dt^2 + \dx^2  \mright)$.
    The second term $\widetilde{\bigeps}_2$ given by
    \begin{align*}
        \widetilde{\bigeps}_2 \coloneq - \lambda_1 \int_{Q_T} \sum_{i,j \in \Z} \sum_{n=0}^{N_T-1} \mleft( f_1^k\mleft( t^{n}, \xb^{i+\frac{1}{2},j}, \varrho^k(s,\yb), \Rb_{i+\frac{1}{2},j}^{n} \mright) - f_1^k\mleft( t^{n}, \xb^{i-\frac{1}{2},j}, \varrho^k(s,\yb), \Rb_{i-\frac{1}{2},j}^{n} \mright)  \mright)& \\
        \left( \textnormal{sgn}(\rho_{i,j}^{k,n+1} - \varrho^k(s,\yb)) - \textnormal{sgn}(\rho_{i,j}^{k,n} - \varrho^k(s,\yb)) \right)
        \mleft( \int_{C_{i,j}} \phi\mleft(s,\yb,t^{n+1},\xb \mright) d\xb \mright) ds \, d\yb&
    \end{align*}
    is first treated as in \cite[pp.\ 3383-3384]{aggarwal2024well} (summation by parts, adding a zero) and further splitted into $\widetilde{\bigeps}_{21},\widetilde{\bigeps}_{22},\widetilde{\bigeps}_{23},\widetilde{\bigeps}_{24},\widetilde{\bigeps}_{25}$.
    Thus, we need to \reva{estimate} these terms. In our setting we have
    \begin{align*}
        \widetilde{\bigeps}_{21} = - \lambda_1 \int_{Q_T} \sum_{i,j \in \Z} \sum_{n=1}^{N_T} \mleft( f_1^k\mleft( t^{n-1}, \xb^{i+\frac{1}{2},j}, \varrho^k(s,\yb), \Rb_{i+\frac{1}{2},j}^{n-1} \mright) - f_1^k\mleft( t^{n-1}, \xb^{i-\frac{1}{2},j}, \varrho^k(s,\yb), \Rb_{i-\frac{1}{2},j}^{n-1} \mright)  \mright) \\
        \textnormal{sgn}(\rho_{i,j}^{k,\reva{n}} - \varrho^k(s,\yb))  \mleft( \int_{C_{i,j}} \phi\mleft(s,\yb,t^{n},\xb \mright)-\phi\mleft(s,\yb,t^{n+1},\xb \mright) d\xb \mright) ds \, d\yb
    .\end{align*}
    Again, first estimating the difference of the fluxes yields
    \begin{align*}
        \left| \widetilde{\bigeps}_{21} \right|& \leq c_4 \lambda_1 \dx \int_{Q_T} \left| \varrho^k(s,\yb) \right| \sum_{i,j \in \Z} \sum_{n=1}^{N_T} \mleft| \int_{C_{i,j}} \phi\mleft(s,\yb,t^{n},\xb \mright)-\phi\mleft(s,\yb,t^{n+1},\xb \mright) d\xb \mright| ds \, d\yb \\
        &= c_4 \dt \int_{Q_T} \left| \varrho^k(s,\yb) \right| \sum_{n=1}^{N_T} \left| \omega_{\varepsilon_0}\mleft( s-t^n \mright) - \omega_{\varepsilon_0}\mleft( s-t^{n+1} \mright)  \right| ds \, d\yb \leq c_4  \|\varrho^k \|_{L^1(Q_T)} \frac{\dt}{\varepsilon_0}
    ,\end{align*}
    since the discrete BV-norm on the time mesh of the test function $\omega_{\varepsilon_0}$ does not exceed the continuous BV-norm.
    Moreover, 
    \begin{align*}
        \left| \widetilde{\bigeps}_{22} + \widetilde{\bigeps}_{23} \right| 
        \leq \lambda_1 \int_{Q_T} \sum_{i,j \in \Z} \sum_{n=1}^{N_T} \Bigl| &f_1^k\mleft( t^{n}, \xb^{i+\frac{1}{2},j}, \varrho^k(s,\yb), \Rb_{i+\frac{1}{2},j}^{n} \mright) - f_1^k\mleft( t^{n}, \xb^{i-\frac{1}{2},j}, \varrho^k(s,\yb), \Rb_{i-\frac{1}{2},j}^{n} \mright)\\
        & - f_1^k\mleft( t^{n-1}, \xb^{i+\frac{1}{2},j}, \varrho^k(s,\yb), \Rb_{i+\frac{1}{2},j}^{n-1} \mright) + f_1^k\mleft( t^{n-1}, \xb^{i-\frac{1}{2},j}, \varrho^k(s,\yb), \Rb_{i-\frac{1}{2},j}^{n-1} \mright)  \Bigr|\\
        & \times \int_{C_{i,j}} \phi\mleft( s,\yb,t^{n+1}, \xb \mright) d\xb \; ds \, d\yb 
    .\end{align*}
    Thus, we again need an estimate on the flux terms. Successively applying the mean value theorem we obtain 
    \begin{align*}
        \Bigl| &f_1^k\mleft( t^{n}, \xb^{i+\frac{1}{2},j}, \varrho^k(s,\yb), \Rb_{i+\frac{1}{2},j}^{n} \mright) - f_1^k\mleft( t^{n}, \xb^{i-\frac{1}{2},j}, \varrho^k(s,\yb), \Rb_{i-\frac{1}{2},j}^{n} \mright)\\
        & - f_1^k\mleft( t^{n-1}, \xb^{i+\frac{1}{2},j}, \varrho^k(s,\yb), \Rb_{i+\frac{1}{2},j}^{n-1} \mright) + f_1^k\mleft( t^{n-1}, \xb^{i-\frac{1}{2},j}, \varrho^k(s,\yb), \Rb_{i-\frac{1}{2},j}^{n-1} \mright)  \Bigr|\\
        \leq &\; c_5 \left| \varrho^k(s,\yb) \right|  \dx (\dt + \dx) 
    \end{align*}
    and hence
    \begin{align*}
        \left| \widetilde{\bigeps}_{22} + \widetilde{\bigeps}_{23} \right| 
        &\leq c_5 (\dt + \dx) \int_{Q_T} \left| \varrho^k(s,\yb) \right| \sum_{i,j \in \Z} \sum_{n=1}^{N_T} \int_{C_{i,j}} \phi\mleft( s,\yb,t^{n+1}, \xb \mright) \dt \, d\xb \; ds \, d\yb \\
        &\leq  c_5 (\dt + \dx) \| \varrho^k \|_{L^1(Q_T)} \mathcal{O}\mleft( 1 + \dt \mright) = \mathcal{O}\mleft( \dt + \dx \mright) 
    ,\end{align*}
    where an estimate similar to \eqref{eq:discrnormomega} was used.
    The estimates for $\widetilde{\bigeps}_{24},\widetilde{\bigeps}_{25}$ follow \revd{similarly} and we have proven $A_3^1 + \biglambda_3'^1 = \mathcal{O}\mleft(\frac{\dx}{\varepsilon} + \frac{\dt}{\varepsilon_0} + \dx + \dt  \mright)$. Consequently, with the analogue estimates for the \mbox{$x_2$-directions}, we have estimated all the terms, achieving the result that needed to be established.
\end{proof}

\end{document}